\documentclass[11pt,a4paper,reqno]{amsart}
\usepackage{amsmath,amssymb,amsbsy,amsthm,dsfont}
\usepackage{graphicx,a4wide}
\usepackage{multirow}
\usepackage[utf8]{inputenc}
\usepackage[T1]{fontenc}
\usepackage{float}
\usepackage{natbib}
\usepackage{url}
\usepackage{ulem}

\usepackage[urlcolor=red, linkcolor=blue, colorlinks=true]{hyperref}
\usepackage{xcolor}
\newtheorem{theorem}{Theorem}[section]
\newtheorem{lemma}[theorem]{Lemma}
\newtheorem{remark}{Remark}

\newcommand{\m}{\boldsymbol{m}}

\definecolor{salmon}{rgb}{1.0, 0.6, 0.6}

\newcounter{hypA}
\newcounter{subhypA}[hypA]

\renewcommand{\thehypA}{A\arabic{hypA}}

\newenvironment{hypA}{
  \refstepcounter{hypA}
  \begin{itemize}
  \item[({\bf \thehypA})]
}{
  \end{itemize}
}

\newenvironment{subhypA}{
  \refstepcounter{subhypA}
  \begin{itemize}
  \item[({\bf \arabic{subhypA})}] % ← affichage simple
}{
  \end{itemize}
}

\definecolor{salmon}{rgb}{1.0, 0.6, 0.6}

\begin{document}

\title[Minimization of an unknown function in a multivariate regression model]{Optimal minimization of an unknown function in a nonparametric multivariate regression model thanks to a dimension reduction approach}
\author{Cédric Adam}
\address{Andra, 1/7 Rue Jean Monnet, 92290 Châtenay-Malabry, France and Université Paris Cité and Sorbonne Université, CNRS, Laboratoire de Probabilités, Statistique et Modélisation, F-75013 Paris, France}
\author{Ilaria Giulini and Céline Lévy-Leduc}
\address{Université Paris Cité and Sorbonne Université, CNRS, Laboratoire de Probabilités, Statistique et Modélisation, F-75013 Paris, France}
\date{\today}

\begin{abstract}
\textcolor{black}{In this paper, we propose a novel approach for estimating the minimum of a smooth function and its location from observations corresponding to a multivariate regression function depending a priori on $d$ variables but actually only on $r<d$ \textit{active} variables and corrupted by some additional noise. Our method consists of two steps: The first one is a variable selection approach which is used for identifying the $r$ active variables on which $f$ depends and the second one consists in estimating the minimum of the function and its location. The estimation of the minimizers is obtained by using a projected gradient descent where the gradient is estimated using a local polynomial approximation of the regression function limited to its active variables obtained in the first step. The estimation of the minimum is obtained by evaluating the estimator of the regression function using a local polynomial approach at the estimator of one of the minimizers previously obtained. We establish non asymptotic upper bounds for the quadratic risk of the  estimators of the minimizers and of the minimum and prove that they reach the optimal rate that could be expected as if the active variables were known beforehand up to a factor smaller than a power of a logarithmic term.}
\end{abstract}

\keywords{multivariate regression model, variable selection, minimization of a function}

\maketitle

% {\color{magenta}
% \begin{enumerate}
% \item quelle notation $\{ x , \text{blabla}\}$ vs $\{ x ; \text{blabla}\}$ ?
% \end{enumerate}
% }

% \textcolor{black}{ATTENTION LES HYPOTHESES ONT CHANGÉ DE NUMERO !!!!!!}

\section{Introduction}

\textcolor{black}{Let us consider the following nonparametric multivariate regression model:
\begin{equation}\label{eq: regression model}
 \forall i \in \{1,...,n\}, \hspace{1ex} y_i = f(x_i) + \xi_i=f(x_i^{(1)},\dots,x_i^{(d)})+\xi_i,
 \end{equation} 
 where %$n\geq 3$,
 $f : \mathbb{R}^d \longrightarrow \mathbb{R}$ is an unknown function, $x_1,\dots,x_n$ are independent identically distributed (i.i.d.) random variables with an unknown density $p$ and $\xi_1,\dots,\xi_n$ are centered random variables independent from $x_1,\dots,x_n$. Estimating the minimum of the function $f$ and its location from noisy observations $y_1,\dots,y_n$ in the previous model is a problem which has received a lot of attention.}

% From \cite{Cauchy:1847} who introduced the gradient descent algorithm, the problem of estimating a minimizer and the minimum of a function has been a very active research field. This is especilly due to its enormous possible applications. We can cite several well known general optimization algorithms as the Accelerated Gradient Descent (AGD) introduced in \cite{Nesterov} which increases the speed of the original gradient descent by adding an inertial term or the Stochastic Gradient Descent which is adapted for a sum of functions and was first used in \cite{SGD}.

% as Broyden–Fletcher–Goldfarb–Shanno (BFGS) an algorithm introduced in {\color{purple} ref} which uses an estimation of the second derivative and is still used in practice.

\textcolor{black}{Two different types of design are considered in the literature: the active design where the statistician chooses the value of $x_1,\dots,x_n$ at which $f$ is computed and the passive one where the $x_1,\dots,x_n$ are sampled independently from some probability distribution.}

%\textcolor{blue}{je ne pense pas que ce soit une bonne idée de distinguer le offline et le online dans active et passive design parce que dans l'active design cela représente à peine un minuscule paragraphe et dans le passive on ne dit rien donc j'ai enlevé cette distinction.}

%The minimization in a case of a regression model has been studied a lot recently especially because it can be applied in many concrete data driven approaches. The litterature consider mainly two different designs : the active design which is when the statistician makes a choice for the $x_i$ and the passive one where the $x_i$ are random. The design can also be online (data coming one after the other) and offline (all of them being available directly).

\textcolor{black}{The problem of estimating the minimizer of a regression function in an active or sequential design has been studied since the 50s. Inspired by the algorithm for root-finding devised by \cite{RobbinsMonro}, \cite{KW} introduced the KW method which aims at estimating the minimizer of a univariate regression function. Their idea consists in approximating the gradient of the function by the central difference method and then {\color{black} in using} a gradient descent generally attributed to \cite{Cauchy:1847}. This algorithm was extended to the multivariate case by \cite{Blum1954} and \cite{Dupac1957} established an upper bound of order $n^{-2/3}$ for the quadratic risk of the estimation of the minimizer of a regression function in the {\color{black} univariate} case when it is assumed to be strongly convex and $\beta$-Hölder with $\beta=3$. Then, by refining the classical central finite-difference estimator of the gradient, \cite{Fabian1967} proved that his novel approach reached a quadratic risk of order $n^{-(\beta-1)/\beta}$ for odd integers $\beta \geq 3$ and functions having bounded $\beta$th partial derivatives. This rate was established to be optimal by \cite{polyak_tsybakov:1990}.}

\textcolor{black}{The estimation of the minimum value in the active design framework has also been widely studied. %In the litterature, this convergence rate appears in several works. For instance, 
  \cite{MokkademPelletier2007} proposed a more general algorithm than the ones proposed by Kiefer, Wolfowitz and Blum %by adding a \textit{companion} which approximates the size of the maximum of the regression function
  and proved its asymptotic normality with a $\sqrt n$ scaling, which was proved to be optimal when the {\color{black} noise} is Gaussian.
  % if $f$ satisfies some local assumptions.
  %It is worth noting that the rate $1/\sqrt{n}$ cannot be improved under Gaussian noise as even an oracle with knowledge of the true minimizer $x^*$ cannot converge faster.
  Over the class of strongly convex and $\beta$-Hölder functions, \cite{Belister2012} devised a new procedure which achieves the $1/\sqrt{n}$ optimal rate for $\beta>2$. Ten years later, \cite{Akhavan2022} proposed a more {\color{black} computationally} feasible estimator achieving the same convergence rate.
  Therefore, for the estimation of both the minimum and the minimizer of a regression function in the active design framework, the
  optimal convergence rates do not depend on {\color{black} the dimension} $d$ and hence do not suffer from the curse of {\color{black} dimensionality}.}
  %obtained bounds under an active online design do not depend on the dimension $d$, so these methods do not suffer from the curse of dimensionality.

% \textcolor{black}{As far as the active offline design is concerned, it corresponds to the issue of design of experiments which is a field of research on its own and is for instance explained in the reference book of \cite{montgomery}.}
% Note that the case of an active offline design corresponds to the theory of Optimal Experimental Design (OED) which results in a fixed design and, thus, is outside the scope of the present work. However, it has been an active area of research since the pioneering work of \cite{OED}.

\textcolor{black}{Since many practical situations involve passive designs, we will focus on this framework in the rest of the paper.
%As many practical applications involve offline designs, we consider the case of an i.i.d. passive offline design in this work.
The estimation of the minimizer in this case was first studied in \cite{HN87}.} They proposed an estimator that uses a kernel-based local estimator involving the first and second derivatives of the kernel to replace the unknown gradient and performs a Robbins-Monro-type update to estimate the root. They proved that, under suitable assumptions, it converges almost surely and in the quadratic mean with a $n^{(1-\gamma)/2}$ rate, for a parameter $1/5 < \gamma < 1/2$. \textcolor{black}{Later, \cite{Tsy90} proved that the minimax optimal rate for estimating the minimizer is 
\begin{equation}\label{rate tsybakov}
n^{\frac{-2(\beta-1)}{2\beta+d}}
\end{equation}
for the class of strongly convex $\beta$-Hölder ($\beta \geq 2$) regression functions.
%We recall that in problems of estimating a deterministic parameter from some observations, an estimator is called minimax if its maximal risk is minimal among all estimators. 
Using a Taylor polynomial approximation, \cite{Tsy90} also proposed an algorithm reaching this optimal rate. However, the marginal density of the design points needs to be known to define this estimator. This {\color{black} requirement} was removed by \cite{tsybakov} who devised an estimator that still reaches the optimal rate (up to a logarithmic factor) without knowing {\color{black} the density}.}

\textcolor{black}{The {\color{black} literature} about the estimation of the minimum value under i.i.d. passive design is relatively scarce. However, \cite{Ibragimov1982} or \cite{Lepskii1994} studied the problem of estimating the maximum of a univariate function observed in the Gaussian white noise model and \cite{tsybakov}
% proved the conjecture of \cite{Belitser2021} and
established that the minimax rate for estimating the minimum value of a regression function is of order 
\begin{equation}\label{rate tsybakovbis}
n^{-\beta/(2\beta+d)}
\end{equation}
over the class of strongly convex and $\beta$-Hölder functions with $\beta\geq 2$. Moreover, they {\color{black} proposed} a recursive algorithm achieving this optimal rate up to a logarithmic factor which does not require the knowledge of the density.}

\textcolor{black}{Thus, when the number of variables $d$ on which $f$ a priori depends is large, the convergence rates \eqref{rate tsybakov} and \eqref{rate tsybakovbis} are considerably altered. However,
%The dimension, called $d$ here, is often very large and one can notice it alters a lot the convergence rates of these methods.
%Nevertheless
in many applications, the function $f$ actually depends only on a few unknown variables among $x^{(1)},\dots,x^{(d)}$ usually called \textit{active} variables. This is the reason why we will assume in the following that there exists a function $g : \mathbb{R}^r\longrightarrow \mathbb{R}$, with $r<d$ such that
\begin{equation}\label{def : variables selection}
\forall x= (x^{(1)},...,x^{(d)}) \in \mathbb{R}^d, \hspace{1ex} f(x) = g(x^{(j_1)},...,x^{(j_r)}), \text{ with }
j_1 < j_2 < ... < j_r.
\end{equation}
This can be rewritten as follows:
\begin{equation} \label{def: g}
\forall x \in \mathbb{R}^d, f(x) = g(Wx),
\end{equation}
for some matrix $W$ of size $r\times d$ satisfying $W W^\top = I_r$, where $\top$ denotes the matrix transposition and $I_r$ is the identity matrix of size $r \times r $. Here, $W$ has exactly one 1 in each row and the other entries are equal to zero.}
% More precisely, we suppose that there exists a subset of indices $J = \{j_1,...,j_r\} \subset \{1,...,d\}$ such that  
% \begin{equation}\label{def : variables selection}
% \forall x= (x_1,...,x_d)^\top \in \mathbb{R}^d, \hspace{1ex} f(x) = g(x_{j_1},...,x_{j_r}), \text{ with }
% j_1 < j_2 < ... < j_r.
% \end{equation}

\textcolor{black}{In the sequel, we will focus on the problem of estimating the minimum of $f$ on a compact convex subset $\Theta$ of $\mathbb{R}^d$ as well as its location from $y_1,\dots,y_n$ and $x_1,\dots,x_n$ defined in (\ref{eq: regression model}) where $f$ is defined in (\ref{def: g}) without knowing neither $g$ nor $r$.}

\textcolor{black}{Finding $J = \{j_1,...,j_r\} \subset \{1,...,d\}$ will allow us to identify the relevant variables on which $f$ depends and thus $r$. It corresponds to the variable selection issue which has been extensively studied using regularized approaches since the end of the 90s in many different frameworks.
Efficient methodologies have been devised when the variables $x_1,\dots,x_n$ and $y_1,\dots,y_n$ in (\ref{eq: regression model}) are linearly related. Notable examples include the Lasso regression formulated by \cite{Tib96} which consists in penalizing the least-squares criterion by the $\ell_1$ norm of the vector to estimate and one of its variant, the Elastic Net, defined by \cite{Elastic_net} mixing in the penalization the $\ell_1$ and $\ell_2$ norms of the vector to estimate. When the relationship between {\color{black} $x_1,\dots, x_n$ and $y_1,\dots, y_n$} is non linear, the variable selection problem is much more challenging. 
For instance, \cite{HSIC} introduced a feature-wise kernelized Lasso method tailored for variable selection in \textcolor{black}{nonlinear} models called HSIC-Lasso and its variation NOCCO-Lasso. This approach uses the kernel trick within the regular Lasso problem.
Tree-based methods such as Random Forests introduced by \cite{breiman2001random} are also widely used for variable selection in regression models.
More recently, Artificial Neural Networks (ANN) have gained interest for variable selection and estimation in regression models.
We will present just a few examples of them. 
Regularized approaches with different ANN architectures were proposed in the work of \cite{li2016deep} where a ridge regularization approach is considered for the weights of the first and hidden layers. Similarly, \cite{feng2017sparse} introduced the SPINN method which is a single-layer neural network with a sparse group {\color{black} Lasso} regularization.
% They later proposed another approach with deeper neural networks in \cite{feng2022ensembled}. \cite{ye2018variable} adapted the SPINN method by adding a greedy elimination algorithm to iteratively remove one variable at a time and determine if the empirical loss decreases.
Later, \cite{lemhadri2021lassonet} proposed LassoNet, a residual feed-forward neural network architecture introduced in \cite{He2016residual} which incorporates a regularization approach based on the Lasso criterion. 
Another research direction focuses on developing methods using splines for piecewise polynomial fitting such as Multivariate Adaptive Regression Splines (MARS) introduced by \cite{friedman1991multivariate}. In the same way, \cite{lin2006component} developed COSSO, a regularization method for component selection and smoothing splines where the penalty term is the sum of the component norms.
A few articles have also proposed considering a local polynomial approximation which has been successfully used for performing variable selection in \cite{J} and has exhibited optimal performance in the minimization of a strongly convex and Hölder regression function as shown in \cite{tsybakov}.
}

\textcolor{black}{This is the reason why we shall prove in this paper that combining an adaptation of \cite{tsybakov} with the variable selection approach of \cite{J} leads to the optimal rate that one could expect for the estimation of the minimum of $f$ as well as its location as if the number of active variables $r$ was known beforehand.}

\textcolor{black}{This paper is organized as follows. Section \ref{section: method} describes our approach. Its theoretical properties are established in Section \ref{section: theoretical results} and the detailed proofs are given in Section \ref{sec:proofs} using the technical lemmas of Section \ref{section: technical lemmas} and the additional results given and proved in the Appendix.}

\section{Description of the method}\label{section: method}
The method we propose consists of two steps: we first estimate the \textcolor{black}{set of active variables, also called active subspace,} with the method given in \cite{J} and then we develop an optimization algorithm inspired by the one introduced in \cite{tsybakov} to estimate the \textcolor{black}{minimizers} and the minimum of {\color{black} the function} $f$. %{\color{magenta} [ajouter une phrase pour souligner qu'on utilise deux fois un "local polynomial approach"?]}
\subsection{Estimation of the active subspace}\label{section:variable selection}

We first explain how to estimate {\color{black} $J=\{ j_1,  j_2, \dots, j_r\}$}, the subset of indices defined in \eqref{def : variables selection} using the method proposed in \cite{J}. %{\color{teal}[NB $J$ est d\'efini bien apr\`es \eqref{def : variables selection}]}

\textcolor{black}{Let $x_0$ be a given point of $\mathbb{R}^d$ satisfying (\ref{ass: density}) and (\ref{ass: distinguishable properties}) given hereafter}
%\textcolor{black}{Let $x_0$ be a given point of $\Theta'$ defined as follows:
%\begin{equation}\label{eq:Theta_prime}
%  \Theta' = \{x+y; x\in \Theta, \Vert y \Vert \leq 1\},
%\end{equation}
%}
 and denote $\partial_k f(x_0)$ the partial derivative of $f$ with respect to its $k$th coordinate computed at $x_0$. \textcolor{black}{If $f$ is assumed to be smooth enough, typically $f$ belongs to $\mathcal{F}_\beta(L)$, {\color{black} the class of $\beta$-H\"older functions},  defined in (\ref{ass: holder}) of (\ref{ass: regularity f}),} the idea is to estimate $(f(x_0),\partial_1 f(x_0),...,\partial_d f(x_0))$  by using a local polynomial approach
\textcolor{black}{which consists in approximating $f$ by the polynomial of degree $\ell$ appearing in the Taylor expansion of $f$ at $x_0$ given in \eqref{eq:taylor_expansion}.
 {\color{black} More precisely,} let $S$ be the cardinality of the set $\{\m \in \mathbb{N}^d, \vert \m \vert \leq \ell \}$, namely
  $S= \#\{\m \in \mathbb{N}^d, \vert \m \vert \leq \ell \}$ with $\vert \m\vert=\vert(m_1,...,m_d)\vert = \sum_{i=1}^d m_i$ and{\color{black}, given $u\in\mathbb R^d$, let}
\begin{equation}\label{def U and S}
U(u)=\left(\frac{u^{\m^{(1)}}}{\m^{(1)}!},...,\frac{u^{\m^{(S)}}}{\m^{(S)}!}\right)^{\!\top}, \hspace{1.5ex}
\end{equation}
where $\{\m^{(1)},...,\m^{(S)}\}$ is the set $\{\m \in \mathbb{N}^d, \vert \m \vert \leq \ell \}$ ordered such that $\m^{(1)} = (0,...,0)$ and $\m^{(k)} = \left(\mathds{1}_{k-1}(j)\right)_{1\leq j\leq d}$, for $k \in \{2,...,d+1\}$, where $\mathds{1}_{i}(j)=1$ if $j=i$ and 0 if $j\neq i$. Moreover, for $\m=(m_1,...,m_d)\in \mathbb{N}^d$, and $u=(u_1,...,u_d)\in \mathbb{R}^d$, we use the notations $\m! = m_1!...m_d!$ and $u^{\m} = u_1^{m_1}...u_d^{m_d}$. Then,
\begin{equation}\label{eq:taylor_expansion}
  f(x)\approx\sum_{|\m|\leq\ell}\frac{1}{\m!}D^{\m}f(x_0)(x-x_0)^{\m}=\theta(x_0) U\left(\frac{x - x_0}{h}\right),
\end{equation}
where {\color{black} $h$ is a bandwidth and}
$$
D^{\m}f(x_0)=\frac{\partial^{|\m|}f}{\partial u_1^{m_1}\dots\partial u_d^{m_d}}(x_0)\quad\textrm{and}\quad\theta(x_0)=(h^{|\m^{(1)}|}D^{\m^{(1)}}f(x_0),\dots,h^{|\m^{(S)}|}D^{\m^{(S)}}f(x_0)).
$$
 Since $f$ depends only on a few variables, the vector $\theta(x_0)$ is sparse and has zero-coordinates when the partial derivatives involve variables having indices which do not belong to $J$.} This is the reason why we \textcolor{black}{will use a Lasso criterion for finding the active variables and hence estimating $J$.}

% For $\ell$ the degree of the polynomial used in the local polynomial approach, define \textcolor{black}{$S$ as the cardinality of the set $\{m \in \mathbb{N}^d, \vert m \vert \leq \ell \}$:
% $S= \#\{m \in \mathbb{N}^d, \vert m \vert \leq \ell \}$ with $\textcolor{black}{\vert m\vert}=\vert(m_1,...,m_d)\vert = \sum_{i=1}^d m_i$} and 
% \begin{equation}\label{def U and S}
% U(u)=\left(\frac{u^{m^{(1)}}}{m^{(1)}!},...,\frac{u^{m^{(S)}}}{m^{(S)}!}\right)^{\!\top}, \hspace{1.5ex}
% \end{equation}
% where $(m^{(1)},...,m^{(S)})$ is the set $\{m \in \mathbb{N}^d, \vert m \vert \leq \ell \}$ ordered such that $m^{(1)} = (0,...,0)$ and $m^{(k)} = \left(\mathds{1}_{k-1}(j)\right)_{1\leq j\leq d}$, for $k \in \{2,...,d+1\}$, \textcolor{black}{where $\mathds{1}_{i}(j)=1$ if $j=i$ and 0 if $j\neq i$}. Moreover, for $m=(m_1,...,m_d)\in \mathbb{N}^d$, and $u=(u_1,...,u_d)\in \mathbb{R}^d$, we use the notations $m! = m_1!...m_d!$ and $u^m = u_1^{m_1}...u_d^{m_d}$.

\textcolor{black}{In this section, we} take $\ell =1$ which implies $S=d+1$ and \textcolor{black}{estimate the vector}
\textcolor{black}{$\theta(x_0)=(f(x_0),h\;\partial_1f(x_0),\dots,h\;\partial_df(x_0))$ by using the following criterion:}
\begin{align}\label{def: thetahat}
  \textcolor{black}{\boldsymbol{\widehat{\theta}}(x_0,\lambda)}&\textcolor{black}{=(\widehat{\theta_0}(x_0,\lambda),...,\widehat{\theta_d}(x_0,\lambda))}\nonumber\\
  &=\arg \min_{\theta \in \mathbb{R}^{d+1}}
\left[
\frac{1}{nh^d} \sum_{i=1}^{n}
\left(
y_i - \theta U\left(\frac{x_i - x_0}{h}\right)
\right)^2
K^u\left(\frac{x_i - x_0}{h}\right)
+ 2\lambda \|\theta\|_1
\right],
\end{align}
{\color{black}where} 
\begin{equation}\label{Ku}
K^u = \frac{1}{2^d}\mathds{1}_{B_{\infty}(0,1)}
\end{equation} 
{\color{black}is} the uniform kernel on {\color{black}the unit ball} $B_{\infty}(0,1)$ with 
$B_{\infty}(x,\eta)= \{y\in\mathbb{R}^d; \Vert x-y \Vert_{\infty} \leq \eta \}$
for $x \in \mathbb{R}^d$ and $\eta >0$
{\color{black} and the $\ell_1$-norm is defined by $\Vert u \Vert_1 = \sum_{i=1}^d \vert {\color{black}u_i}\vert$ for all $u={\color{black}(u_1,...,u_d)}\in\mathbb{R}^d$.
}
%\textcolor{black}{In (\ref{def: thetahat}), $\Vert\cdot\Vert_1$ denotes the $\ell_1$-norm and is defined as follows: for all $x=(x_1,...,x_d)$ in $\mathbb{R}^d$, $\Vert x \Vert_1 = \sum_{i=1}^d \vert x_i\vert$.} 

The specific values to take for the bandwidth $h$ and the regularization parameter $\lambda$ to \textcolor{black}{get} the theoretical bounds proved in the next sections are given in Lemma \ref{lemma : J} of Section \ref{section: technical lemmas}.

The non-zero coordinates of $\textcolor{black}{\boldsymbol{\hat{\theta}}(x_0,\lambda)=(\widehat{\theta_1}(x_0,\lambda),...,\widehat{\theta_d}(x_0,\lambda))}$ defined in \eqref{def: thetahat} give
\textcolor{black}{the following estimator $\hat{J}$ of $J$}:
\begin{equation}\label{def : Jhat}
\textcolor{black}{\hat{J} = \hat{J}(\lambda)=}\left\{j,\hspace{1ex} \textcolor{black}{\widehat{\theta_j}(x_0,\lambda)} \neq 0 \right\},
\end{equation}
\textcolor{black}{where we removed the dependence in $\lambda$ to alleviate the notations.}

We then define $\hat{W}$ the estimator of $W$ as the matrix of size $(\#\hat{J})\times d$ such that every row has a 1 at the indices of $\hat{J}$, where \textcolor{black}{the cardinality $\#\hat{J}$ of $\hat{J}$ is not necessarily equal to $r$.}
We will see in the sequel that under mild assumptions, {\color{black} with high probability,} $\hat{J}=J$ which {\color{black} is equivalent to $\hat{W}=W$ by (\ref{def : variables selection})}. 
More precisely, Theorem 1 of \cite{J} implies that, under suitable assumptions,
\begin{equation}\label{eq: proba}
\mathbb{P}(\hat{J} \neq J) \leq c_0\exp(c_0d)\exp(-c_1n h^{d+2})
\end{equation}
for some positive constants $c_0$ and $c_1$ and {\color{black}a} bandwith $h$.
For the reader's convenience, this theorem is recalled in Lemma \ref{lemma : J} of Section \ref{section: technical lemmas}.

\subsection{Minimization}
We further explain here the minimization method in the case of a perfect reconstruction, \textcolor{black}{that is} when $ J=\hat{J}$ since  according to \eqref{eq: proba} the probability {\color{black} of having} $\hat{J}\neq J$ using the estimation procedure described just above is exponentially decreasing.

%One shall use the same procedure replacing the unknown $W$ by its estimator $\hat{W}$.

For $i \in \{1,...,n\}$ let us define $z_i = Wx_i$. \textcolor{black}{Since, {\color{black} by \eqref{def: g}},} $f(x_i) = g(z_i)$,  \eqref{eq: regression model} can be rewritten as follows: 
\begin{equation}\label{eq: reg g}
 \forall i \in \{1,....,n\}, \hspace{1.5ex} y_i = g(z_i)+\xi_i.
\end{equation}
We \textcolor{black}{shall use} the following method inspired by the one in \cite{tsybakov} to minimize {\color{black} the function} $g$.

For a step size $\eta_k={\color{black}1/k}$ and given $\widehat{\nabla g}$ an estimator of the gradient of $g$ we perform a gradient descent and  define the sequence
\begin{equation}\label{eq: gradient descent}
\left\{
    \begin{array}{ll}
       z^{(1)} \in W\Theta, \\
       z^{(k+1)} = \text{Proj}_{W\Theta}\left(z^{(k)} - \eta_k\widehat{\nabla g}(z^{(k)})\right),
    \end{array}
\right.
\end{equation}  %\textcolor{black}{ATTENTION : $z^{(k)}$ prête à confusion ça pourrait être interprété comme une puissance il faut écrire à mon avis $z^{(1)}$, $z^{(k)}$...}
where $\text{Proj}_{W\Theta}$ is the Euclidean projection onto the convex set $W\Theta = \{W\theta, \theta \in \Theta\}$. More generally, for a matrix $M$ and a set $X$, we denote
\textcolor{black}{$MX$ the set $\{Mx;\; x\in X\}$ when the matrix multiplication is possible.}

The estimator $\widehat{\nabla g}$ is obtained by using a local polynomial regression of degree \textcolor{black}{$\ell=\lfloor \beta \rfloor$}
  % \textcolor{black}{(A VERIFIER)}
 {\color{black} since, by (\ref{ass: holder}) of (\ref{ass: regularity f}), $f$ belongs to $\mathcal{F}_\beta(L)$}:
\begin{equation}\label{def: ghat}
\forall z \in W\Theta, \hspace{1.5ex} \widehat{\nabla g}(z)=h_n^{-1}\left(MB_{n,\lambda_n}(z)^{-1}D_n(z)\right),
\end{equation}
where $h_n$ is a {\color{black}kernel-}bandwidth and for any $z \in W\Theta$, {\color{black} and $U$ defined as in \eqref{def U and S},} we denote
\begin{equation}\label{def Bn}
B_{n}(z) = n^{-1}h_n^{-r}\sum_{i=1}^n U\left(\frac{z_i-z}{h_n}\right)U\left(\frac{z_i-z}{h_n}\right)^{\!\top} K\left(\frac{z_i-z}{h_n}\right)
\end{equation} %where $U$ is defined in \eqref{def U and S}
%so that 
and
\begin{equation}\label{def Bnlambda}
B_{n,\lambda_n} = B_n + \lambda_nI_S
\end{equation}
 with $\lambda_n$ a regularization parameter, $I_S$ the identity matrix of size $S\times S$ and  
%\begin{equation}\label{def D}
\[
D_{n}(z) = n^{-1}h_n^{-r}\sum_{i=1}^nU\left(\frac{z_i-z}{h_n}\right) K\left(\frac{z_i-z}{h_n}\right)y_i.
\]
%\end{equation}
In \eqref{def: ghat}, $M$ is the matrix such that
\[
%\begin{equation}\label{def M}
\forall i\in \{1,...,r\}, j\in \{1,...,S\}, \hspace{1.5ex} M_{i,j} = \left\{
    \begin{array}{ll}
       1 \mbox{ if } j=i+1, \\
       0 \mbox{ otherwise}.
    \end{array}
\right.
%\end{equation}
\]

Note that the values of the {\color{black}kernel-}bandwidth {\color{black} $h_n$} and of the regularization parameter {\color{black} $\lambda_n$ are given in \eqref{def: values hn et lambdan}}.

After $N\in \mathbb{N^*}$ iterations of \eqref{eq: gradient descent} we obtain an estimator 
{\color{black}
\begin{equation}\label{hatz_W}
\hat{z}_W=z^{(N)}
\end{equation}
}of $z^*$, the minimizer of $g$. {\color{black} However,} $W$ needs to be known. Hence, in practice we define $\hat{z}_{\hat{W}}$ after $N\in \mathbb{N^*}$ iterations of the same procedure but using $\hat{W}$ instead of $W$ \textcolor{black}{and $\hat{r}=\#\hat{J}$ instead of $r$.} Then, our final \textcolor{black}{estimator $\hat{x}$ of} a minimizer of $f$ on $\Theta$ is
% \begin{equation}\label{def xhat}
% \hat{x} = \text{Proj}_{\left \{ x \in \Theta; \hat Wx = \hat{z}_{\hat{W}}\right \}}\left( \hat W^\top \hat z_{\hat W}\right).
% \end{equation}
% Note that this definition means that $\hat z_{\hat W}$ is converted in a vector of $\mathbb{R}^d$ and, then, projected on $\Theta$ without changing the values {\color{black} corresponding to} the indices of $\hat J$.
\textcolor{black}{
 \begin{equation}\label{def xhat}
\hat{x} = \widetilde{\text{Proj}}_{\Theta}\left( \hat W^\top \hat z_{\hat W}\right),
\end{equation}
where $\widetilde{\text{Proj}}_{\Theta}$ means that the components of $\hat W^\top \hat z_{\hat W}$ associated to {\color{black}the} indices belonging to the complementary set of
$\hat{J}$ are projected onto $\Theta$ without changing the values corresponding to the indices of $\hat{J}$. Note that this is possible thanks to the projection in (\ref{eq: gradient descent})}.

Define the estimator $\hat{g}$ of $g$ as
\begin{equation}\label{def: g estimation}
\forall z \in W\Theta, \hspace{1.5ex} \hat{g}(z)={\color{black}\hat{g}_W(z)=}\widetilde MB_{n,\lambda_n}(z)^{-1}D_n(z),
\end{equation}
for $\widetilde M$ the matrix such that $\widetilde M_{1,1} =1$ and all the other entries are zeros.
%\textcolor{black}{ATTENTION : j'ai l'impression qu'il faut dire que cette matrice est une matrice ligne contenant un 1 en 1ère position et des zéros après.}
% In the same way as for the minimizer, the previous estimator requires the knowledge of $W$. 
{\color{black} Note that, by definition, the estimator $\hat g$ of $g$} requires the knowledge of $W$. 
Thus, we define $\hat{g}_{\hat{W}}$ the estimator of $g$ {\color{black} obtained by using} the same procedure but with $\hat{W}$ instead of $W$ \textcolor{black}{and $\hat{r}=\#\hat{J}$ instead of $r$.}

Then, we define our estimator of the minimum of $f$ on $\Theta$ as:
\begin{equation}\label{def: estimation min}
\widehat{f^*} = \hat{g}_{\hat{W}}\left(\hat{z}_{\hat{W}}\right).
\end{equation}
{\color{black} where $\hat{z}_{\hat{W}}$ is defined in \eqref{hatz_W}.}%{\color{gray} [je trouve qu'on a besoin de la ref quand on regarde le thm 3.3]}

To establish the theoretical properties of these estimators we shall use the following assumptions.

\begin{hypA}\label{ass: assumption K}
\textbf{Kernel}

We assume that the kernel $K$ is a $L_K$-Lipschitz function that has a compact support contained in the unit Euclidean ball, and satisfies 
{\color{black} the following conditions for $u\in\mathbb R^d$}
\begin{equation*}
K(u) \geq 0, \quad \int K(u)\mbox{d}u = 1, \quad \sup_{u\in \mathbb{R}^d} K(u) < \infty.
\end{equation*}
\end{hypA}

\begin{hypA}\label{ass: min g}
\textbf{Minimizer}

The function $g$ attains its minimum in $W\Theta$ at point $z^*$ and $\nabla g(z^*) = 0$.

\end{hypA}
\begin{hypA}\label{ass: regularity f}
\textbf{Regularity}
\begin{subhypA}\label{ass: holder} %{\color{magenta} [meme $L$ que dans (A1)?]} \textcolor{black}{je suis d'accord il faudrait mettre $L_1$ dans la 1ère hypothèse par exemple.}
  
 The function $f$ belongs to the class $\mathcal{F}_{\beta}(L)$ of $\beta$-Hölder functions with $\beta \geq 2$, where, for $\beta, L > 0$, we denote by $\mathcal{F}_\beta(L)$ the class of $\ell = \lfloor \beta \rfloor$ times differentiable functions 
$f : \mathbb{R}^d \to \mathbb{R}$ such that
\begin{equation}\label{eq: holder norme 2}
\forall x, x' \in \mathbb{R}^d, \quad
\left| 
f(x) - \sum_{\substack{\m \in \mathbb{N}^d, \\ |\m| \le \ell}} \frac{1}{\m!} D^{\m} f(x')\,{\color{black}(x - x')}^m 
\right|
\le L \|x - x'\|^\beta,
\end{equation}
where $\| \cdot \|$ denotes the Euclidean norm and for $\m = (m_1,...,m_d) \in \mathbb{N}^d$ and $u = (u_1,...,u_d) \in \mathbb{R}^d$,
%\begin{equation}\label{def: truc m}
\[
\vert \m \vert = \sum_{i=1}^d \vert m_i \vert, \quad \m! = m_1!...m_d!, \quad D^{\m}f = \frac{\partial^{\vert \m \vert}f}{\partial x_1^{m_1}...\partial x_d^{m_d}}, \quad u^{\m} = u_1^{m_1}...u_d^{m_d}.
%\end{equation}
\]
\end{subhypA}

\begin{subhypA}\label{ass: hessienne}
  {\color{black}{The Hessian matrix is such that $\lVert \nabla^2 f \lVert_{\text{op}}$ is bounded \textcolor{black}{by some positive constant on $\Theta$}, where
  {\color{black} $\lVert M \rVert_{\text{op}} = \sup_{x \in \mathbb{R}^q \backslash \{0\}}\frac{\lVert Mx \rVert}{\lVert x \rVert}$ for a given $p\times q$ matrix $M$.}
%      $\lVert \cdot \rVert_{\text{op}}$ of a matrix $M$ of size $p\times q$ is defined by $\lVert M \rVert_{\text{op}} = \sup_{x \in \mathbb{R}^q \backslash \{0\}}\frac{\lVert Mx \rVert}{\lVert x \rVert}$.
      }}
%This assumption is satisfied for $\beta \geq 3$.
%{\color{black} [$C$ : on a besoin de donner un nom \`a la constante? $C$ d\'ej\`a utilis\'e en (A6)]}
\end{subhypA}
\begin{subhypA}\label{ass: strong convexity}
For $\alpha > 1$, the function $g$ is $\alpha$-strongly convex on $W\Theta$ which means that for any $z,y \in W\Theta$, we have 
$$g(y) \geq g(z) +\nabla g(z)^\top(y-z)+\frac{\alpha}{2}\Vert z-y\Vert ^2.
$$
\end{subhypA}
\begin{subhypA}\label{ass: bounded}
  The function $f$ is uniformly bounded on the set
\begin{equation}\label{eq:Theta_prime}
  \Theta' = \{x+y; x\in \Theta, \Vert y \Vert \leq 1\}.
\end{equation}
  
\end{subhypA}
\end{hypA}

\begin{hypA}\label{ass: noise}\textbf{Noise}

%For all $i,j \in \{1,...n\}$ 
{\color{black}We} assume that {\color{black} $\xi_1,\dots, \xi_n$ and $x_1,\dots, x_n$} are independent. We also suppose that {\color{black}$\xi_1,\dots,\xi_n$} are centered and satisfy the sub-Gaussian assumption:
$$\exists \sigma>0,\quad\forall i \in \{1,...n\}, \quad \mathbb{E}\left(\exp(\vert \xi_i \vert /\sigma)\right)<\infty.$$
\end{hypA}

\begin{hypA}\label{ass: density}
\textbf{Density}

{\color{black}We assume that $x_1, \dots, x_n$} are {\color{black} i.i.d.} with a density $p$ and {\color{black} that there exist} some constants $\eta$, $p_m > 0$, $p_M \geq 1$ and $\tilde{L} > 0$ such that
\begin{itemize}
  \item $p$ is $\tilde{L}$-Lipschitzian around $x_0$, that is for any $t \in \mathcal{B}_\infty(x_0,1)$,
\[
|p(x_0) - p(t)| \leq \tilde{L} \|x_0 - t\|_\infty,
\]
%\item $\mathcal{B}_\infty(x_0,\eta) \subset \operatorname{supp}(p)$
\item $\forall x \in \Theta'\cup\{x_0\}, \quad \mathcal{B}_\infty(x,\eta) \subset \operatorname{supp}(p)$ and $p_m \leq p(y)$ for every $y \in \mathcal{B}_\infty(x,\eta)$,
  \textcolor{black}{where $\Theta'$ is defined in (\ref{eq:Theta_prime})},
\item $\forall x \in \mathbb{R}^d, \quad p(x)\leq p_{M} <\infty$.
\end{itemize}

\end{hypA}

\begin{hypA}\label{ass: distinguishable properties}
  \textbf{Distinguishability}
  
  There exists a constant $C \geq 1512\left(\frac{p_{M}}{p_{m}}\right)L\sqrt{r}$ such that
\begin{subhypA}\label{ass: distinguishable}
 $\vert \partial_j f(x_0)\vert \geq C$ for any $j\in J$, where $p_{m}$ and $p_{M}$ are defined in \eqref{ass: density} \textcolor{black}{and} $L$ is the Hölder constant defined in \eqref{ass: holder} of \eqref{ass: regularity f},

\end{subhypA}
\begin{subhypA}\label{ass: distinguishable2}
$\vert f(x_0) \vert  > Ch$ where $h$ {\color{black} is defined in Lemma \ref{lemma : J} and corresponds to} the bandwidth used in the variable selection method.
\end{subhypA}
\end{hypA}

\section{Theoretical results} \label{section: theoretical results}
 
By \eqref{ass: min g}, $g$ has a unique minimizer $z^*$ in the space $W\Theta {\color{black}= \{Wx, x\in \Theta\}}$.
%{\color{salmon} [deja defini apr\`es \eqref{eq: gradient descent}, \`a enlever ?]}.
We can then write $z^* = Wx^*$, where $x^* \in \Theta$ {\color{black} is a minimizer of {\color{black} the function} $f$ defined} {\color{black} in \eqref{def: g}}. Since $f$ has the same value for all the elements of $\text{Ker}(W)+x^*$, estimating a minimizer of {\color{black} the} function $f$ on $\Theta$ means approaching $\left(\text{Ker}(W)+x^*\right)\cap \Theta$.

Hence, in order to assess the estimation error we shall consider the quadratic error:
\begin{equation}\label{def e}
\lVert W(x^* -\hat{x})\rVert^2=\lVert W^\top W(x^* -\hat{x})\rVert^2, 
\end{equation}
for some estimator $\hat{x} \in \Theta$. Note that the equality holds since $WW^\top = {\color{black}I_r}$.

%In the following results, we \textcolor{black}{will use the following notation:} $\vert \Theta \vert = \sup_{x,y \in \Theta} \lVert x-y \rVert$.
\textcolor{black}{To present} our main results we need to introduce the Lambert functions 
\begin{equation}\label{def: lambert functions}
%\mathcal{W}:
{\color{black} \mathcal{W}_{-1}}: \quad [-e^{-1}, 0[ \quad \longrightarrow \quad ]-\infty,-1] \quad \text{ and } \quad 
%\mathcal{W}_{-1}
{\color{black} \mathcal{W}} :\quad  [-e^{-1}, +\infty[\quad \longrightarrow \quad [-1,+\infty[
\end{equation}
as the inverses of $x \mapsto xe^x$ on \textcolor{black}{their domains.}

For $\beta \geq 2$, we define 
% \begin{equation}\label{def: N beta}
% N^0_{n,\beta} = \left\lceil \left(\frac{n}{\mathcal{W}(n)}\right)^{\frac{2(\beta-1)}{2\beta +r}}\right\rceil, \quad N^1_{n,\beta} = \left\lfloor \left(\frac{n}{\mathcal{W}(n)} \right)^{ \frac{2(\beta-1)}{2\beta +r}}\right\rfloor,
% \end{equation} where $\lfloor \cdot \rfloor$ and $\lceil \cdot \rceil$ respectively denotes the floor and the ceiling function.
\textcolor{black}{\begin{equation}\label{def: N beta}
 N_{n,\beta} = \left\lceil \left(\frac{n}{\mathcal{W}(n)}\right)^{\frac{2(\beta-1)}{2\beta +r}}\right\rceil,
\end{equation}
where $\lceil \cdot \rceil$ denotes the ceiling function.}

In this section we also \textcolor{black}{use the following {\color{black}kernel-}bandwidth and regularization parameter:}
\begin{equation}\label{def: values hn et lambdan}
 h_n =  \left(\frac{\mathcal{W}(n)}{n}\right)^{\frac{1}{2\beta+r}} \text{, } \lambda_n=  \left(\frac{\mathcal{W}(n)}{n}\right)^{\frac{\beta}{2\beta+r}}.
\end{equation}

\begin{theorem}\label{thm:minimizer} \textbf{About the minimizer estimation}
{\color{black}Let  $\mathcal W $ and $\mathcal{W}_{-1}$ be the Lambert functions defined in \eqref{def: lambert functions}. Let $c_1$ and $h$ be the positive constants introduced in \eqref{eq: proba} and let 
\[
n \geq \max\left(-\frac{\mathcal{W}_{-1}(-c_1h^{d+2})}{c_1h^{d+2}}\mathds{1}_{\{c_1h^{d+2}< e^{-1}\}},3\right).
\]}
Assume that \eqref{ass: assumption K} -- \eqref{ass: distinguishable properties} hold and that {\color{black} $y_1,\dots, y_n$} satisfy Model \eqref{eq: regression model} {\color{black} where {\color{black} $\xi_1,\dots,\xi_n$} are assumed to be Gaussian random variables.}
{\color{black} Consider} the estimator $\hat{x}$  of a minimizer of $f$ on $\Theta$ defined in \eqref{def xhat}. 
{\color{black} {\color{black} Let} the {\color{black}kernel-}bandwidth $h_n$ and the regularization parameter $\lambda_n$ {\color{black}be} defined in \eqref{def: values hn et lambdan} and {\color{black} assume}
that the  number of iterations is such that}
\begin{equation}\label{eq:minoration_N}
N > \textcolor{black}{N_{n,\beta}}\left(\left(\frac{n}{\mathcal{W}(n)}\right)^{\frac{2(\beta-1)}{2\beta+r}}\vert\Theta\vert^2\left(\frac{\alpha-1}{\alpha+1}\right)+1\right)
\end{equation}
% $$N > N^0_{n,\beta}\left(\left(\frac{n}{\mathcal{W}(n)}\right)^{\frac{2(\beta-1)}{2\beta+r}}\vert\Theta\vert\left(\frac{\alpha-1}{\alpha+1}\right)+1\right)\text{ and } n \geq \max\left(-\frac{\mathcal{W}_{-1}(-c_1h^{d+2})}{c_1h^{d+2}}\mathds{1}_{\{c_1h^{d+2}< e^{-1}\}},3\right),$$ $c_1$ and $h$ being positive constants introduced in \eqref{eq: proba}, satisfies
%\[
%\mathbb{E}\left(\lVert W (x^*-\hat{x})\rVert^2 \right)\leq B \left(\frac{\mathcal{W}(n)}{n}\right)^{\frac{2(\beta-1)}{2\beta+r}},
%\]
where \textcolor{black}{$N_{n,\beta}$} is defined in (\ref{def: N beta}) {\color{black} and $\vert \Theta \vert = \sup_{x,y \in \Theta} \lVert x-y \rVert$}. Then 
\[
\mathbb{E}\left(\lVert W (x^*-\hat{x})\rVert^2 \right)\leq B \left(\frac{\mathcal{W}(n)}{n}\right)^{\frac{2(\beta-1)}{2\beta+r}},
\]
for \textcolor{black}{some} positive constant $B$. 
\end{theorem}

%\textcolor{black}{on avait dit qu'on ferait une remarque sur l'hypothèse gaussienne. Il faut parler de l'optimalité de la vitesse.}
\textcolor{black}{Note that the Gaussian assumption is only used for identifying the active variables on which $f$ depends. Note also that based on the bounds established by \cite{tsybakov}, the upper bound obtained in Theorem \ref{thm:minimizer} is optimal up to a factor depending on the Lambert function which is smaller than the logarithm as mentioned in (\ref{eq:majoration_W_log}) {\color{black} of Lemma  \ref{lemma: proporties lambert function}}. Hence, we slightly improved the results of \cite{tsybakov}  in the more challenging framework where the number of variables {\color{black}$r$} and the variables {\color{black} themselves} on which $f$ actually depends {\color{black}are} not known and we show that our approach achieves the same rates as if the active variables were known beforehand.}

Theorem \ref{thm:minimizer} is proved in Section \ref{section: thm minimizer proof} and comes from the following theorem which is proved in Section \ref{section: proof tsybakov}.

\begin{theorem}\label{thm: tsybakov}
{\color{black} Let $\mathcal W $ and $\mathcal{W}_{-1}$ be the Lambert functions defined in \eqref{def: lambert functions}}.
{\color{black} Assume that \eqref{ass: assumption K} -- \eqref{ass: density}} hold and that {\color{black}$y_1,\dots, y_n$} satisfy \eqref{eq: reg g}. Let the {\color{black}kernel-}bandwidth {\color{black} $h_n$} and the regularization parameter {\color{black}$\lambda_n$} be {\color{black} as} in \eqref{def: values hn et lambdan} {\color{black} and let $N_{n,\beta}$
  %$N^0_{n,\beta}$ and $ N^1_{n,\beta}$
  be defined in  \eqref{def: N beta}.} 
Then, for \textcolor{black}{$z^{(N)}$}, the last iteration of the gradient descent \eqref{eq: gradient descent} with step $\eta_k = k^{-1}$, %\frac{1}{k}$, 
we have \textcolor{black}{for $n\geq 3$},
\begin{equation}\label{eq:thm5.1}
	\text{if } N > \textcolor{black}{N_{n,\beta}}, \quad \mathbb{E}(\lVert \textcolor{black}{z^{(N)}}- z^* \rVert^2) \leq \frac{N_{n,\beta} -1}{N-1} {\color{black} \vert\Theta\vert^2} + \frac{N-N_{n,\beta}}{N-1}A\left(\frac{\alpha+1}{\alpha-1}\right)\left(\frac{\mathcal{W}(n)}{n}\right)^{\frac{2(\beta-1)}{2\beta +r}},
\end{equation}

\begin{equation}\label{eq:thm5.2}
  \text{if } 2\leq N \leq \textcolor{black}{N_{n,\beta}}, \quad \mathbb{E}(\lVert \textcolor{black}{z^{(N)}}- z^* \rVert^2) \leq A\left(\frac{\alpha+1}{\alpha-1}\right)\frac{1+\log(N-1)}{N-1},
\end{equation}
for some positive constant $A$. %{\color{salmon}[meme $A$ dans les deux inegalit\'es?]}
\end{theorem}

% {\color{magenta} Il me semble que dans l'\'eq \eqref{eq:thm5.1} il a $\vert\Theta\vert^2$ et pas $\vert\Theta\vert$ [cfr p.9]}

\textcolor{black}{\begin{remark}\label{rmk1}
In practice, the exponent $\beta$ appearing in the H\"older condition is often unknown. Since $\beta \geq 2$ {\color{black} by \eqref{ass: holder} of \eqref{ass: regularity f}, }we can take
\begin{equation*}
h_n =  \left(\frac{\mathcal{W}(n)}{n}\right)^{\frac{1}{4+r}} \text{, } \lambda_n=  \left(\frac{\mathcal{W}(n)}{n}\right)^{\frac{4}{4+r}}
\end{equation*}
in practical situations. For those values of bandwidth and regularization parameters, Inequalities \eqref{eq:thm5.1} and \eqref{eq:thm5.2} hold by replacing $\beta$ by $2$.
\end{remark}
\begin{remark} The main difference between the results presented above and those of \cite{tsybakov} is that the logarithmic factor is replaced by the Lambert function  $\mathcal{W}$. This modification slightly improves the bound \textcolor{black}{by using (\ref{eq:majoration_W_log}) of Lemma  \ref{lemma: proporties lambert function}} and recovers the $\log$ at the first order, since
%The structure of the bounds is identical to the ones in \cite{tsybakov}, except that the logarithmic factor is replaced by $\mathcal{W}(n)$.
%Using 
the first-order Taylor expansion of the Lambert function at infinity is:
\begin{equation*}
\mathcal{W}(x)
= \log x
-
\log\log x
+
o(\log\log x).
\end{equation*}
% one can see that the bound is slightly improved and that we recover the $\log$ at the first order.
Moreover, Theorem \ref{thm: tsybakov} provides an upper bound for the quadratic risk of $z^{(N)}$ whatever the number of iterations $N$ and the number of observations $n$ which is just assumed to be greater than 3.
\end{remark}
}

\begin{theorem}\label{thm: real minimum} \textbf{About the minimum estimation}
{\color{black}Let  $\mathcal W $ and $\mathcal{W}_{-1}$ be the Lambert functions defined in \eqref{def: lambert functions}. Let $c_1$ and $h$ be the positive constants introduced in \eqref{eq: proba} and let 
\[
n \geq \max\left(-\frac{\mathcal{W}_{-1}(-c_1h^{d+2})}{c_1h^{d+2}}\mathds{1}_{\{c_1h^{d+2}< e^{-1}\}},3\right).
\]}
Assume that \eqref{ass: assumption K} -- \eqref{ass: distinguishable properties} hold and that {\color{black} $y_1,\dots, y_n$} satisfy Model \eqref{eq: regression model} {\color{black} where {\color{black}$\xi_1,\dots, \xi_n$} are} Gaussian random variables. 
{\color{black} Consider} the estimator $\widehat f^*$  of the minimum of $f$ on $\Theta$ defined in \eqref{def: estimation min}
{\color{black} using the {\color{black}kernel-}bandwidth $h_n$ and the regularization parameter $\lambda_n$ defined in \eqref{def: values hn et lambdan}.}
\textcolor{black}{Then, after a number of iterations }
\[
N > N_{n,\beta}\left(\left(\frac{n}{\mathcal{W}(n)}\right)^{\frac{2(\beta-1)}{2\beta+r}}\vert\Theta\vert^2\left(\frac{\alpha-1}{\alpha+1}\right)+1\right),
\]
{\color{black} with $N_{n,\beta}$ defined in \eqref{def: N beta}, }
\begin{equation*}
\mathbb{E}\left(\left \lvert f(x^*)-\widehat{f^*}\right\rvert \right)\leq B'\left(\frac{\mathcal{W}(n)}{n}\right)^{\frac{\beta}{2\beta +r}}, 
\end{equation*}
where $B'$ is a positive constant {\color{black} and $x^*\in \Theta$ is a minimizer of $f$}. 
\end{theorem}

This {\color{black} theorem} is proved in Section \ref{section: thm real minimum proof}.

%\textcolor{black}{Il faut parler de l'optimalité de la vitesse ???}

\textcolor{black}{As for the previous theorem we note that, based on the bounds established by \cite{tsybakov}, the upper bound obtained in Theorem \ref{thm: real minimum} is optimal up to a factor smaller than a logarithmic term. Moreover, we observe that we obtain the same rates of convergence as if the active variables on which $f$ actually depends were known beforehand, which is obviously not the case in practice. We thus slightly improve the results of \cite{tsybakov} in a more challenging context.}
\section{Proofs}\label{sec:proofs}

\subsection{Proof of Theorem \ref{thm: tsybakov}}\label{section: proof tsybakov}
By definition of the Algorithm \eqref{eq: gradient descent} and the contracting property of the Euclidean projection onto a convex set, if $N$ is the total number of iterations, we have for $k \in \{1,...,N-1\}$.
\begin{align*}
\mathbb{E}(\lVert z^{(k+1)} - z^* \rVert^2 ) = & \mathbb{E}\left(\lVert\text{Proj}_{W\Theta}\left( z^{(k)} -\eta_k \widehat{\nabla g}(z^{(k)})\right) - z^* \rVert^2 \right)\\
\leq & \mathbb{E}\left(\lVert z^{(k)} -\eta_k \widehat{\nabla g}(z^{(k)}) - z^* \rVert^2\right)\\\
= &{\color{black} \mathbb{E}\left(\lVert z^{(k)} - z^* \rVert^2\right)- 2\eta_k \mathbb{E}\left((z^{(k)} - z^*)^{\!\top}  \widehat{\nabla g}(z^{(k)}) \right)+\eta_k^2\mathbb{E}\left(\lVert  \widehat{\nabla g}(z^{(k)}) \rVert^2 \right).}
\end{align*}
%
%Thus, 
%\begin{align*}
%\mathbb{E}(\lVert z^{(k+1)} - z^* \rVert^2) \leq  \mathbb{E}\left(\lVert z^{(k)} - z^* \rVert^2\right)- 2\eta_k \mathbb{E}\left((z^{(k)} - z^*)^{\!\top}  \widehat{\nabla g}(z^{(k)}) \right)+\eta_k^2\mathbb{E}\left(\lVert  \widehat{\nabla g}(z^{(k)}) \rVert^2 \right).
%\end{align*}

Using the Cauchy-Schwarz inequality, we get 
\begin{equation*}
-\mathbb{E}\left[(z^{(k)} - z^*)^{\!\top}  \widehat{\nabla g}(z^{(k)}) \right] \leq  -\mathbb{E}\left[(z^{(k)}-z^*)^{\!\top} \nabla g (z^{(k)}) \right]+\mathbb{E}\left(\lVert z^{(k)} - z^* \rVert \lVert \widehat{\nabla g}(z^{(k)}) - \nabla g(z^{(k)}) \rVert\right).
\end{equation*}

Using the fact that $g$ is $\alpha$-strongly convex by \eqref{ass: strong convexity} of \eqref{ass: regularity f} we have
\begin{equation*}
g(z^*) -  g(z^{(k)}) \geq (z^* -  z^{(k)})^{\!\top} \nabla g (z^{(k)}) + \frac{\alpha}{2} \Vert z^* - z^{(k)} \Vert^2.
\end{equation*}

As $z^*$ minimizes $g$ we get 
\begin{equation*}
(z^{(k)}-z^*)^{\!\top} \nabla g (z^{(k)}) \geq \frac{\alpha}{2} \Vert z^* - z^{(k)} \Vert^2
\end{equation*}

and thus 
\begin{equation*}
-\mathbb{E}\left[(z^{(k)}-z^*)^{\!\top} \nabla g (z^{(k)})\right] \leq -\frac{\alpha}{2}\mathbb{E}\left[ \Vert z^* - z^{(k)} \Vert^2\right].
\end{equation*}

Hence,
\begin{align*}
\mathbb{E}(\lVert z^{(k+1)} - z^* \rVert^2)& \leq  (1 - \alpha \eta_k)\mathbb{E}\left(\lVert z^{(k)} - z^* \rVert^2\right) +2\eta_k \mathbb{E}\left[\lVert z^{(k)} - z^* \rVert \lVert \widehat{\nabla g}(z^{(k)}) - \nabla g(z^{(k)}) \rVert\right] \\
& + \eta_k^2\mathbb{E}\left[\lVert  \widehat{\nabla g}(z^{(k)}) \rVert^2 \right].
\end{align*}

Since for all $a,b$ in $\mathbb{R}$ and  $\gamma > 0$, $2ab \leq \gamma a^2+\frac{1}{\gamma}b^2$ we get 

\begin{align*}
\mathbb{E}(\lVert z^{(k+1)} - z^* \rVert^2 )
 & \leq  (1 - \alpha \eta_k)\mathbb{E}\left(\lVert z^{(k)} - z^* \rVert^2\right)
+\gamma\eta_k \mathbb{E}\left(\lVert z^{(k)} - z^* \rVert^2\right) \\ & \qquad + \frac{\eta_k}{\gamma} \mathbb{E}\left(\lVert \widehat{\nabla g}( z^{(k)}) - \nabla g(z^{(k)}) \rVert^2\right)  + \eta_k^2\mathbb{E}\left(\lVert  \widehat{\nabla g}(z^{(k)}) \rVert^2 \right)\\
& \leq  (1 - (\alpha-\gamma)\eta_k)\mathbb{E}\left(\lVert z^{(k)} - z^* \rVert^2\right) \\
&\qquad + \frac{\eta_k}{\gamma} \mathbb{E}\left(\lVert \widehat{\nabla g}( z^{(k)}) - \nabla g(z^{(k)}) \rVert^2\right)
+ \eta_k^2\mathbb{E}\left[\lVert  \widehat{\nabla g}(z^{(k)}) \rVert^2 \right].
\end{align*}

Denoting $r_k = \mathbb{E}(\Vert z^{(k)} - z^* \Vert^2)$ and taking the supremum over $W\Theta$ in the previous inequality we have

\begin{equation}\label{rk+1}
r_{k+1} \leq  (1 - (\alpha-\gamma)\eta_k)r_{k}
+ \frac{\eta_k}{\gamma}
\mathbb{E}\left(\sup_{z \in W\Theta}\lVert \widehat{\nabla g}(z) - \nabla g(z) \rVert^2\right) + \eta_k^2
\mathbb{E}\left(\sup_{z \in W\Theta}\lVert  \widehat{\nabla g}(z) \rVert^2 \right).
\end{equation}

%Moreover, we have
{\color{black} The second term on the right-hand side can be bounded by}
\begin{align*}
\mathbb{E}\left[\sup_{z\in W\Theta} \lVert  \widehat{\nabla g}(z)-\nabla g(z)\rVert^2\right]
&\leq
2\mathbb{E}
\left[
\sup_{z\in W\Theta}
\left(
\lVert \widehat{\nabla g}(z)-\mathbb{E}\widehat{\nabla g}(z)\rVert^2
\right)
\right]+
2\mathbb{E}
\left[
\sup_{z\in W\Theta}
\left(
\lVert
\mathbb{E}\widehat{\nabla g}(z)-\nabla g(z)
\rVert^2
\right)
\right].
\end{align*}
%which can be seen as decomposition of the error in terms of bias and variance.

\textcolor{black}{Since \eqref{ass: assumption K} -- \eqref{ass: density} hold,} {\color{black} according to the definition of $h_n$ and $\lambda_n$ in \eqref{def: values hn et lambdan}, we have that $nh_n^r= \lambda_n^{-2}\mathcal{W}(n)$. Moreover, by (\ref{eq:croissance W}) {\color{black}of Lemma \ref{lemma: proporties lambert function}}, $nh_n^r \geq\mathcal{W}(n)$. Thus, $nh_n^r \geq \max(\mathcal{W}(n),\lambda_n^{-2}\mathcal{W}(n))$ so we can use Lemma \ref{lemma 12} to bound the first term of the above inequality and Lemma \ref{lemma 11} for the second one and hence obtain}
%Hence, we get
%{\color{black} Using the fact that, by (\ref{eq:minoration_W}), $\mathcal W(n) \geq 1$ for $n\geq3$,} we get 
\begin{align*}
\mathbb{E}\left[\sup_{z\in W\Theta}
\lVert \widehat{\nabla g}(z)-\nabla g(z)\rVert^2
\right]
&\leq
2A_2h_n^{-r-2}n^{-1}\mathcal{W}(n)
+2A_1^2
\left(
h_n^{\beta-1}
+h_n^{-1}\lambda_n
+h_n^{-1-r/2}n^{-1/2}
\right)^2
\\
&\leq
2A_2h_n^{-r-2}n^{-1}\mathcal{W}(n)
+
4A_1^2
\left(
h_n^{2(\beta-1)}
+h_n^{-2}\lambda_n^2
+h_n^{-2-r}n^{-1}
\right).
\end{align*}
\textcolor{black}{By (\ref{eq:minoration_W}) {\color{black}of Lemma \ref{lemma: proporties lambert function}}, we thus get that
$$\mathbb{E}\left[\sup_{z\in W\Theta}
\lVert \widehat{\nabla g}(z)-\nabla g(z)\rVert^2
\right]\leq
{\color{black}A'}
\left(
h_n^{-r-2}n^{-1}\mathcal{W}(n)
+
h_n^{2(\beta-1)}
+
h_n^{-2}\lambda_n^2
\right),
$$}
%{\color{magenta} [je ne suis pas sure que la constante soit exacte]} 
\textcolor{black}{where $A'=2\max\left(2A_2,4A_1^2\right)$,}
% {\color{magenta}[\`a enlever?]}{\color{gray} Note that balancing the terms leads to
% \begin{equation*}h_n =  \left(\frac{\mathcal{W}(n)}{n}\right)^{\frac{1}{2\beta+r}} \text{ and } \lambda_n=  \left(\frac{\mathcal{W}(n)}{n}\right)^{\frac{\beta}{2\beta+r}}.
% \end{equation*}
% }
% %Note that for these values we have $nh_n^r = \lambda^{-2}\mathcal{W}(n)$ and $nh_n^r \geq \mathcal{W}(n)$ so the assumption made to apply Lemma \ref{lemma 11} and Lemma \ref{lemma 12} is satisfied.
% {\color{magenta} Hence, defining} $A' = 2\max\left(A_2,2A_1^2\right)$
\textcolor{black}{which, by (\ref{def: values hn et lambdan}), gives}
\begin{equation}\label{eq: decomposition g}
\mathbb{E}
\left[
\sup_{z\in W\Theta}
\lVert \widehat{\nabla g}(z)-\nabla g(z)\rVert^2
\right]
\leq
\textcolor{black}{3A'}
\left(
\frac{\mathcal W(n)}{n}
\right)^{\frac{2(\beta-1)}{2\beta+r}}.
\end{equation}
\textcolor{black}{Using this bound in \eqref{rk+1} gives}
\begin{equation}\label{eq: recurrence rk}
r_{k+1}
\leq
(1-(\alpha-\gamma)\eta_k)r_k
+
\textcolor{black}{3}\frac{\eta_k}{\gamma}
A'
\left(
\frac{\mathcal W(n)}{n}
\right)^{\frac{2(\beta-1)}{2\beta+r}}
+
\eta_k^2
\mathbb E
\left[
\sup_{z\in W\Theta}
\lVert \widehat{\nabla g}(z)\rVert^2
\right].
\end{equation}
{\color{black} Observe that}, by \eqref{ass: holder} of \eqref{ass: regularity f}, $\lVert\nabla f\rVert$ is uniformly bounded on $\Theta$ which implies that $\Vert \nabla g\Vert$ is uniformly bounded on $W\Theta$. \textcolor{black}{Thus, by \eqref{eq: decomposition g} and (\ref{eq:croissance W}), we get}
\begin{equation*}
\mathbb E
\left[
\sup_{z\in W\Theta}
\lVert \widehat{\nabla g}(z)\rVert^2
\right]
\leq
\mathbb E
\left[
\sup_{z\in W\Theta}
\lVert\nabla g(z)\rVert^2
\right]
+
\mathbb E
\left[
\sup_{z\in W\Theta}
\lVert
\widehat{\nabla g}(z)-\nabla g(z)
\rVert^2
\right] \leq A'',
\end{equation*}
%we get 
%\begin{equation*}
%\mathbb E
%\left[
%\sup_{z\in W\Theta}
%\lVert \widehat{\nabla g}(z)\rVert^2
%\right]
%\leq A'',
%\end{equation*}
for some positive constant $A''$. %{\color{magenta}[ajouter une ligne pour expliquer comment on obtient une constante?]}

{\color{black} Since by (\ref{ass: strong convexity}) of (\ref{ass: regularity f}), $\alpha>1$,} we {\color{black} take $\gamma = (\alpha-1)/2 >0$} {\color{black} in (\ref{rk+1})}. 
%and have $\gamma>0$ since $\alpha >1$ by \eqref{ass: strong convexity} of \eqref{ass: regularity f}. 
{\color{black} Since $\gamma \leq \alpha-1$, $1-(\alpha-\gamma)\eta_k\leq 1-\eta_k$ hence,  for $\eta_k = 1/k$,} \eqref{eq: recurrence rk} becomes
\begin{equation}\label{rk+1_bis}
r_{k+1} \leq  \left(1 - \frac{1}{k}\right)r_k+ \frac{2}{k(\alpha-1)} A\left(\frac{\mathcal{W}(n)}{n}\right)^{\frac{2(\beta-1)}{2\beta+r}}  + \frac{1}{k^2}A,
\end{equation}
where \textcolor{black}{$A = \max(3A',A'')$.}
\textcolor{black}{We will consider two cases: $k\geq\left(\frac{n}{\mathcal{W}(n)}\right)^{\frac{2(\beta-1)}{2\beta+r}}$ and $k<\left(\frac{n}{\mathcal{W}(n)}\right)^{\frac{2(\beta-1)}{2\beta+r}}$.}
%which of the term dominates the other.

{\color{black} Let us first consider the case where $k \geq \left(\frac{n}{\mathcal{W}(n)}\right)^{\frac{2(\beta-1)}{2\beta+r}}$. Then,}
\begin{equation*}
r_{k+1} \leq  \left(1 - \frac{1}{k}\right)r_k+\frac{\alpha+1}{ k(\alpha-1)}  A\left(\frac{\mathcal{W}(n)}{n}\right)^{\frac{2(\beta-1)}{2\beta+r}}.
\end{equation*}
Denoting $a =\frac{\alpha+1}{\alpha-1}A\left(\frac{\mathcal{W}(n)}{n}\right)^{\frac{2(\beta-1)}{2\beta+r}}$, \textcolor{black}{the previous inequality can be rewritten as follows for all $k \geq \textcolor{black}{N_{n,\beta}}$:}
\begin{equation*}
r_{k+1} \leq  \left(1 - \frac{1}{k}\right)r_k+ \frac{a}{k},
\end{equation*}
which, by induction, implies
\begin{equation*}
r_{k+1} \leq r_{N_{n,\beta}}\prod_{i=\textcolor{black}{N_{n,\beta}}}^k \left(1-\frac{1}{i}\right)+ \frac{k-\textcolor{black}{N_{n,\beta}}+1}{k}a, 
\end{equation*}
{\color{black}where} $r_{\textcolor{black}{N_{n,\beta}}} = \mathbb{E}\left(\Vert \textcolor{black}{z^{(N_{n,\beta})}}-z^*\Vert^2\right)\leq \vert\Theta\vert^2$. \textcolor{black}{Hence we obtain that for all $k \geq N_{n,\beta}$,}
\begin{equation*}
\quad r_{k+1} \leq \frac{N_{n,\beta} -1}{k} \vert\Theta\vert^2 + \frac{k-N_{n,\beta}+1}{k}\left(\frac{\alpha+1}{\alpha-1}\right)A\left(\frac{\mathcal{W}(n)}{n}\right)^{\frac{2(\beta-1)}{2\beta+r}},
\end{equation*}
which leads to \eqref{eq:thm5.1} {\color{black} by taking $k=N-1\geq N_{n,\beta}$}. %{\color{magenta}[Attention: dans  \eqref{eq:thm5.1} il n'y a pas le carr\'e]}. 
%Note that as $N>N_{n,\beta}^0$ by assumption, we have $N-1 \geq N_{n,\beta}^0$ so the choise $k=N-1$ is valid here.
%
{\color{black} Let us now consider the case where $k < \left(\frac{n}{\mathcal{W}(n)}\right)^{\frac{2(\beta-1)}{2\beta+r}}$. Then,}
Inequality \eqref{rk+1_bis} becomes
\begin{equation*} 
r_{k+1} \leq  \left(1 - \frac{1}{k}\right)r_k+ \frac{\alpha+1}{k^2(\alpha-1)}A,
\end{equation*}
which, by induction, leads to 
\begin{equation*}
r_{k+1} \leq A\left(\frac{\alpha+1}{\alpha-1}\right)\frac{1}{k}\sum_{j=1}^k\frac{1}{j}.
\end{equation*}
\textcolor{black}{Since for all $k \in \mathbb{N^*}, \sum_{j=1}^k 1/j \leq 1+\log(k)$, we get:
\begin{equation*}
r_{k+1} \leq A\left(\frac{\alpha+1}{\alpha-1}\right)\frac{1+\text{log}(k)}{k},
\end{equation*}}
which leads to \eqref{eq:thm5.2} {\color{black} by taking $k=N-1< \left(\frac{n}{\mathcal{W}(n)}\right)^{\frac{2(\beta-1)}{2\beta+r}}$} \textcolor{black}{since $N\leq N_{n,\beta}$ implies $N<1+\left(\frac{n}{\mathcal{W}(n)}\right)^{\frac{2(\beta-1)}{2\beta+r}}$.} 
%\sout{Using the same argument as in the previous case, the choice of this value of $k$ is legitimate.}

\subsection{Proof of Theorem \ref{thm:minimizer}}\label{section: thm minimizer proof} ~

%We start by bounding the quadratic error defined in \eqref{def e}.

%{\color{magenta}[Ici j'ai enlev\'e le lemme (Lemma4.1) et je l'ai integr\'e  directement \`a la proof]}

%\begin{lemma}\label{lemma: bound}
Let $\hat{W}$ be {\color{black}the} estimator of $W$ obtained using the method described in \textcolor{black}{Section \ref{section: method} after (\ref{def : Jhat})}. \textcolor{black}{Let $x^*$} be a minimizer of $f$ and {\color{black} $\hat x$} its estimator defined in \eqref{def xhat}. Observe that $z^* = Wx^*$.
{\color{black} Let  also $\hat{z}_W $ be the estimator of $z^*$ defined in \eqref{hatz_W}.}
\textcolor{black}{Then, we} can decompose the quadratic error as
\begin{align*}
\mathbb{E}\left(\lVert W (x^*-\hat{x})\rVert^2 \right) &= \mathbb{E}\left(\left \lVert W x^*-W\hat x \right \rVert^2 \mathds{1}_{\hat{W} = W}\right) +\mathbb{E}\left(\left \lVert W x^*-W\hat x\right \rVert^2 \mathds{1}_{\hat{W}\neq W}\right) \\
& \leq \mathbb{E}\left(\left \lVert W x^*-\hat z_{W}\right \rVert^2\right) +\vert \Theta \vert^2 \mathbb{P}(\hat{W} \neq W),
\end{align*}
\textcolor{black}{where the last inequality comes from the fact that $W\hat x = \hat z_{W}$, by (\ref{def xhat}), and from the fact that $x^*$ and $\hat x$ are in $\Theta$ and $\Vert W \Vert_{\text{op}} \leq 1$.}

According to the definition of $W$ and $\hat{W}$, the equality $\hat{W}=W$ is equivalent to $\hat{J}=J$. {\color{black} By Lemma \ref{lemma : J}}
\begin{align}\label{eq:maj_x_star_x_hat}
\mathbb{E}\left(\lVert W (x^*-\hat{x})\rVert^2 \right)  \leq \mathbb{E}\left(\left \lVert W x^*-\hat z_{W}\right \rVert^2\right)
+ \vert \Theta \vert^2  c_0\exp(c_0d)\exp(-c_1n h^{d+2}).
\end{align}
%\sout{Note that the inequality \eqref{eq: holder norme 2} also holds for the $\ell_1$-norm $\Vert \cdot \Vert_1$.}

{\color{black} Recalling that,} by definition, $Wx^* =z^*$ minimizes $g$ on $W\Theta$, the term 
\[
\mathbb{E}\left(\left \lVert W x^*- \hat{z}_W \right \rVert^2\right) = \mathbb{E}\left(\lVert z^* - \hat{z}_W\rVert^2\right)
\] 
corresponds to the expected quadratic error of the minimization of $g$ using the regression model \eqref{eq: reg g} and the dataset $(z_i,y_i)_{1\leq i\leq n}$, {\color{black} where by definition $z_i = Wx_i$.}

Since, by {\color{black}Assumption} \textcolor{black}{(\ref{eq:minoration_N})}, \textcolor{black}{$N> N_{n,\beta}$}, we can apply Theorem \ref{thm: tsybakov} to get 
\begin{equation*}
  \mathbb{E}\left(\lVert z^* - \hat{z}_W \rVert^2\right) \leq \frac{\textcolor{black}{N_{n,\beta}} -1}{N-1} {\color{black}\vert\Theta\vert^2}
  + \frac{N-\textcolor{black}{N_{n,\beta}}}{N-1}A\left(\frac{\alpha+1}{\alpha-1}\right)\left(\frac{\mathcal{W}(n)}{n}\right)^{\frac{2(\beta-1)}{2\beta +r}}
\end{equation*}
so that, \textcolor{black}{by (\ref{eq:maj_x_star_x_hat}), we have}
\begin{align*}
\mathbb{E}\left(\lVert W (x^*-\hat{x})\rVert^2 \right) &\leq \frac{N_{n,\beta} -1}{N-1} {\color{black}\vert\Theta\vert^2}+ \frac{N-N_{n,\beta}}{N-1}A\left(\frac{\alpha+1}{\alpha-1}\right)\left(\frac{\mathcal{W}(n)}{n}\right)^{\frac{2(\beta-1)}{2\beta +r}}\\ &+\vert \Theta \vert ^2 c_0\exp(c_0d)\exp(-c_1n h^{d+2}).
\end{align*}

Since, \textcolor{black}{by {\color{black}Assumption} (\ref{eq:minoration_N})}, $N-N_{n,\beta} \geq N_{n,\beta}\left(\frac{n}{\mathcal{W}(n)}\right)^{\frac{2(\beta-1)}{2\beta+r}}\textcolor{black}{\vert\Theta\vert^2}\left(\frac{\alpha-1}{\alpha+1}\right)$,
%\sout{and $n \geq \mathcal{W}(n)$ (see Eq. \eqref{eq:croissance W} of Lemma \ref{lemma: proporties lambert function})},
the second term in the right-hand side is greater than the first one \textcolor{black}{multiplied by $A$. We thus get}
\begin{align*}
  \mathbb{E}\left(\lVert  W (x^*-\hat{x})\rVert^2 \right) &\leq (A+1)\textcolor{black}{\left(\frac{N-N_{n,\beta}}{N-1}\right)}\left(\frac{\alpha+1}{\alpha-1}\right)\left(\frac{\mathcal{W}(n)}{n}\right)^{\frac{2(\beta-1)}{2\beta +r}}\\
                                                          &+\vert \Theta \vert^2 c_0\exp(c_0d)\exp(-c_1n h^{d+2})\\
                                                          &\leq (A+1)\left(\frac{\alpha+1}{\alpha-1}\right)
                                                            \left(\frac{\mathcal{W}(n)}{n}\right)^{\frac{2(\beta-1)}{2\beta +r}}
                                                            +\vert \Theta \vert^2 c_0\exp(c_0d)\exp(-c_1n h^{d+2}),
\end{align*}
\textcolor{black}{since $N_{n,\beta}\geq 1$ by (\ref{eq:croissance W}) {\color{black}of Lemma \ref{lemma: proporties lambert function}}.}
%For $n$ large enough, the exponential term is smaller than the other one. {\color{magenta} More precisely}, the inequality
\textcolor{black}{Using Lemma \ref{lemmaW} \textcolor{black}{with $\gamma=2(\beta-1)/(2\beta + r)$}, we have}
\begin{equation}\label{eq: terme 1 domine}
\left(\frac{\mathcal{W}(n)}{n}\right)^{\frac{2(\beta-1)}{2\beta + r}}
\geq
\exp\left(-c_1 h^{d+2} n\right)
\end{equation}
\textcolor{black}{which is satisfied for all $n\geq 3$ if $ c_1 h^{d+2} \geq e^{-1} $ and for $n\geq -\frac{\mathcal{W}_{-1}(-c_1 h^{d+2})}{c_1 h^{d+2}}$ if $ c_1 h^{d+2} < e^{-1}$.}
%({\color{magenta} cfr} Section \ref{section: proof of 34} {\color{magenta} for more details}).
{\color{black} Choosing} $B = (A+1)\left(\frac{\alpha+1}{\alpha-1}\right)+\vert \Theta \vert^2 c_0\exp(c_0d)$ concludes the proof.

\subsection{Proof of Theorem \ref{thm: real minimum}}\label{section: thm real minimum proof}~

{\color{black} Using that, by \eqref{def: estimation min}, $\widehat{f^*} = \hat{g}_{\hat W}(\hat{z}_{\hat{W}})$, we have}
\begin{equation}\label{eq: decomposition minimum}
\mathbb{E}\left(\left \lvert f(x^*)-\widehat{f^*}\right\rvert \right)\leq \mathbb{E}\left(\left \lvert f(x^*)-f(\hat{x})\right\rvert \right)+\mathbb{E}\left(\left \lvert f(\hat{x})-\hat{g}_{\hat W}(\hat{z}_{\hat{W}})\right\rvert \right).
\end{equation} 
%{\color{magenta}[Ici j'ai enlev\'e le lemme (Lemma4.2) et je l'ai integr\'e  directement \`a la proof]}

{\color{black} Observe that, by definition of $g$, $\vert f(x^*)- f(\hat{x})  \vert  = \vert g( Wx^*)- g(W\hat{x})\vert$.}
% {\color{magenta} We first consider the} right hand side of Eq. \eqref{eq: decomposition minimum}. {\color{magenta} According to/ By definition of $g$, we have that 
% \[
%  \vert f(x^*)- f(\hat{x})  \vert  = \vert g( Wx^*)- g(W\hat{x})\vert
% \]
% }
By the strong convexity assumption \eqref{ass: strong convexity} of \eqref{ass: regularity f} and the fact that $Wx^*{\color{black}=z^*}$ minimizes $g$, we get 
\begin{align*}
0 \leq g(W\hat{x})- g(Wx^*) 
\leq 
(W\hat{x} - Wx^*)^{\!\top} \nabla g(W\hat{x}) 
- \frac{\alpha}{2} \Vert Wx^*- W\hat{x} \Vert^2,
\end{align*}
which, using {\color{black} the} Cauchy-Schwarz inequality, leads to 
\begin{equation*}
 \vert g(W\hat{x})- g(Wx^*)  \vert 
 \leq 
 \Vert W\hat{x}-Wx^* \Vert \cdot \Vert \nabla g(W\hat{x}) \Vert  
 +\frac{\alpha}{2} \Vert Wx^*-W\hat{x} \Vert^2.
\end{equation*}
Since by \eqref{ass: min g}, $\nabla g(Wx^*) = 0$, the fact that $\Vert W \Vert_{\text{op}} \leq 1$ and for all $x \in \mathbb{R}^d$, $f(x) = f(W^\top Wx)$, $\nabla f(x) = \nabla f(W^\top Wx)$ and $\nabla ^2 f(x) = \nabla^2f(W^\top Wx)$, we have
\begin{align*}
{\color{black}\Vert \nabla g (W\hat{x}) \Vert }
&= \Vert \nabla g (W\hat{x}) - \nabla g(Wx^*) \Vert
= \Vert W\nabla f (\hat{x}) - W\nabla f(x^*) \Vert\\
&\leq \Vert \nabla f (W^\top W \hat{x}) - \nabla f(W^\top W x^*) \Vert\\
& = \left \Vert \int_0^1 \nabla^2f(W^\top W\hat{x}+tW^\top W(x^*- \hat{x}))(W^\top W x^*-W^\top W \hat{x})dt\right \Vert \\
& = \left \Vert \int_0^1 \nabla^2f(\hat{x}+t(x^*- \hat{x}))(W^\top W x^*-W^\top W \hat{x})dt\right \Vert \\
& \leq \int_0^1 \left \Vert \nabla^2f(\hat{x}+t(x^*- \hat{x}))\right \Vert_{\text{op}}\Vert W^\top W(\hat{x}-x^*) \Vert dt 
 \leq \int_0^1 C\Vert W^\top W(\hat{x}-x^*) \Vert dt,
\end{align*}
 where the last inequality {\color{black} comes from \eqref{ass: hessienne} of \eqref{ass: regularity f} and the fact that $\hat{x}+t(x^*- \hat{x})$ is in $\Theta$ since $x^*$ and $\hat{x}$ both belong to $\Theta$ which is convex.} % $x^*$ and $\hat{x}$ belong to $\Theta$ and, since $\Vert \nabla^2 f\Vert_{\text{op}}$ is uniformly bounded {\color{magenta} on $\Theta$ by \eqref{ass: hessienne} of \eqref{ass: regularity f}}, \sout{there exists a positive constant $C$ such that} $\forall t \in [0,1],\left\Vert \nabla^2f(\hat{x}+t(x^*-\hat{x}))\right \Vert_{\text{op}} \leq C$. {\color{magenta} [$C$ est  la constante de \eqref{ass: regularity f} non ?]}
\textcolor{black}{Thus, using that $\Vert W^\top W(x^* - \hat{x})\Vert^2=\Vert W(x^* - \hat{x})\Vert^2$, we obtain that 
\begin{align*}
 \vert f(x^*)- f(\hat{x})  \vert  = \vert g( Wx^*)- g(W\hat{x})\vert 
  \leq \left(C+ \frac{\alpha}{2}\right) \Vert W(x^* - \hat{x})\Vert^2.
 \end{align*} }
{\color{black} Hence, by Theorem \ref{thm:minimizer}, we get} 
\begin{equation}\label{eq-lemma4.2}
\mathbb{E}\left(\vert f(x^*)-f(\hat{x}))\rvert \right)\leq \left(C+\frac{\alpha}{2}\right)B \left(\frac{\mathcal{W}(n)}{n}\right)^{\frac{2(\beta-1)}{2\beta+r}},
\end{equation}
\textcolor{black}{where $B$ is a positive constant.}
%\vskip20mm
%The next Lemma bounds the first term of the right hand side using the fact that by \eqref{ass: strong convexity} of \eqref{ass: regularity f} the function $g$ is $\alpha$-strongly convex.
%\begin{lemma}\label{thm:minimum}
%
%Assume that \eqref{ass: assumption K} -- \eqref{ass: distinguishable properties} hold and that the $y_i$'s satisfy Model \eqref{eq: regression model} such that the $\xi_i$'s are Gaussians. The estimator $\hat{x}$  of a minimizer of $f$ defined in \eqref{def xhat}  after a number of iterations 
%$N \geq \left(\frac{n}{\mathcal{W}(n)}\right)^{\frac{4(\beta-1)}{2\beta+r}}\left(\vert\Theta\vert\left(\frac{\alpha-1}{\alpha+1}\right)+1\right)+1$ and $n \geq -\frac{\mathcal{W}_{-1}(-c_1h^{d+2})}{c_1h^{d+2}}\mathds{1}_{\{ch^{d+2}< e^{-1}\}}$, $c_1$ and $h$ being positive constants introduced in \eqref{eq: proba}, satisfies
%\[
%\mathbb{E}\left(\vert f(x^*)-f(\hat{x}))\rvert \right)\leq \left(C+\frac{\alpha}{2}\right)B \left(\frac{\mathcal{W}(n)}{n}\right)^{\frac{2(\beta-1)}{2\beta+r}},
%\]
%where $C$ and $\alpha$ are postive constants defined in \eqref{ass: regularity f}, $B$ is the constant appearing in  Theorem \ref{thm:minimizer} and the values of $h_n$ and $\lambda_n$ are defined in \eqref{def: values hn et lambdan}.
% \end{lemma}

\textcolor{black}{Let us now consider} the second term of the right-hand side of \eqref{eq: decomposition minimum}.
\textcolor{black}{Since $f$ and $\hat g_{\hat W}$ are continuous, they are bounded on $\Theta$ and $\hat W\Theta$, respectively. Thus, there exists a positive
  constant $c$ such that
\begin{align*}
\mathbb{E}\left(\vert f(\hat{x}) - \hat g_{\hat W}(\hat z_{\hat W}) \vert \right) 
  &= \mathbb{E}\left( \left \vert f(\hat x) - \hat g_{\hat W}(\hat z_{\hat W})\; \right \vert \mathds{1}_{W=\hat W} \right) + \mathbb{E}\left( \left \vert f(\hat x) - \hat g_{\hat W} (\hat z_{\hat W})\right \vert \mathds{1}_{W \neq \hat W} \right)\\
 &\leq \mathbb{E}\left( \left \vert g(W\hat x) - \hat g(\hat z_W) \right \vert \mathds{1}_{W=\hat W}\right)
   + c\mathbb{P}(W \neq \hat W)\\
&\leq \mathbb{E}\left( \left \vert g(\hat z_W) - \hat g(\hat z_W) \right \vert^2 \right)^{\frac{1}{2}}
+ c\mathbb{P}(W \neq \hat W),
\end{align*}
where the last inequality comes from the Cauchy-Schwarz inequality. Thus, we get
$$
\mathbb{E}\left(\vert f(\hat{x}) - \hat g_{\hat W}(\hat z_{\hat W}) \vert \right)
\leq \mathbb{E}\left[\sup_{z \in W\Theta} | \hat{g}(z) - g(z) |^2 \right] ^{\frac{1}{2}}
+ c\mathbb{P}(W \neq \hat W).
$$
}

%Using {\color{black}the} Cauchy-Schwarz inequality and the fact that $f$ and $\hat g_{\hat W}$ are respectively bounded on $\Theta$ and $\hat W\Theta$, we have, for a constant $c>0$: \textcolor{black}{comment peut-on être sûr que $\hat g_{\hat W}$ est bornée ?}
% \begin{align*}
% \mathbb{E}\left(\vert f(\hat{x}) - \hat g_{\hat W}(\hat z_{\hat W}) \vert \right) 
% &= \mathbb{E}\left( \left \vert f(\hat x) - \hat g_{\hat W}(\hat z_{\hat W})\; \right \vert \mathds{1}_{W=\hat W} \right) + \mathbb{E}\left( \left \vert f(\hat x) - \hat g_{\hat W} (\hat z_{\hat W})\right \vert \mathds{1}_{W \neq \hat W} \right).
% \end{align*}
% {\color{magenta} Recall that by definition of $f(\hat x) = g(W \hat x) = g(\hat z_W)$ and that, when $W=\hat W$, the estimator $\hat g_{\hat W}$ of $g$ is equal to $\hat g$, defined in Eq. \eqref{def: estimation min}.} Then, 
% \begin{align*}
% \mathbb{E}\left(\vert f(\hat{x}) - \hat g_{\hat W}(\hat z_{\hat W}) \vert \right) &= \mathbb{E}\left( \left \vert g(W\hat x) - \hat g(\hat z_W) \right \vert \mathds{1}_{W=\hat W}\right)
% + c\mathbb{P}(W \neq \hat W)\\
% &= \mathbb{E}\left( \left \vert g(\hat z_W) - \hat g(\hat z_W) \right \vert \right)
% + c\mathbb{P}(W \neq \hat W)\\
% & \leq \mathbb{E}\left( \left \vert g(\hat z_W) - \hat g(\hat z_W) \right \vert^2 \right)^{\frac{1}{2}}
% + c\mathbb{P}(W \neq \hat W)\\
% & \leq \mathbb{E}\left[\sup_{z \in W\Theta} | \hat{g}(z) - g(z) |^2 \right] ^{\frac{1}{2}}
% + c\mathbb{P}(W \neq \hat W).
% \end{align*}
%where we used $W\hat x = \hat z_{\hat W}$ by definition. 

According to {\color{black} the} definition of the estimator $\hat{W}$, the equality $\hat{W}=W$ is equivalent to $\hat{J}=J$.
%\textcolor{black}{A QUEL ENDROIT A-T-ON BESOIN DE CELA ?}
%Note that the {\color{magenta}definition of $\beta$-H\"older functions in} \eqref{eq: holder norme 2} also holds for the $\ell_1$-norm $\Vert \cdot \Vert_1$. {\color{magenta}[on en a besoin?]}
\textcolor{black}{Thus, by Lemma \ref{lemma : J}, we have}
\begin{align*}
\mathbb{E}\left(\vert f(\hat{x}) - \hat g_{\hat W}(\hat z_{\color{black}\hat W}) \vert \right)  \leq \mathbb{E}\left[\sup_{z \in W\Theta} | \hat{g}(z) - g(z) |^2 \right] ^{\frac{1}{2}}
+ c c_0\exp(c_0d)\exp(-c_1n h^{d+2}),
\end{align*}
%textcolor{black}{attention il n'y a pas de diametre de theta au carré pour le 2ème terme.}

%for $h$ the bandwidth used in \eqref{def: thetahat} and both $c_0$ and $c_1$ being positive constants.

% {\color{magenta} Remark that, by definition, $h_n$ and $\lambda_n$ in Eq. \eqref{def: values hn et lambdan} satisfy the assumptions of Lemma \ref{lemma 11bis} and Lemma \ref{lemma 12bis} [besoin de le dire?].
% Hence, using Lemma \ref{lemma 11bis} and Lemma \ref{lemma 12bis} and the fact that  $\mathcal{W}(n)\geq 1$ for $n\geq 3$. 
% we get}

\textcolor{black}{According to the definition of $h_n$ and $\lambda_n$ in \eqref{def: values hn et lambdan}, we have that $nh_n^r= \lambda_n^{-2}\mathcal{W}(n)$. Moreover, by (\ref{eq:croissance W}) {\color{black}of Lemma \ref{lemma: proporties lambert function}}, $nh_n^r \geq\mathcal{W}(n)$. Thus, $nh_n^r \geq \max(\mathcal{W}(n),\lambda_n^{-2}\mathcal{W}(n))$ and hence Lemmas \ref{lemma 11bis} and \ref{lemma 12bis} can be used to give}
\begin{align*}
\mathbb{E}\left[\sup_{z \in W\Theta} | \hat{g}(z) - g(z) |^2 \right] 
&\leq \mathbb{E}\left[ 2\sup_{z \in W\Theta} | \hat{g}(z) - \mathbb{E}\hat{g}(z)|^2 + 2\sup_{z \in W\Theta} | \mathbb{E}\hat{g}(z) - g(z) |^2 \right]\\
&\leq 2A_2 h_n^{-r} n^{-1} \mathcal{W}(n) + 2A_1^2\left(h_n^{\beta}+\lambda_n+h_n^{-\frac{r}{2}}n^{-\frac{1}{2}}\right)^2 \\
&\leq 2A_2 h_n^{-r} n^{-1} \mathcal{W}(n) + \textcolor{black}{6A_1^2}\left(h_n^{2\beta}+\lambda_n^2+h_n^{-r}n^{-1}\right) \\
&\leq A\left(h_n^{-r} n^{-1} \mathcal{W}(n) + h_n^{2\beta}+\lambda_n^2\right),
\end{align*}
for $A=2A_2+\textcolor{black}{6A_1^2}>0$, \textcolor{black}{where we used (\ref{eq:minoration_W}) {\color{black}of Lemma \ref{lemma: proporties lambert function}} for obtaining the last inequality.} 
{\color{black} By replacing the values of $h_n$ and $\lambda_n$ given in  \eqref{def: values hn et lambdan} we get}
%\begin{equation}\label{eq:lemma13.1}
\[
\mathbb{E} \left[\sup_{z \in W\Theta} | \hat{g}(z) - g(z) |^2 \right]  
\leq A \left(\frac{\mathcal{W}(n)}{n}\right)^{\frac{2\beta}{2\beta+r}}.
\]
%\end{equation}
{\color{black} Hence, the second term of \eqref{eq: decomposition minimum} is bounded by}
$$
\mathbb{E}\left(\vert f(\hat{x}) - \hat g_{\hat W}(\hat z_{\color{black}\hat W}) \vert \right)  \leq \sqrt{A} \left(\frac{\mathcal{W}(n)}{n}\right)^{\frac{\beta}{2\beta+r}}
+ \textcolor{black}{c  c_0\exp(c_0d)\exp(-c_1n h^{d+2})}.
$$
%Note that for $n$ large enough, the exponential term is smaller than the other one. More precisely, the inequality
\textcolor{black}{By Lemma \ref{lemmaW} with $\gamma=\beta/(2\beta+r)$, we have that}
\begin{equation}\label{eq: terme 1 domine bis}
\exp\left(-c_1 h^{d+2} n\right) \leq \left(\frac{\mathcal{W}(n)}{n}\right)^{\frac{\beta}{2\beta + r}}
\end{equation}
for all $n\geq3$ if $ c_1 h^{d+2} \geq e^{-1} $ and for $n\geq -\frac{\mathcal{W}_{-1}(-c_1 h^{d+2})}{c_1 h^{d+2}}$ if $ c_1 h^{d+2} < e^{-1}$.
Thus, \eqref{eq: terme 1 domine bis} holds for $n\geq \max\left(-\frac{\mathcal{W}_{-1}(-c_1 h^{d+2})}{c_1 h^{d+2}}\mathds{1}_{\{c_1 h^{d+2} < e^{-1}\}},3\right)$
and we have
\begin{equation}\label{eq: condition n}
\mathbb{E}\left(\vert f(\hat{x}) - \hat g_{\hat W}(\hat z_{\color{black}\hat W}) \vert \right) \leq B''\left(\frac{\mathcal{W}(n)}{n}\right)^{\frac{\beta}{2\beta +r}},
\end{equation}
where \textcolor{black}{$B'' = \sqrt{A}+cc_0\exp(c_0d)$} is a positive constant.
\textcolor{black}{Combining \eqref{eq: condition n} with \eqref{eq-lemma4.2} for} $n\geq \max\left(-\frac{\mathcal{W}_{-1}(-c_1 h^{d+2})}{c_1 h^{d+2}}\mathds{1}_{\{c_1 h^{d+2} < e^{-1}\}},3\right)$, \eqref{eq: decomposition minimum} becomes
\begin{align*}
\mathbb{E}\left(\left \lvert f(x^*)-\widehat{f^*}\right\rvert \right) &\leq B''\left(\frac{\mathcal{W}(n)}{n}\right)^{\frac{\beta}{2\beta +r}}+\left(C+\frac{\alpha}{2}\right)B\left(\frac{\mathcal{W}(n)}{n}\right)^{\frac{2(\beta-1)}{2\beta+r}}
\leq B'\left(\frac{\mathcal{W}(n)}{n}\right)^{\frac{\beta}{2\beta +r}},
\end{align*} 
\textcolor{black}{where $B' = 2\max\left(B'', \left(C+\frac{\alpha}{2}\right)B\right)$ and the last inequality comes from (\ref{eq:croissance W}) {\color{black}of Lemma \ref{lemma: proporties lambert function}} and the fact that $\beta \geq 2$.}
%{\color{magenta}[peut etre rappeler que $\mathcal W(n) \leq n$ pour $n\geq 0$?]}

\section{Technical lemmas}\label{section: technical lemmas}

\begin{lemma}\label{lemma : J}{\rm[Theorem 1, \cite{J}]}
{\color{black} Assume that  \eqref{ass: regularity f},  \eqref{ass: density} and \eqref{ass: distinguishable properties} hold and let $\xi_1,\dots, \xi_n$ be Gaussian random variables. Then,}
%Suppose that $f$ satisfies the regularity assumption \eqref{ass: regularity f} and the distinguishability assumption \eqref{ass: distinguishable properties}. Assume also that the $x_i$'s satisfy the density assumption \eqref{ass: density} and that the $\xi_i$'s are Gaussian random variables, then 
the subset $\hat{J} \subset \{1,...,d\}$ defined in \eqref{def : Jhat} with a bandwidth $h$ and a regularization parameter $\lambda$ such that
\begin{equation*}
0<h<\min\left(\frac{p_{m}}{32({\color{black}d}+1)M_K\tilde L};\eta\right) \quad \text{and}\quad \lambda = 8\sqrt{3 M_K}p_{M}Lh
\end{equation*}
satisfies, with probability {\color{black}larger} than $$1-c_0\exp(c_0{\color{black}d})\exp(-c_1nh^{{\color{black}d}+2}),$$ the equality $$\hat{J}=J,$$
\textcolor{black}{where $M_K$, $c_0$ and $c_1$ are positive constants.}
\end{lemma}
%{\color{magenta}[est ce normal d'avoir $r$ et $d$ ?]}\textcolor{black}{je suis d'accord : pour moi c'est $d$ partout.}
\begin{proof}
  {\color{black} Assumption 2 of Theorem 1 of \cite{J} is satisfied for the uniform kernel $K^u$ on $B_\infty(0,1)$ {\color{black} defined in \eqref{Ku}} with $M_K=21$. Assumption 5 of this theorem is also valid thanks to \eqref{ass: regularity f} since $\|x-x'\|^{\beta}\leq {\|x-x'\|_1}^{\beta}$. Hence, all the assumptions of Theorem 1 of \cite{J} are satisfied, which concludes the proof.}
  %In this case, the constant  $M_K$ {\color{black} is replaced} by 1. \textcolor{black}{A VERIFIER}
%This Lemma is Theorem 1 of \citep{J} with the uniform kernel $K^u$ on $B_\infty(0,1)$ replacing the constant $M_K$ by 1, which satifies the needed assumption.
\end{proof}

\begin{lemma}\label{lemma : g holderienne}
{\color{black} Let $f$ be a function satisfying the $\beta$-H\"older condition  \eqref{ass: holder} of \eqref{ass: regularity f}.} Then the function $g$ defined in \eqref{def: g} is a $\beta$-H\"older function of $\mathbb{R}^r$.
%The $\beta$-Hölder assumption \eqref{ass: holder} of \eqref{ass: regularity f} on $f$ appearing in Eq. \eqref{eq: regression model} implies that $g$ defined in \eqref{def: g} is a $\beta$-Hölder function of $\mathbb{R}^r$
\end{lemma}

\begin{proof}
Let $z$ and $z'$ in $\mathbb{R}^{\color{black}r}$. {\color{black} We want to prove  that,}
$$ \left| 
g(z) - \sum_{\substack{{\color{black}\m\in\mathbb N^r}\\ |\m| \le \ell}} \frac{1}{\m!} D^{\m} g(z')\,{\color{black}(z - z')}^{\m} 
\right|
\le L \|z - z'\|^\beta,$$
where $L$ is the constant in \eqref{ass: holder} of \eqref{ass: regularity f}. {\color{black} Since, $g(z) = f(W^\top z)$, for any $z\in \mathbb R^r$, we have, by definition of $W$, that
\[
g(z) - \sum_{\substack{\m\in\mathbb N^r\\ |\m| \le \ell}} \frac{1}{\m!} D^{\m} g(z')\,{\color{black}(z - z')}^{\m} 
 = f(W^\top z) - \sum_{\substack{{\m\in\mathbb N^r}\\ |\m| \le \ell}} \frac{1}{\m!} D^{\m} f(W^\top z')\,(W^\top {\color{black}z} - W^\top {\color{black}z'})^{\m} .
\]
{\color{black} Since $f$ is $\beta$-Hölder, by (\ref{eq: holder norme 2})}, 
\[
 \left|  f(W^\top z) - \sum_{\substack{\m\in\mathbb N^d\\ |\m| \le \ell}} \frac{1}{m!} D^m f(W^\top z')\,(W^\top {\color{black}z} - W^\top {\color{black}z'})^m  \right|
{\leq L \|W^\top(z-z')\|^\beta }
\le L \|z - z'\|^\beta,
\]
where the last inequality holds since $\Vert W \Vert_{\text{op}} \leq 1$.
}
To conclude the proof, it is enough to show that 
\begin{equation}\label{to show 5.2}
 \sum_{\substack{{\color{black}\m\in\mathbb N^d}\\ |\m| \le \ell}} \frac{1}{\m!} D^{\m} f(W^\top z')\,(W^\top {\color{black}z} - W^\top {\color{black}z'})^{\m} =  \sum_{\substack{{\color{black}\m\in\mathbb N^r}\\ |\m| \le \ell}} \frac{1}{\m!} D^{\m} f(W^\top z')\,(W^\top {\color{black}z} - W^\top {\color{black}z'})^{\m}.
\end{equation}
\textcolor{black}{Since, for $\m=(m_1,\dots,m_d)$ such that $m_j \neq 0$ for $j \notin J$, $D^{m} f(W^\top z')$ is the null function, the sum in the l.h.s of (\ref{to show 5.2}) can be limited to {\color{black}the} $\m$ such that $m_j=0$ for $j\notin J$. Then, for those $\m$ of $\mathbb{N}^d$, {\color{black} we have that} $D^m f(W^\top z') = D^{m_{\mid J}}f(W^\top z')$, $ {\color{black}\left( m_{\mid J}\right)! }= m!$ using the convention $0^0=1$ and $(W^\top {\color{black}z} - W^\top {\color{black}z'})^m=(W^\top {\color{black}z} - W^\top {\color{black}z'})^{m_{\mid J}}$, where $m_{{\mid J}}=(m_j)_{j\in J}$, which concludes the proof.}

\end{proof}

\begin{lemma}\label{lemma: formule densité}
If {\color{black} $x_1,\dots, x_n$} satisfy \eqref{ass: density}, then $z_i=Wx_i $ for $i \in \{1,...,n\}$ are {\color{black} i.i.d.} with a density $\tilde{p}$ such that there \textcolor{black}{exists} {\color{black}a positive constant} {\color{black} $\tilde{p}_{m}$} satisfying 
\begin{equation*}
\forall z \in W\Theta', \quad 0<\tilde{p}_{m}\leq \tilde{p}(z) \leq 1,
\end{equation*}
where $W\Theta' = \{Wx; x \in \Theta' \}.$ %\textcolor{black}{je ne comprends pas pourquoi on parle de $\Theta$' on avait dit qu'on ne le mettait plus : A REPRENDRE.}

%\textcolor{black}{remplacer $\tilde{p}_{M}$ par 1 et vérifier que ça ne change rien ailleurs.}

\end{lemma}

\begin{proof}
{\color{black} Note that} $\mathbb{R}^d = \ker(W) \oplus \ker(W)^\perp$, where $\oplus$ {\color{black} denotes the direct sum, \textit{i.e.}} for all $x \in \mathbb{R}^d$ there exists a unique decomposition $x = x_1+x_2$ such that $x_1 \in \ker(W), x_2 \in \ker(W)^\perp$.
Hence, {\color{black} for $X$ a random vector having a probability density function $p$ and {\color{black} for} any bounded measurable function $\varphi : \mathbb{R}^r \to \mathbb{R}$, we have}
\begin{align*}
\mathbb{E}(\varphi(WX))
&=
\int_{\mathbb{R}^d} \varphi(Wx)\, p(x)\, dx \\
&=
\int_{\ker(W)^\perp}\int_{\ker(W)}
\varphi\bigl(W(x_1+x_2)\bigr)\,
p(x_1+x_2)\, dx_1\, dx_2 \\
&=
\int_{\ker(W)^\perp}
\varphi(Wx_2)
\left(
\int_{\ker(W)}
p(x_1+x_2)\, dx_1
\right)\, dx_2 .
\end{align*}
Using the change of variables $x_2 = W^\top z$ we get
\begin{align*}
\mathbb{E}(\varphi(WX))
&=
\int_{\mathbb{R}^r}
\varphi(z)
\left(
\int_{\ker(W)}
p(x_1 + W^\top z)\, dx_1
\right)\, dz .
\end{align*}
\textcolor{black}{Thus, the probability density function of \(Z=WX\) is}
\begin{equation}\label{eq:density z}
\tilde p(z)
=
\int_{\ker(W)}
p(x_1 + W^\top z)\, dx_1,
\qquad
\forall z \in \mathbb{R}^r .
\end{equation}
\textcolor{black}{Since $p$ is a probability density function
\begin{align*}
\tilde p(z)\le
  \int_{\mathbb{R}^d} p(y)\, dy =1.
\end{align*}}
{\color{black} Let $z\in W\Theta'$. Thus, there exists $x\in \Theta'$ such that $z = Wx$, hence \eqref{eq:density z} becomes}
\begin{align*}
\tilde p(z) =
\int_{\ker(W)}
p(x_1 + W^\top W x)\, dx_1 .
\end{align*}
\textcolor{black}{Since $WW^\top  = I_r$, $ x - W^\top W x \in \ker(W)$ and thus using the change of variables $u=x_1-(x- W^\top W x)$ gives}
% {\color{magenta}Since $WW^\top  = I_r$, we have that} $ x - W^\top W x \in \ker(W).$ {\color{magenta}Hence,}
% we perform the change of variable $ x_1 = u + x - W^\top W x$ {\color{magenta} to get}
%Thus,
\begin{align*}
\tilde p(z)
=
\int_{\ker(W)} p(u+x)\, du.
\end{align*}
\textcolor{black}{Using (\ref{ass: density}) and the fact that $\textrm{dim(ker($W$))}=d-r$, we get
\begin{align*}
\tilde p(z)
\ge
\int_{\ker(W)\cap B_\infty(0,\eta)}
p(u+x)\, du
\ge
p_{m}\int_{\ker(W)\cap B_\infty(0,\eta)} \, du\,
,
\end{align*}
\textcolor{black}{where 
%$\ker(W)\cap B_\infty(0,\eta)=\{(u_{j_k})_{1\leq k\leq d}, u_{j_k}=0,\forall k\in\{1,\dots,r\}\textrm{ and }|u_{j_\ell}|\leq\eta,  \forall \ell\in \{r+1,\dots,d\}\}$
\[
\ker(W)\cap B_\infty(0,\eta)=\big\{(u_{j_k})_{1\leq k\leq d}, \, u_{j_k}=0,\forall k\in\{1,\dots,r\}\textrm{ and }|u_{j_\ell}|\leq\eta,  \forall \ell\in \{r+1,\dots,d\}\big\}
\]
}}
Hence, we get 
\begin{equation*}
\tilde p(z) \ge
p_{m}(2{\color{black}\eta})^{d-r} >0,
\end{equation*}
\textcolor{black}{which concludes the proof if we take $\tilde{p}_ m = p_{m}(2\eta)^{d-r}$.}

% \textcolor{black}{REPRENDRE LA FIN}

% \sout{Let \(z = Wx \in W\Theta'\) (where \(x \in \Theta'\)). }
% {\color{magenta}Let $z\in W\Theta'$. There exists $x\in \Theta'$ such that $z = Wx$.} Then, {\color{magenta} for such a $z$ (?)}
% \begin{align*}
% \tilde p(z) =
% \int_{\ker(W)}
% p(x_1 + W^\top z)\, dx_1 =
% \int_{\ker(W)}
% p(x_1 + W^\top W x)\, dx_1 .
% \end{align*}
% {\color{magenta}Since $WW^T = I_r$, we have that} $ x - W^\top W x \in \ker(W).$ {\color{magenta}Hence,}
% we perform the change of variable $ x_1 = u + x - W^\top W x$ {\color{magenta} to get}
% %Thus,
% \begin{align*}
% \tilde p(z)
% =
% \int_{\ker(W)} p(u+x)\, du
% \ge
% \int_{\ker(W)\cap B_\infty(x,{\color{magenta}\delta})}
% p(u+x)\, du
% \ge
% p_{m}\,
% \lambda\!\left(\ker(W)\cap B_\infty(x,{\color{magenta}\delta})\right),
% \end{align*}
% {\color{magenta}[$\delta$ ou $\eta$? ce n'est pas le $\eta$ de (A4)?]} where $p_m$ {\color{magenta} $\eta$ are the constants} \sout{is a constant} defined in \eqref{ass: density} and $\lambda(\cdot)$ is the Lebesgue measure in $\mathbb{R}^d$.
% Hence, {\color{magenta} since $\mathrm{dim}(W)=r,$}
% \begin{equation*}
% \tilde p(z) \ge
% p_{m}(2{\color{magenta}\delta})^{d-r} >0
% \end{equation*}
% %Defining $\tilde{p}_ m = p_{m}(2{\color{magenta}\delta})^{d-r}$ and $\tilde p_M = 1$ concludes the proof.
\end{proof}

%\section{Appendix}

%\subsection{\color{magenta}Properties of the Lambert functions}~

{\color{black}In the following we prove some useful results on the properties of the Lambert functions $\mathcal W$ and $\mathcal W_{-1}$, defined in  \eqref{def: lambert functions}.}

\begin{lemma}\label{lemma: proporties lambert function}
{\color{black}Let $\mathcal W$ be the Lambert function defined in \eqref{def: lambert functions}.}
 The following inequalities hold
%For $\mathcal{W}$ the Lambert function introduced in \eqref{def: lambert functions}, we have
\begin{equation}\label{eq:croissance W}
\forall x \geq 0, \hspace{1.5ex}
x \geq \mathcal{W}(x),
\end{equation}
\begin{equation}\label{eq:minoration_W}
\forall x\geq e, \hspace{1.5ex} \mathcal{W}(x)\geq 1,
\end{equation}
\begin{equation}\label{eq:majoration_W_log}
\forall x\geq e, \hspace{1.5ex} \mathcal{W}(x)\leq\log(x),
\end{equation}
\begin{equation}\label{prop log}
\forall \kappa \geq 1, \hspace{1.5ex} \forall x \geq 0, \hspace{1.5ex} \mathcal{W}(\kappa x) \leq \kappa \mathcal{W}(x),
\end{equation}
\begin{equation}\label{eq:on W}
\textcolor{black}{\forall q \in \mathbb{N}^* \textrm{ and } n\geq 3,} \hspace{1.5ex} 5 \mathcal{W}(n) \geq \mathcal{W}\left(n^{2+\frac{2}{q}}\right).
\end{equation}
\end{lemma}

\begin{proof}
{\color{black}Since the function} $x \mapsto xe^x$ {\color{black} is increasing} on $\mathbb{R}_+$, \textcolor{black}{it is also the case of its inverse $\mathcal{W}$ hence} we have, for $x \geq 0$, 
$$ x \geq \mathcal{W}(x) \Longleftrightarrow x e^x \geq x,$$ which, using {\color{black} the fact that} $e^x \geq {\color{black}1}$ for $x\geq0$ proves \eqref{eq:croissance W}.
\textcolor{black}{In the same way, to prove \eqref{eq:minoration_W}, we use that
  $$ \mathcal{W}(x) \geq 1 \Longleftrightarrow x \geq e.$$}
\textcolor{black}{To prove \eqref{eq:majoration_W_log}, we use that
  $$\mathcal{W}(x)\leq\log(x) \Longleftrightarrow x\leq x\log(x),$$
which is valid for all $x\geq e$.}
{\color{black}To prove \eqref{prop log} we use} the same properties {\color{black} and} have {\color{black} that,} for $x \geq 0$ and $\kappa \geq 1$,
\begin{align*}
\mathcal{W}(\kappa x) \leq \kappa \mathcal{W}(x) & \iff \kappa x \leq \kappa \mathcal{W}(x)e^{\kappa \mathcal{W}(x)} \\
& \iff x \leq \mathcal{W}(x)e^{\kappa \mathcal{W}(x)},
\end{align*}
which is always satisfied for $\kappa \geq 1$, \textcolor{black}{since $\mathcal{W}$ is non negative on $\mathbb{R}_+$ and, by definition of $\mathcal{W}$,}
$x=\mathcal W(x) e^{\mathcal W(x)}$.

% Moreover, for $x\geq 0, \gamma > 0$ and $a \geq 0$:
% \begin{align*}
% \gamma \mathcal{W}(x) \geq \mathcal W(x^a) &\iff
% \gamma \mathcal{W}(x)e^{\gamma \mathcal{W}(x)} \geq x^a\\
% &\iff \gamma \mathcal{W}(x)^{1-\gamma}\left(\mathcal{W}(x)e^{\mathcal{W}(x)}\right)^{\gamma} \geq x^a\\
% &\iff \gamma \mathcal{W}(x)^{1-\gamma}x^\gamma \geq x^a\\
% &\iff \gamma \geq \left(\frac{\mathcal{W}(x)}{x}\right)^{\gamma-1}x^{a-1}.
% \end{align*}
% Hence, we have
% \begin{align*}
% \gamma \mathcal{W}(n) \geq \mathcal{W}\left(n^{2+\frac{2}{q}}\right) 
% &\iff \gamma \geq \left(\frac{\mathcal{W}(n)}{n}\right)^{\gamma-1}n^{2+\frac{2}{q}-1}.
% \end{align*}
% {\color{magenta}[ok jusqu'ici, mais je ne comprends pas la fin - voir commentaire plus bas]}
% The worst case being for $q=1$, $\gamma \geq \mathcal{W}(n)^{\gamma-1}n^{4-\gamma}$ is a sufficient condition for the wanted inequality to hold for all $q\in \mathbb{N}^*$. We note that with $\gamma = 5$ this condition is satisfied for all $n \in \mathbb{N}^*$, which proves \eqref{eq:on W}.

{\color{black} Let us now prove (\ref{eq:on W}). Since $\mathcal{W}$ is the inverse of $x\mapsto {\color{black}xe^x}$, applying this function to (\ref{eq:on W}) leads to
\[n^{2+\frac{2}{q}} \leq 5 \mathcal W(n) e^{5\mathcal W(n)}.
\]
Observing that the r.h.s is equal to $5(\mathcal{W}(n)e^{\mathcal{W}(n)})^5\mathcal{W}(n)^{-4}=5n^5\mathcal{W}(n)^{-4}$, proving (\ref{eq:on W}) is equivalent to proving that $\mathcal{W}(n)^4\leq 5n^{3-\frac{2}{q}}.$ Hence, {\color{black} since $1\leq 3-2/q$,} it is enough to prove that $\mathcal{W}(n)^4\leq 5n$.
  % [\`a mon avis ici il faut expliquer pourquoi $q$ disparait]}
Since for all $x\geq e$, $\mathcal{W}(x)\leq \log(x)$, it is thus enough to prove that, {\color{black}for $n\geq 3$,} $\log(n)^4\leq 5n$  which is true.
% \begin{align*}
% 5\mathcal W(n) \geq \mathcal W(n^{2+\frac{2}{q}}) & \iff n^{2+\frac{2}{q}} \leq 5 \mathcal W(n) e^{5\mathcal W(n)}\\
% & \iff n^{2+\frac{2}{q}} \leq 5 n^5 \mathcal W(n)^{-4}\\
% & \iff \mathcal W(n)^4 \leq 5 n^{3-\frac{2}{q}}
% \end{align*}
% ensuite il faut utiliser  \eqref{eq:croissance W} : $\mathcal W(n) \leq n$ ?
}
\end{proof}

\begin{lemma}\label{lemmaW} {\color{black} Let $n\geq3$, \textcolor{black}{$\gamma \in ]0,1]$} and let $\mathcal W $ and $\mathcal{W}_{-1}$ be the Lambert functions defined in \eqref{def: lambert functions}.} {\color{black} Using the notations of Lemma \ref{lemma : J}, let} $c= c_1 h^{d+2}$. \textcolor{black}{Then, the {\color{black}following} inequality}
\[
\left(\frac{\mathcal{W}(n)}{n}\right)^{\gamma}
\geq
\exp\left(-c n\right)
\]
\textcolor{black}{is always satisfied} if $c \geq 1/e$, {\color{black} whereas, if $0<c <1/e$, it only holds for $n\geq -\frac{\mathcal{W}_{-1}(-c)}{c}$.}
\end{lemma}

\begin{proof}
  {\color{black} Since $\gamma\in ]0,1]$, \textcolor{black}{using (\ref{eq:croissance W}) {\color{black} of Lemma \ref{lemma: proporties lambert function}}, it is enough to prove that $$
      \frac{\mathcal{W}(n)}{n}\geq
\exp\left(-c n\right).$$} Moreover, since by (\ref{eq:minoration_W}), $\mathcal{W}(n) \geq1$, for all $n\geq 3,$
it is enough to show that 
\[
n^{-1}
\geq
\exp\left(-c n\right)
\]
or equivalently, \textcolor{black}{since $c$ is positive,} 
\begin{equation}\label{eq_cn}
-c \leq-cn\exp\left(-c n\right).
\end{equation}
\textcolor{black}{Using that, for all $x\in\mathbb{R}$, $x\exp(x)\geq -1/e$, we get that, if $c\geq 1/e$, \eqref{eq_cn} is always satisfied.} On the other hand, if $0<c< 1/e$, for $n > 1/c$, \eqref{eq_cn} is equivalent to 
\[
\mathcal{W}_{-1}(-c) \geq -cn,
\]
where we used that the Lambert function $\mathcal{W}_{-1}$ is non-increasing, which concludes the proof since
  $n\geq -\mathcal{W}_{-1}(-c)/c$. Note that
this assumption can be done since $ -\mathcal{W}_{-1}(-c)/c > 1/c$ thus implying that $n>1/c$, which concludes the proof. 
}
\end{proof}

\newpage

\section{\textcolor{black}{Appendix}}

{\color{black} In this section we present some results inspired from \cite{tsybakov}.}
%{\color{magenta} The first remark is that all the results in \cite{tsybakov} hold for the function $f$ defined on $\mathbb R^d$. According to our Lemmas \ref{lemma : g holderienne} and \ref{lemma: formule densité} these results can also be applied to the function $g$ defined in Eq. \eqref{def: g}. 
\textcolor{black}{More precisely, Lemma \ref{lemma 11}  and Lemma \ref{lemma 12} correspond to Lemma 11 and Lemma 12 of \cite{tsybakov} but without choosing any specific value for the {\color{black} kernel-}bandwidth $h_n$ and the regularization parameter $\lambda_n$ used in the local polynomial approach introduced in \eqref{def: ghat}. {\color{black}Another} difference is that in Lemma \ref{lemma 12} the logarithmic factor has been replaced by the Lambert function $\mathcal W$ defined in \eqref{def: lambert functions}.
Moreover, while Lemma \ref{lemma 11}  and Lemma \ref{lemma 12} hold for the gradient estimator $\widehat{\nabla g}$ of $\nabla g$, 
Lemma \ref{lemma 11bis} and Lemma \ref{lemma 12bis} hold directly for the estimator $\hat g$ of $g$. 
}

Note that in the following, the constant $A$, {\color{black} $A_1$ and $A_2$ are} not necessarily the same from one line to the other.
%{\color{magenta}
%Recall that $\Theta \subseteq \mathbb R^d$ is a compact set and $\Theta' = \{x+y; x\in \Theta, \Vert y\Vert\leq1\}.$
%Let $W$ be the $r\times d$ matrix defined in Eq. \eqref{def: g} and recall that $z_i=W^Tx_i$, with $x_i$ defined in \eqref{eq: regression model} {\color{orange} and satisfying \eqref{ass: density} (?)}.
%}
%
For {\color{black}the sake of simplicity, in all this section we denote}
\begin{equation}\label{def: real theta_r}
 \Theta_r = W\Theta %\quad \text{and} \quad \Theta_r'=W\Theta',}%\hspace{1.5ex}\text{and}\hspace{1.5ex} {\color{salmon} \Theta_r'' = \Theta_r'+ { B(0,1)} \ \text{[\`a enlever]}}
\end{equation}
%{\color{black}[$ \Theta_r'$ n'est plus utilis\'e] }
where %{\color{salmon}\sout{$\Theta'$ and}}
$W$ {\color{black} is defined after (\ref{def: g}).} %{\color{salmon}\sout{, respectively.}}
% %{\color{orange}[notation $ B(0,1)$ pas introduite]}
% %and, {\color{magenta} given a set $\Omega$, we denote $\vert \Omega \vert = \sup_{u,v \in \Omega} \lVert u-v \rVert.$ }

\subsection{Lemmas for \textcolor{black}{the} gradient estimation}\label{lemmas:gradient_estimation}
\begin{lemma}\label{lemma 11} 

\textbf{Bound \textcolor{black}{for} the gradient estimator \textit{bias}}

{\color{black} Let $\mathcal W$ be the Lambert function defined in \eqref{def: lambert functions}.}
{\color{black} Under \eqref{ass: assumption K} -- \eqref{ass: density}}, the estimator {\color{black}  $\widehat{\nabla g}$ } of the gradient {\color{black} of $g$} defined in \eqref{def: ghat} for some kernel-bandwidth $h_n$ and regularization {\color{black} parameter} $\lambda_n$ such that $nh_n^r \geq \max \left(\mathcal{W}(n),\lambda_n^{-2}\mathcal{W}(n)\right)$ satisfies
\begin{equation*}
\forall z \in \Theta_r,\quad \left\Vert \mathbb{E}\left(\widehat{\nabla g}(z)\right) - \nabla g(z) \right \Vert \leq A_1\left(h_n^{\beta-1}+h_n^{-1}\lambda_n+h_n^{-1-\frac{r}{2}}n^{-\frac{1}{2}}\right),
\end{equation*}
where $A_1$ is a positive constant {\color{black}and $\Theta_r$ is defined in \eqref{def: real theta_r}.}
\end{lemma}
\begin{proof}
By Lemmas \ref{lemma : g holderienne} and \ref{lemma: formule densité}, the assumptions of Lemma 11 in \cite{tsybakov}
%\cite[Lemma 11]{tsybakov} 
are satisfied for the function $g$. The inequality is obtained at the end of {\color{black}the} proof {\color{black} of Lemma 11, before replacing the values of $h_n$ and $\lambda_n$. Here we choose} $k=n$.
\end{proof}

\begin{lemma}\label{lemma 12}\textbf{Bound \textcolor{black}{for} the gradient estimator \textit{variance}}

{\color{black} Let $\mathcal W$ be the Lambert function defined in \eqref{def: lambert functions}.}
{\color{black} Under \eqref{ass: assumption K} -- \eqref{ass: density}}, consider {\color{black} the} estimator {\color{black} $\widehat{\nabla g}$ } of the gradient {\color{black} of $g$} defined in \eqref{def: ghat} for some kernel-bandwidth $h_n$ and regularisation {\color{black} parameter} $\lambda_n$ such that $nh_n^r \geq \max \left(\mathcal{W}(n),\lambda_n^{-2}\mathcal{W}(n)\right)$.
{\color{black}Then,}
\begin{equation*}
        \mathbb{E}\left[\sup_{z \in \Theta_r} \lVert  \widehat{\nabla g}(z) - \mathbb{E}[\widehat{\nabla g}(z)]\rVert^2 \right] \leq A_2h_n^{-r-2}n^{-1}\mathcal{W}(n),
\end{equation*}
where $\Theta_r$ {\color{black} is} defined in \eqref{def: real theta_r} and $A_2$ is a positive constant.
\end{lemma}

\begin{proof}
By Lemmas \ref{lemma : g holderienne} and \ref{lemma: formule densité}, the assumptions of {\color{black} Lemma 12 in} \cite{tsybakov}
%\cite[Lemma 12]{tsybakov} 
are satisfied for {\color{black} the function} $g$. Following the steps of the proof {\color{black} of Lemma 12} and choosing $k=n$ we {\color{black} have}
\begin{equation*}
        \mathbb{E}\left[\sup_{z \in \Theta_r} \lVert  \widehat{\nabla g}(z) - \mathbb{E}[\widehat{\nabla g}(z)]\rVert^2 \right]  \leq h_n^{-2}A\left(a+b+c\right),
\end{equation*}
where $A$ is a positive constant and 
\begin{align*}
&a = \left(\mathbb{E}\left[\sup_{z\in \Theta_r}\Vert B_{n,\lambda_n}(z)^{-1}\Vert^4_{\color{black}\text{op}}\right]\mathbb{E}\left[\sup_{z\in \Theta_r}\Vert G_n(z)\Vert^4_{\color{black}\text{op}}\right]\right)^{\frac{1}{2}},\\
&b = \left(\mathbb{E}\left[\sup_{z\in \Theta_r}\Vert B_{n,\lambda_n}(z)^{-1}\Vert^4_{\color{black}\text{op}}\right]\mathbb{E}\left[\sup_{z\in \Theta_r}\Vert C_{n}(z)- \mathbb{E}(C_n(z))\Vert^4_{\color{black}\text{op}}\right]\right)^{\frac{1}{2}},\\
&c = \mathbb{E}\left[\sup_{z\in \Theta_r}\Vert B_{n,\lambda_n}(z)^{-1}- (\mathbb{E}(B_{n,\lambda_n}{\color{black}(z)}))^{-1}\Vert^2_{\color{black}\text{op}}\right]\sup_{z\in \Theta_r}\mathbb{E}\left[\Vert C_{n}(z)\Vert^2_{\color{black}\text{op}}\right],
\end{align*}
\textcolor{black}{where $B_{n,\lambda_n}$ is defined in \eqref{def Bnlambda},
% Define, for any $z\in \Theta_r$, 
\begin{equation}\label{def: Gn}
 G_n(z) = n^{-1}h_n^{-r}\sum_{i=1}^n U\left(\frac{z_i-z}{h_n}\right) K\left(\frac{z_i-z}{h_n}\right)\xi_i,\;\forall z\in \Theta_r
 \end{equation}
 and
 \begin{equation}\label{def: Cn}
 %\forall z \in W\Theta, \hspace{1.5ex}
 C_{n}(z) = n^{-1}h_n^{-r}\sum_{i=1}^nU\left(\frac{z_i-z}{h_n}\right) K\left(\frac{z_i-z}{h_n}\right)g(\textcolor{black}{z_i}),\;\textcolor{black}{\forall z\in \Theta_r},
 \end{equation}
$U$ being defined in \eqref{def U and S}.}
{\color{black} Since $nh_n^r \geq \max\left(\lambda_n^{-2}\mathcal{W}(n),\mathcal{W}(n)\right)$ we can use \eqref{eq:lemma20.2} and \eqref{eq:lemma20.3} of Lemma \ref{lemma 20}  to get}
\begin{align*}
&a \leq A\lambda_{min}^{-4}\mathbb{E}\left[\sup_{z\in \Theta_r}\Vert  G_n(z)\Vert^4_{\color{black}\text{op}}\right]^{\frac{1}{2}}\\
&b \leq A\lambda_{min}^{-4}\mathbb{E}\left[\sup_{z\in \Theta_r}\Vert C_{n}(z)- \mathbb{E}(C_n(z))\Vert^4_{\color{black}\text{op}}\right]^{\frac{1}{2}}\\
& c \leq Ah_n^{-r}n^{-1}\mathcal{W}(n)\sup_{z\in \Theta_r}\mathbb{E}\left[\Vert C_{n}(z)\Vert^2_{\color{black}\text{op}}\right].
\end{align*}
{\color{black} To conclude the proof, we use} Lemma \ref{lemma 21} {\color{black} to bound the term} $a$, Lemma \ref{lemma 23} {\color{black} to bound the term} $b$ and {\color{black}Lemma \ref{lemma 22} to bound the term} $c$. 
%, {\color{magenta} so that }
%we obtain 
%\begin{equation*}
%a,b,c \leq Ah_n^{-r}n^{-1}\mathcal{W}(n).
%\end{equation*}
\end{proof}

\begin{lemma}\label{lemma 11bis} 

\textbf{Bound \textcolor{black}{for} the ponctual estimation \textit{bias}}

{\color{black} Let $\mathcal W$ be the Lambert function defined in \eqref{def: lambert functions}.}
{\color{black} Under \eqref{ass: assumption K} -- \eqref{ass: density}}, the estimator $\hat{g}$ of $g$ defined in \eqref{def: g estimation} for some kernel-bandwidth $h_n$ and regularisation {\color{black} parameter} $\lambda_n$ such that $nh_n^r \geq \max \left(\mathcal{W}(n),\lambda_n^{-2}\mathcal{W}(n)\right)$ satisfies
\begin{equation*}
\forall z \in \Theta_r,\quad \left\Vert \mathbb{E}\hat{g}(z) -  g(z) \right \Vert \leq {\color{black}A_1}\left(h_n^{\beta}+\lambda_n+h_n^{-\frac{r}{2}}n^{-\frac{1}{2}}\right)
\end{equation*}
{\color{black} where $\Theta_r$ is defined in \eqref{def: real theta_r} and $A_1$ is a positive constant.}
\end{lemma}
\begin{proof}
{\color{black} The proof is analogous to the proof of Lemma \ref{lemma 11}, where instead of the gradient estimator $\widehat{\nabla g}$ defined in \eqref{def: ghat}  we consider the estimator $\hat g$ of $g$. If we compare the two definitions, we can see that the only difference in the definition of $\hat g$ is that the matrix $M$ is replaced by the matrix $\widetilde M$ and that there is no factor $h_n^{-1}$. This implies that} we obtain the same bound as in Lemma \ref{lemma 11} up to the factor $h_n$.
%{The definition of $\hat g$ is analogous to the definition of $\widehat{\nabla g}$: $M$ is replaced by $\tilde{M}$ and there is no factor $h_n^{-1}$ (see Eq \eqref{def: ghat}). Thus one can aslo follow the steps of Lemma 11 of \cite{tsybakov} and get the same bound as in Lemma \ref{lemma 11} up to the factor $h_n$.}
\end{proof}

\begin{lemma}\label{lemma 12bis}\textbf{Bound of the ponctual estimation \textit{variance}}

{\color{black} Let $\mathcal W$ be the Lambert function defined in \eqref{def: lambert functions}.}
{\color{black} Under \eqref{ass: assumption K} -- \eqref{ass: density}}, the estimator $\hat{g}$ of $g$ defined in \eqref{def: g estimation} for some kernel-bandwidth $h_n$ and regularisation {\color{black} parameter} $\lambda_n$ such that $nh_n^r \geq \max \left(\mathcal{W}(n),\lambda_n^{-2}\mathcal{W}(n)\right)$, satisfies
\begin{equation*}
        \mathbb{E}\left[\sup_{z \in \Theta_r} \lVert  \hat g(z) - \mathbb{E}[\hat g(z)]\rVert^2 \right] \leq {\color{black}A_2}h_n^{-r}n^{-1}\mathcal{W}(n)
\end{equation*}
{\color{black} where $\Theta_r$ is defined in \eqref{def: real theta_r} and $A_2$ is a positive constant.}
\end{lemma}
\begin{proof}
  {\color{black} The proof is analogous to the proof of Lemma \ref{lemma 12}, where instead of the gradient estimator $\widehat{\nabla g}$ defined in \eqref{def: ghat}  we consider the estimator $\hat g$ of $g$. If we compare the two definitions, we can see that the only difference in the definition of $\hat g$ is that the matrix $M$ is replaced by the matrix $\widetilde M$ and that there is no factor $h_n^{-1}$.}
  %This implies that} we obtain the same bound as in Lemma \ref{lemma 12} up to the factor $h_n^2$.

%The definition of $\hat g$ is analogous to the definition of $\widehat{\nabla g}$: $M$ is replaced by $\tilde{M}$ and there is no factor $h_n^{-1}$ (see Eq \eqref{def: ghat}). Thus one can aslo follow the steps of Lemma 12 of \cite{tsybakov} and get the same bound as in Lemma \ref{lemma 12} up to the factor $h_n^2$.
\end{proof}

\subsection{\color{black} Technical lemmas used for proving the lemmas of Section \ref{lemmas:gradient_estimation}}

{\color{black}
In this section we introduce the technical results used for proving Lemmas \ref{lemma 11}, \ref{lemma 12}, \ref{lemma 11bis} and \ref{lemma 12bis}.
More precisely, Lemmas \ref{lemma 17} and \ref{lemma 22} correspond to Lemmas 17 and 22 of \cite{tsybakov}, respectively. As in the previous section, the results in \cite{tsybakov} hold for the function $f$, while the results presented here hold for the function $g$ defined in \eqref{def: g}. A sketch of the proof is provided for those lemmas. 
Lemma \ref{lemma 20} corresponds to Lemma 20 of \cite{tsybakov} and since it is the result that differs the most from the original one,
% In particular, we replace the logarithmic factor of Lemma 20 by the Lambert function $\mathcal W$ defined in \eqref{def: lambert functions}.
a complete proof is given hereafter. 
}

\begin{lemma}\label{lemma 17}
%{\color{magenta} According to the notation of the previous section, let $\Theta_r$ be the subset defined in \eqref{def: real theta_r}.
% For all $i \in {\color{black}\{1,\dots, n\}}$, let, for any $z\in \Theta_r$,
% \[
% R_{i}(z) = U\left( \frac{z_i-z}{h_n}\right) K\left( \frac{z_i-z}{h_n}\right)
% \]
% where $U$ is defined in \eqref{def U and S} and $K$ is a kernel {\color{black} satisfying \eqref{ass: assumption K}.}
% For any $q \geq 1$, let
% $\nu_q = \int_{\mathbb{R}^d} \|U(u)K(u)\|^q \, du$. Under \eqref{ass: density}, \textcolor{black}{for any $i \in \{1,\dots, n\}$}, we have
% \[
% \sup_{z \in \Theta_r} \mathbb{E}\left(\|R_{i}({\color{black}z})\|^q\right) \leq \nu_q h_n^r
% \]
\textcolor{black}{Under assumptions (\ref{ass: assumption K}) and (\ref{ass: density})} there exists a positive constant $\lambda_{\min}$ such that
\[
\inf_{z \in \Theta_r} {\color{black}\mu_{\min} }(\mathbb{E}(B_n(z))) \geq \lambda_{\min},
\]
{\color{black}where $B_n$ is defined in \eqref{def Bn}} and {\color{black} $\mu_{\min}(\mathbb{E}(B_n(z)))$ denotes the smallest eigenvalue of $\mathbb{E}(B_n(z))$.}
%\sout{\color{orange} [le $\lambda_{min}$ \`a droite et celui \`a gauche ne repr\'esentent pas la meme chose ou je me trompe ?]}
\end{lemma}
\begin{proof} {\color{black} By Lemma \ref{lemma: formule densité}, we can apply Lemma 17 of \cite{tsybakov} to the function  $g$
with $k=n$.}
%We apply Lemma 17 of \cite{tsybakov}, adapt the notations to $g$, take $k=n$ and use $\tilde{p}_M = 1$ (see Lemma \ref{lemma: formule densité}).
\end{proof}

\begin{lemma}\label{lemma 22}
  {\color{black}Let $\Theta_r$ be defined in \eqref{def: real theta_r}.}
  Under \eqref{ass: assumption K}, \eqref{ass: bounded} of \eqref{ass: regularity f} and \eqref{ass: density}, \textcolor{black}{there exists a positive constant $A$ such that, if $nh_n^r\geq 1$, the matrix $C_n$ defined  in \eqref{def: Cn} satisfies}
\begin{equation}\label{eq:lemma22.1}
\sup_{z \in \Theta_r} \mathbb{E}
\Bigl(
\bigl\| C_{n}(z) - \mathbb{E}(C_{n}(z)) \bigr\|_{\mathrm{op}}^{4}
\Bigr)
\leq An^{-2}h_n^{-2r},
\end{equation}
%\textcolor{black}{where $\Theta_r$ is defined in \eqref{def: real theta_r}}
{\color{black}and}
\begin{equation}\label{eq:lemma22.2}
\sup_{z \in \Theta_r} \mathbb{E}
\bigl\| C_{n}(z) \bigr\|_{\mathrm{op}}^{4}
\leq A.
\end{equation}
\end{lemma}
\begin{proof}
By Lemma \ref{lemma: formule densité}, {\color{black} the assumptions of} Lemma 22 in \cite{tsybakov}
are satisfied for the function $g$.  The inequality is obtained {\color{black} by choosing} $k=n$.
%By Lemma \ref{lemma: formule densité}, we can apply Lemma 22 of \cite{tsybakov} to $g$ and we take $k=n$.  
\end{proof}

\begin{lemma}\label{lemma 20} {\color{black} Let $n\geq3$.}
%{\color{magenta} Let $\mathcal W$ be the Lambert function defined in \eqref{def: lambert functions} and let $\Theta_r$ be the subset defined in \eqref{def: real theta_r}.}
{\color{black}Under \eqref{ass: assumption K},  \eqref{ass: regularity f}, \eqref{ass: noise} } and \eqref{ass: density}, for {\color{black} a bandwidth $h_n\leq 1$ such that} $nh_n^r \geq \mathcal W(n)$, \textcolor{black}{there exists a positive constant $A$ such that}
\begin{equation}\label{eq:lemma20.1}
\mathbb{E}
\left(\sup_{{\color{black}z} \in \Theta_r}
\bigl\| B_{n,\lambda_n}({\color{black}z}) - \mathbb{E}(B_{n,\lambda_n}({\color{black}z})) \bigr\|_{\mathrm{op}}^{4}
\right)
\leq A h_n^{-2r} n^{-2}
\mathcal{W}(n)^2, 
\end{equation}
where \textcolor{black}{$\Theta_r$, $B_{n,\lambda_n}$ and $\mathcal W$ are defined in \eqref{def: real theta_r}, \eqref{def Bnlambda} and \eqref{def: lambert functions}, respectively.}\\
{\color{black} Furthermore,} if $nh_n^r\geq \lambda_n^{-2}\mathcal{W}(n),$ we have 
\begin{equation}\label{eq:lemma20.2}
\mathbb{E}\left(\sup_{{\color{black}z} \in \Theta_r}\Vert B_{n,\lambda_n}({\color{black}z})^{-1} \Vert_{\mathrm{op}}^4 \right) \leq {\color{black} 16} \lambda_{\text{min}}^{-4},
\end{equation}
{\color{black} where $\lambda_{\text{min}}$ is defined in Lemma \ref{lemma 17},} and 
\begin{equation}\label{eq:lemma20.3}
\mathbb{E}
\left(\sup_{{\color{black}z} \in \Theta_r}
\bigl\| B_{n,\lambda_n}({\color{black}z})^{-1} - (\mathbb{E}(B_{n,\lambda_n}({\color{black}z})))^{-1} \bigr\|_{\mathrm{op}}^{2}
\right)
\leq A h_n^{-r} n^{-1}
\mathcal{W}(n).
\end{equation}
%{\color{black} [il me semble que ici $A=A_n = \lambda_{min}^{-2}\lambda_n^{-2}$]}
\end{lemma}

\begin{proof}
%{\color{magenta} According to the notation introduced in Section \ref{section: method}, let $S$ be cardinality of the set $\{ {\bf m} \in \mathbb N^d\ ;\ |{\bf m}|\leq \ell \}$ and $U$ be defined in \eqref{def U and S}.} {\color{orange} [$S$ ou $d+1$?]} 
{\color{black}According to the notation introduced in Section \ref{section:variable selection}} for $s,s' \in {\color{black}\{1,\dots, S\}}$ and $i \in {\color{black}\{1,\dots, n\}}$, define the functions 
\begin{align}\label{eq:def_phi}
\phi^{(s,s')}_i : \quad & \Theta_r \longrightarrow \mathbb{R}\nonumber\\
& {\color{black}z} \longmapsto \left(U\left(\frac{z_i-z}{h_n}\right)U\left(\frac{z_i-z}{h_n}\right)^{\!\top}\right)_{\color{black}s,s'}K\left(\frac{z_i-z}{h_n}\right).
\end{align}
%{\color{salmon}[plus loin on ne met plus $(s,s')$ mais directement $ss'$]}
%{\color{olive} QUESTION d'ILARIA : est-ce que définir cette fonction sur le compact $\Theta_r$ ne suffit pas à ne plus avoir de problème ?? : }
%
%\textcolor{black}{j'ai encore un problème sur le caractère lipschitz pour la partie polynomiale à cause du $1/h_n$.}
%\textcolor{black}{ATTENTION!!! Ecrire un petit lemme pour justifier}
%
By Lemma \ref{lem:lipschitz}, the functions $\phi^{(s,s')}_i$'s  are Lipschitz with Lipschitz constant equal to $\gamma_K h_n^{-1}$.
{\color{black} Hence} denoting 
\begin{equation}\label{eq: def fiss}
F^{(s,s')}_i = \phi^{(s,s')}_i - \mathbb{E}(\phi^{(s,s')}_i)
\end{equation}
we get {\color{black} for $z,y \in \Theta_r$ that}
\begin{equation}\label{eq: Fiss lipschitz}
\left| \sum_{i=1}^{n} \left( F_i^{(s,s')}(z) - F_i^{(s,s')}(y) \right) \right|
\leq {\color{black}2 \gamma_K} nh_n^{-1} \lVert z - y \rVert .
\end{equation}
\textcolor{black}{On the other hand, we have 
%\color{black} according to the definition of $B_{n,\lambda_n}$ in Eq. \eqref{def Bnlambda} and the definition of $\phi^{(s,s')}_i $, we can write,
for all $z\in \Theta_r$,}
\[
B_{n,\lambda_n}(z) = n^{-1}h_n^{-r}\sum_{i=1}^n \sum_{s,s'=1}^S \phi^{(s,s')}_i(z) +\lambda_n I_S.
\]
\textcolor{black}{Using the Jensen} inequality we have
\begin{align*}
\mathbb{E}\!\left[
\sup_{z \in \Theta_r} \| B_{n,\lambda_n}(z)-\mathbb{E}(B_{n,\lambda_n}(z)) \|_{\mathrm{op}}^{4}
\right]
& \leq \mathbb{E}\!\left[
\sup_{z \in \Theta_r} \left(\sum_{s,s'=1}^S \left| n^{-1}h_n^{-r}\textcolor{black}{\sum_{i=1}^n}F_i^{(s,s')}(z) \right|   \right)^{4}
\right]
\\
& \leq S^6 \mathbb{E}\!\left[
\sup_{z \in \Theta_r} \sum_{s,s'=1}^S \left| n^{-1}h_n^{-r}\textcolor{black}{\sum_{i=1}^n}F_i^{(s,s')}(z) \right|^{4}
\right]
\\
&\leq S^6 \sum_{s,s'=1}^{S}
\mathbb{E}\!\left[
\sup_{z \in \Theta_r}
\left|
n^{-1} h_n^{-r}
\sum_{i=1}^{n} F_i^{(s,s')}(z)
\right|^{4}
\right].
\end{align*}

\textcolor{black}{Let 
\begin{equation}\label{epsilon_n}
\varepsilon_n
= \sqrt{\frac{1}{2n h_n^{-r-2}}
\mathcal{W}\!\left(8 n h_n^{-r-2} \vert \Theta_r \vert^{2}\right)}
\end{equation}
where we recall that $\vert \Theta_r \vert =\sup_{u,v\in \Theta_r} \|u-v\|$, 
and consider $\mathcal{N}$
the \textcolor{black}{minimal $\varepsilon_n$-net of $\Theta_r$} with respect to the Euclidean norm.}
%the corresponding $\varepsilon$-net $\mathcal{N}$ of $\Theta_r$ satisfying Lemma \ref{eq: cardinal}} \sout{with respect to the Euclidean norm.}
%{\color{magenta} Let } $\varepsilon_n \in \left]0, \vert \Theta_r \vert\right]$, {\color{magenta} where $\vert \Theta_r \vert =\sup_{u,v\in \Theta_r} \|u-v\|$ and} let $\mathcal{N}$ be the minimal $\varepsilon$-net of $\Theta_r$
%with respect to the Euclidean norm.
\textcolor{black}{Then, we can write}
\begin{align*}
\sup_{z \in \Theta_r}
\left|
n^{-1} h_n^{-r}
\sum_{i=1}^{n} F_i^{(s,s')}(z)
\right|^{4} 
&\leq \Bigg( \sup_{y \in \mathcal{N}}
\left|
n^{-1} h_n^{-r}
\sum_{i=1}^{n} F_i^{(s,s')}(y)
\right| \\
& \quad \quad + \sup_{\Vert z-y \Vert \leq \varepsilon_n}
\left|
n^{-1} h_n^{-r}
\sum_{i=1}^{n} \left(F_i^{(s,s')}(z)-F_i^{(s,s')}(y)\right)
\right| \Bigg)^{4},
\end{align*}
which, using {\color{black} the} Jensen inequality, leads to 
\begin{align*}
\mathbb{E}\!\left(
\sup_{z \in \Theta_r} \| B_{n,\lambda_n}(z)-\mathbb{E}(B_{n,\lambda_n}(z)) \|_{\mathrm{op}}^{4}
\right)
&\leq 8S^6 \sum_{s,s'=1}^{S}
\mathbb{E}\left(\sup_{z \in \mathcal{N}}
\left|
n^{-1} h_n^{-r}
\sum_{i=1}^{n} F_i^{(s,s')}(z)
\right|^4\right)\\
& + 8S^6\sum_{s,s'=1}^{S}
\mathbb{E}\left(\sup_{\Vert z-y \Vert \leq \varepsilon_n}
\left|
n^{-1} h_n^{-r}
\sum_{i=1}^{n} (F_i^{(s,s')}(z)-F_i^{(s,s')}(y))
\right|^{4}\right).
\end{align*}
By \eqref{eq: Fiss lipschitz}, we obtain 
%\begin{align*}
\begin{equation}\label{star}
\mathbb{E} \left(
\sup_{z \in \Theta_r} \| B_{n,\lambda_n}(z)-\mathbb{E}(B_{n,\lambda_n}(z)) \|_{\mathrm{op}}^{4}
\right)
%&
\leq 8S^6 \left[\sum_{s,s'=1}^{S}
\mathbb{E}\!\left(\sup_{z \in \mathcal{N}}
\left|
n^{-1} h_n^{-r}
\sum_{i=1}^{n} F_i^{(s,s')}(z)
\right|^4\right)
%\& 
+ %8S^6
{\color{black} 16\gamma_K^4} h_n^{-4r-4}\varepsilon_n^4\right].
%\end{align*}
\end{equation}
Note that {\color{black} we can rewrite the first term on the right-hand side of \eqref{star} as}
\begin{align*}
\mathbb{E}\!\left(\sup_{z \in \mathcal{N}}
\left|n^{-1} h_n^{-r}
\sum_{i=1}^{n} F_i^{(s,s')}(z)
\right|^4\right) 
& = 4\int_0^{+\infty} t^3 \mathbb{P}\left(\sup_{z \in \mathcal{N}}\left|
n^{-1} h_n^{-r}
\sum_{i=1}^{n} F_i^{(s,s')}(z)
\right| \geq t\right)\mbox{d}t.
%& {\color{teal} \leq 4\left( \int_0^{a} + \int_a^{+\infty}t^3 \mathbb{P}\left(\sup_{z \in \mathcal{N}}\left|n^{-1} h_n^{-r}\sum_{i=1}^{n} F_i^{(s,s')}(z)\right| \geq t\right)\mbox{d}t\right)}
\end{align*}
\textcolor{black}{To apply Lemma \ref{lemma 26}, we shall prove that $\varepsilon_n$ defined in \eqref{epsilon_n} is such that $\varepsilon_n \leq \vert \Theta_r \vert$ and 
\begin{equation}\label{claim}
n^{-1}h_n^{-r} \log \left(2\frac{\vert \Theta_r\vert}{\varepsilon_n}\right) \leq 20 \vert \Theta_r\vert^2.
\end{equation}}
\textcolor{black}{Using (\ref{prop log}) with $\kappa=2 \vert \Theta_r \vert^2$ and (\ref{eq:majoration_W_log}) of Lemma \ref{lemma: proporties lambert function}, we get
\[
\varepsilon_n \leq \sqrt{ \vert \Theta_r \vert^2 \ \frac{\mathcal{W}\!\left(4n h_n^{-r-2}\right)}{n h_n^{-r-2}}} \leq \sqrt{ \vert \Theta_r \vert^2 \ \frac{\log\!\left(4n h_n^{-r-2}\right)}{n h_n^{-r-2}}}.
%\leq  \vert \Theta_r \vert
\]
Moreover $n h_n^{-r-2}\geq n\geq 3$, since $h_n\leq 1$, and $\log(4x)/x \leq 1,$ for all $x\geq 3$, which gives $\varepsilon_n\leq   \vert \Theta_r \vert.$}
%\textcolor{black}{On ne peut pas s'en sortir avec 2$\vert \Theta_r \vert$ au lieu de $\vert \Theta_r \vert$ ?
%  By (\ref{eq:croissance W}),
 % $$
 % \varepsilon_n \leq \vert \Theta_r \vert \sqrt\frac{4n h_n^{-r-2}}{n h_n^{-r-2}}}=2\vert \Theta_r \vert.
%$$ } 
To prove \eqref{claim}, \textcolor{black}{observe} that,
by definition of $\varepsilon_n$ given by (\ref{epsilon_n}), 
\[
2\varepsilon_n^2 nh_n^{-r-2} = \mathcal{W}\!\left(8 n h_n^{-r-2} \vert \Theta_r \vert^{2}\right).
\] 
Moreover, by definition of the Lambert function $\mathcal W$, 
\[
\varepsilon_n^2\exp\left( 2\varepsilon_n^2 nh_n^{-r-2} \right) = 4\vert \Theta_r \vert^{2},
\]
which \textcolor{black}{can be rewritten as
$$
\varepsilon_n\exp\left(\varepsilon_n^2 nh_n^{-r-2} \right) = 2\vert \Theta_r \vert,
$$
which leads to}
\begin{equation}\label{eps_n ineq}
\varepsilon_n^2 nh_n^{-r-2} = \log\left(2\frac{\vert \Theta_r\vert}{\varepsilon_n}\right).
\end{equation}
Thus, by (\ref{epsilon_n}), we have
\begin{align}\label{ineq_suppl}
n^{-1}h_n^{-r} \log \left(2\frac{\vert \Theta_r\vert}{\varepsilon_n}\right)  = \varepsilon_n^2h_n^{-2r-2} & = \frac{1}{2} \mathcal{W}\!\left(8 n h_n^{-r-2} \vert \Theta_r \vert^{2}\right)n^{-1}h^{-r}\nonumber\\
& \leq 4  \vert \Theta_r \vert^{2} \mathcal{W}\!\left( n h_n^{-r-2}\right)n^{-1}h^{-r}
\end{align}
where the last inequality follows from \eqref{prop log} of Lemma \ref{lemma: proporties lambert function} with $\kappa=8 \vert \Theta_r \vert^2$. Observe that, since by assumption,
\begin{equation}\label{eq:min_nh_nr}
  nh_n^r \geq \mathcal W(n)\geq 1,
\end{equation}
\textcolor{black}{where the last inequality comes from (\ref{eq:minoration_W}), we get}
% and $h_n \le 1$ {\color{orange} [on a besoin de $h<1$?]} 
\[ n h_n^{-r-2}
\le n^{2+\frac{2}{r}}.
\]
{\color{black} Using that $\mathcal{W}$ is non decreasing, we get that
  \begin{equation}\label{eq:W_croissante}
    \mathcal{W}(n h_n^{-r-2})\leq\mathcal{W}(n^{2+\frac{2}{r}})
  \end{equation}
  and hence, by  \eqref{eq:on W} of Lemma \ref{lemma: proporties lambert function}, \eqref{ineq_suppl} becomes
\begin{equation}\label{20theta}
  n^{-1}h_n^{-r} \log \left(2\frac{\vert \Theta_r\vert}{\varepsilon_n}\right)\leq 20  \vert \Theta_r \vert^{2}n^{-1}h_n^{-r} \mathcal{W}(n)\leq 20  \vert \Theta_r \vert^{2},
\end{equation}
where the last inequality comes from (\ref{eq:min_nh_nr}).
%  \eqref{eq:on W} of Lemma \ref{lemma: proporties lambert function} 
%we conclude that 
%\begin{equation}\label{20theta}
%\varepsilon_n^2 nh_n^{-r-2}\leq 20  \vert \Theta_r \vert^{2}n^{-1}h_n^{-r} \mathcal{W}(n) \leq 20  \vert \Theta_r \vert^{2}
%\end{equation}
}
{\color{black} Since all the assumptions of  Lemma \ref{lemma 26} are satisfied, we obtain that
\begin{align*}
\mathbb{E}\!\left[\sup_{z \in \mathcal{N}}
\left|
n^{-1} h_n^{-r}
\sum_{i=1}^{n} F_i^{(s,s')}(z)
\right|^4\right]
& = 4\int_0^{+\infty} t^3 \mathbb{P}\left(\sup_{z \in \mathcal{N}}\left|
n^{-1} h_n^{-r}
\sum_{i=1}^{n} F_i^{(s,s')}(z)
\right| \geq t\right)\mbox{d}t\\
& \leq  An^{-2}h_n^{-2r}\log\left(\frac{2\vert \Theta_r \vert}{\varepsilon_n}\right)^2,
\end{align*}
\textcolor{black}{where $A$ is a positive constant.} Thus, \eqref{star} becomes 
}
\begin{equation}\label{eq: condition chiante}
\mathbb{E}\!\left[
\sup_{z \in \Theta_r} \| B_{n,\lambda_n}(z)-\mathbb{E}(B_{n,\lambda_n}(z)) \|_{\mathrm{op}}^{4}
\right] \leq A\left (n^{-2}h_n^{-2r}\log\left(\frac{2\vert \Theta_r \vert}{\varepsilon_n}\right)^2 + h_n^{-4r-4}\varepsilon_n^4\right)
\end{equation}
{\color{black} for a suitable positive constant $A$.
  By the first two equalities of \eqref{ineq_suppl} %\eqref{eps_n ineq} and (\ref{epsilon_n}),
} %{\color{salmon}[ce n'est pas plutot le calcul fait entre \eqref{ineq_suppl} et \eqref{eps_n ineq} ? ]}
  we get %\eqref{20theta} and the definition of $\varepsilon_n$, 
\begin{align*}
\mathbb{E}\!\left[
\sup_{z \in \Theta_r} \| B_{n,\lambda_n}(z)-\mathbb{E}(B_{n,\lambda_n}(z)) \|_{\mathrm{op}}^{4}
\right] \leq 2 A h_n^{-4r-4}\varepsilon_n^4=\frac{2A}{4} n^{-2} h_n^{-2r}\mathcal{W}(8nh_n^{-r-2}\vert \Theta_r \vert^2)^2.% \\ 
%& \leq 800A n^{-4}h_n^{-4r} \vert \Theta_r \vert^{4} \mathcal W(n)^2\\
%&{\color{gray} \leq 800 A n^{-2}h_n^{-2r} \vert \Theta_r \vert^{4} \mathcal W(n)^2}
\end{align*}
By (\ref{prop log}) of Lemma \ref{lemma: proporties lambert function}, we have
\begin{align*}
\mathbb{E}\!\left[
\sup_{z \in \Theta_r} \| B_{n,\lambda_n}(z)-\mathbb{E}(B_{n,\lambda_n}(z)) \|_{\mathrm{op}}^{4}
  \right] &\leq \frac{8^2A}{2}\vert \Theta_r \vert^4 n^{-2} h_n^{-2r}\mathcal{W}(nh_n^{-r-2})^2\\
  &\leq 32 A\vert \Theta_r \vert^4 n^{-2} h_n^{-2r}\mathcal{W}(n^{2+\frac{2}{r}})^2\leq 800A\vert \Theta_r \vert^4 n^{-2} h_n^{-2r}\mathcal{W}(n)^2,
\end{align*}
where the second inequality comes from (\ref{eq:W_croissante}) and the last one from (\ref{eq:on W}), which gives \eqref{eq:lemma20.1}.
% {\color{gray} since by assumption $nh_n^r\geq \mathcal W(n)$ and  $W(n)\geq 1$ for $n\geq 3$ by \eqref{eq:minoration_W} of Lemma \ref{lemma: proporties lambert function}}
% {\color{orange} [La 2eme inegalite n'est pas mieux que la 3eme?] }

{\color{black}
Let us now prove \eqref{eq:lemma20.2}. By the Jensen inequality, we have 
\begin{align*}
\mathbb{E}\left[\sup_{z \in \Theta_r}\Vert B_{n,\lambda_n}(z)^{-1} \Vert_{\text{op}}^4 \right]
&\leq \mathbb{E}\left[\sup_{z \in \Theta_r}\left(\Vert B_{n,\lambda_n}(z)^{-1} - (\mathbb{E}B_{n,\lambda_n}(z))^{-1}\Vert_{\text{op}} + \Vert (\mathbb{E}B_{n,\lambda_n}(z))^{-1}\Vert_{\text{op}}\right)^4 \right] \\
&\leq \mathbb{E}\left[\sup_{z \in \Theta_r}\left(8\Vert B_{n,\lambda_n}(z)^{-1} - (\mathbb{E}B_{n,\lambda_n}(z))^{-1}\Vert_{\text{op}}^4 + 8\Vert (\mathbb{E}B_{n,\lambda_n}(z))^{-1}\Vert_{\text{op}}^4\right) \right] \\
& \leq 8\mathbb{E}\left[\sup_{z \in \Theta_r}\Vert B_{n,\lambda_n}(z)^{-1} - (\mathbb{E}B_{n,\lambda_n}(z))^{-1}\Vert_{\text{op}}^4\right] + 8\sup_{z \in \Theta_r}\Vert (\mathbb{E}B_{n,\lambda_n}(z))^{-1}\Vert_{\text{op}}^4.
%& \leq 8\mathbb{E}\left[\sup_{z \in \Theta_r}\Vert B_{n,\lambda_n}(z)^{-1} - (\mathbb{E}B_{n,\lambda_n}(z))^{-1}\Vert^4\right] + 8\lambda_{\text{min}}^{-4},
\end{align*}
Observe that
\begin{multline}\label{eq:inegalité decomposition B}
\Vert B_{n,\lambda_n}(z)^{-1} - (\mathbb{E}(B_{n,\lambda_n}(z)))^{-1}\Vert_{\text{op}}\\ \leq \Vert(\mathbb{E}(B_{n,\lambda_n}(z)))^{-1}\Vert_{\text{op}}\Vert B_{n,\lambda_n}(z)^{-1}\Vert_{\text{op}} \Vert B_{n,\lambda_n}(z) - {\mathbb{E}(B_{n,\lambda_n}(z)))}\Vert_{\text{op}}
\end{multline}
and that for any  $z \in \Theta_r$
\begin{align*}
\Vert B_{n,\lambda_n}(z)^{-1}\Vert_{\text{op}} %= \max( Sp(B_{n,\lambda_n}(z)^{-1}) 
= \frac{1}{\mu_{min}(B_{n,\lambda_n}(z)) }
 = \frac{1}{\mu_{min}(B_n(z))+\lambda_n}
 \leq \frac{1}{\lambda_n},
\end{align*}
where $\mu_{min}(A)$ denotes the smallest eigenvalue of the matrix $A$. 
Moreover, by Lemma \ref{lemma 17}
\[
\sup_{z \in \Theta_r} \Vert (\mathbb{E}(B_{n}(z)))^{-1}\Vert_{\text{op}} =\sup_{z \in \Theta_r}\frac{1}{\mu_{min}(\mathbb{E}(B_{n}(z))) }\leq\lambda_{\text{min}}^{-1}.
\]
Thus, 
\[
\sup_{z \in \Theta_r} \Vert (\mathbb{E}(B_{n,\lambda_n}(z)))^{-1}\Vert_{\text{op}} =\sup_{z \in \Theta_r}\frac{1}{\mu_{min}(\mathbb{E}(B_{n}(z)))+\lambda_n } \leq \lambda_{\text{min}}^{-1}.
\]
Hence, by (\ref{eq:inegalité decomposition B}) {\color{black} and \eqref{eq:lemma20.1}}% and {\color{salmon} since by assumption  $nh_n^r\geq \lambda_n^{-2}\mathcal{W}(n),$}
\begin{align}\label{eq:maj1}
\mathbb{E}\left(\sup_{\color{black}z \in \Theta_r} \Vert B_{n,\lambda_n}(z)^{-1} - (\mathbb{E}B_{n,\lambda_n}(z))^{-1}\Vert_{\text{op}}^2\right) 
%&\leq \lambda_{min}^{-2}\mathbb{E}\left(\sup_z \left(\Vert B_{n,\lambda_n}(z)^{-1}\Vert_{op} \Vert B_{n,\lambda_n}(z) - (\mathbb{E}B_{n,\lambda_n}(z))\Vert_{op}\right)^2\right)\\
%&\leq \lambda_{min}^{-2}\mathbb{E}\left(\sup_z \Vert B_{n,\lambda_n}(z)^{-1}\Vert_{op}^4\right)^{\frac{1}{2}}\mathbb{E}\left( \sup\Vert B_{n,\lambda_n}(z) - (\mathbb{E}B_{n,\lambda_n}(z))\Vert_{op}^4\right)^{\frac{1}{2}}.
\leq \lambda_{min}^{-2} \lambda_n^{-2} h_n^{-{\color{black}r}} n^{-1}\mathcal{W}(n),% {\color{salmon} \leq \lambda_{min}^{-2}}
\end{align}
Using the assumption $nh_n^r \geq \lambda_n^{-2}\mathcal{W}(n)$, we thus get
\begin{align*}
\mathbb{E}\left[\sup_{z \in \Theta_r}\Vert B_{n,\lambda_n}(z)^{-1} \Vert_{\text{op}}^4 \right]
 %\leq 8\lambda_{min}^{-4}\left(\lambda_n^{-4}\mathbb{E}\left[\sup_{z \in \Theta_r}\Vert B_{n,\lambda_n}(z) - (\mathbb{E}B_{n,\lambda_n}(z))\Vert^4\right]+ 1\right)
 \leq 16 \lambda_{min}^{-4},
%& \leq A\lambda_{min}^{-4}\left(\lambda_n^{-4} h_n^{-2{\color{magenta}r}} n^{-2}
%\mathcal{W}(n)^2+ 1\right)\\
%&\leq {\color{salmon} 2A} \lambda_{min}^{-4}.
\end{align*}
which concludes the proof of \eqref{eq:lemma20.2}. 

Using \eqref{eq:inegalité decomposition B}, \eqref{eq:lemma20.1} and \eqref{eq:lemma20.2} we obtain \eqref{eq:lemma20.3}, which concludes the proof of the lemma.

}

\end{proof}

\begin{lemma}\label{lemma 27}
{\color{black} Let $\Theta_r$ be defined in \eqref{def: real theta_r}.}
  Assume that \eqref{ass: assumption K}, {\color{black} \eqref{ass: noise} and \eqref{ass: density}} hold.
  %{\color{olive} and} {\color{olive} according to the notation of Lemma \ref{lemma 20},} denote, {\color{olive} for $z\in \Theta_r$}
%\begin{align*}
%F^{(s,s')}_i {\color{olive}(z)} &= \left(U\left(\frac{z_i-z}{h_n}\right)  U\left(\frac{z_i-z}{h_n} \right)^{\!\top}\right)_{(s,s')}K\left(\frac{z_i-z}{h_n} \right) \\ & \qquad \qquad - \mathbb{E}\Bigg[\left(U\left(\frac{z_i-z}{h_n}\right)U\left(\frac{z_i-z}{h_n}\right)^{\!\top}\right)_{(s,s')}K\left(\frac{z_i-z}{h_n}\right)\Bigg].
%\end{align*}
Then, there exists a positive constant $C$ such that for all positive $t$ and $z\in \Theta_r$
\begin{equation}\label{eq-lemma27}
\mathbb{P}\left(\left|\frac{1}{n}h_n^{-r}\sum_{i=1}^n F_i^{(s,s')}{\color{black}(z)} \right| > t \right) \leq 2 \exp\left(-\frac{t^2nh_n^{r}}{{\color{black}C}(1+t)}\right),
\end{equation}
\begin{equation}\label{eq-lemma27-1}
\mathbb{P}\left(\left|\frac{1}{n}h_n^{-r}\sum_{i=1}^n H_i^{(s)}{\color{black}(z)} \right| > t \right) \leq 2 \exp\left(-\frac{t^2nh_n^{r}}{C(1+t)}\right),
\end{equation}
\begin{equation}\label{eq-lemma27-2}
\mathbb{P}\left(\left|\frac{1}{n}h_n^{-r}\sum_{i=1}^n F_i^{(s)}(z) \right| > t \right) \leq 2 \exp\left(-\frac{t^2nh_n^{r}}{C(1+t)}\right),
\end{equation}
where $F_i^{(s,s')}$ is defined in (\ref{eq: def fiss}), %$H_i^{(s)}$ in \eqref{def his} and $F_i^{(s)}$ in \eqref{eq:def fis}.
\begin{equation}\label{def his}
H_i^{(s)}(z)
=
U\!\left(\frac{z_i - z}{h_n}\right)
_{s}
K\!\left(\frac{z_i - z}{h_n}\right)g(z_i)
-
\mathbb{E}
\left[
U\!\left(\frac{z_i - z}{h_n}\right)
_{s}
K\!\left(\frac{z_i - z}{h_n}\right)g(z_i)
\right],
\end{equation}
and 
\begin{equation}\label{eq:def fis}
F_i^{(s)}(z) = U\left(\frac{z_i-z}{h_n}\right)_{s}K\left(\frac{z_i-z}{h_n}\right)\xi_i
\end{equation}
\end{lemma}

\begin{proof} 

{\color{black} Let $z\in \Theta_r$. To prove \eqref{eq-lemma27}, we show that the Bernstein condition \textcolor{black}{of Lemma \ref{lemma: Bernstein}:}
\[
\mathbb{E}\left[\left|h_n^{-r}F_i^{(s,s')}(z)\right|^q\right]  \leq \frac{q!}{2}\nu H^{q-2}
\]
is satisfied for any $q\in\mathbb N$, $q\geq2$ and for some $\nu,H>0$.
Let $X$ be a random variable. By the Jensen inequality, we have
\[\mathbb{E}\left[\left|X-\mathbb{E}(X)\right|^q\right] \leq 2^{q-1}(\mathbb{E}(|X|^q)+|\mathbb{E}(X)|^q)\leq 2^q\mathbb{E}(|X|^q).
\]
Thus,
\begin{align*}
\mathbb{E}\left[\left|h_n^{-r}F_i^{(s,s')}(z)\right|^q\right]  
& \leq h_n^{-rq}2^q\mathbb{E}\left[\left|\left(U\left(\frac{z_i-z}{h_n}\right)U\left(\frac{z_i-z}{h_n}\right)^{\!\top}\right)_{ss'}K\left(\frac{z_i-z}{h_n}\right)\right|^q\right]\\
& \leq h_n^{-r(q-1)}2^q\int \left \vert \left(U(u)U(u)^{\!\top}\right)_{ss'}K(u)\right \vert^q \tilde p(z+h_nu)\mathrm{d}u.%\\
%& \leq h_n^{-r(q-1)}A^q,
\end{align*}}
%{\color{magenta} Question: $z_i \in \Theta_r', z\in \Theta_r$ et $u\in$? --> $\tilde p(z+h_nu)$ est bien d\'ef?} 
{\color{black}
By Lemma  \ref{lemma: formule densité}, $\tilde p$ is bounded by $1$ %{\color{magenta}[$\tilde p_M$ vs 1]}
and by \eqref{ass: assumption K} the kernel $K$ has a compact support so that  $\left(U(u)U(u)^{\!\top}\right)_{ss'}K(u)$ can be uniformly bounded. Hence, }
\[
\mathbb{E}\left[\left|h_n^{-r}F_i^{(s,s')}(z)\right|^q\right]   \leq h_n^{-r(q-1)}A^q.
\]
{\color{black} Denoting} $\nu = {\color{black}A^2}h_n^{-r}$ and $H = {\color{black}A}h_n^{-r}$ {\color{black} we have}
\[
\mathbb{E}\left[\left|h_n^{-r}F_i^{(s,s')}(z)\right|^q\right] \leq \nu H^{q-2}
 \leq \frac{q!}{2}\nu H^{q-2}.
\]
{\color{black} The result follows from the Bernstein inequality} \textcolor{black}{given} in Lemma \ref{lemma: Bernstein}.
\textcolor{black}{For \eqref{eq-lemma27-1}, we use the same line of reasoning to get}
\begin{align*}
\mathbb{E}\left[\left|h_n^{-r}H_i^{(s)}(z)\right|^q\right]  
\leq h_n^{-r(q-1)}2^q\int \left \vert U(u)_{s}K(u)g(z+h_nu)\right \vert^q \tilde p(z+h_nu)\mathrm{d}u.
\end{align*}
\textcolor{black}{Since $K$ has a compact support and  $g$ can be bounded on a compact we conclude by using the same arguments as those used for proving \eqref{eq-lemma27}.}
To prove \eqref{eq-lemma27-2},
% we recall that $F_i^{(s)}$ is defined \textcolor{black}{by}
%\begin{equation*}
%F_i^{(s)}(z)
%=
%U\!\left(\frac{z_i - z}{h_n}\right)
%_{s}
%K\!\left(\frac{z_i - z}{h_n}\right)\xi_i
%\end{equation*}
%and that, by \eqref{ass: noise}, the $\xi_i$'s and the $z_i$'s are independent so we have :
\textcolor{black}{we use that}
$$\mathbb{E}\left[\left|h_n^{-r}\textcolor{black}{F_i^{(s)}(z)}\right|^q\right]  
= \mathbb{E}\left[\left|h_n^{-r}U\!\left(\frac{z_i - z}{h_n}\right)
_{s}
K\!\left(\frac{z_i - z}{h_n}\right)\right|^q\right]\mathbb{E}\left(\vert\xi_i \vert^q \right),$$
\textcolor{black}{since, by \eqref{ass: noise}, the $\xi_i$'s and the $z_i$'s are independent.}
Then,
% still using that $\tilde p$ is bounded by $1$ (Lemma \ref{lemma: formule densité}) and that the kernel $K$ has a compact support by \eqref{ass: assumption K}, we uniformly bound $U(u)_{s}K(u)$ and get
\textcolor{black}{using the same arguments as previously {\color{black} for a positive constant $A$ and $\sigma$ defined in  \eqref{ass: noise}}
\begin{align*}
\mathbb{E}\left[\left|h_n^{-r}\textcolor{black}{F_i^{(s)}(z)}\right|^q\right]  
%&= \mathbb{E}\left[\left|h_n^{-r}U\!\left(\frac{z_i - z}{h_n}\right)
%_{s}
%K\!\left(\frac{z_i - z}{h_n}\right)\right|^q\right]\mathbb{E}\left(\vert\xi_i \vert^q \right)\\ & \leq h_n^{-r(q-1)}\int \left \vert U(u)_{s}K(u)\right \vert^q \left \vert \xi_i \right \vert ^q \tilde p(z+h_nu)\mathrm{d}u\\
%& \leq A^q h_n^{-r(q-1)}\int \left \vert \xi_i \right \vert ^q \tilde p(z+h_nu)\mathrm{d}u\\
%& = A^q h_n^{-r(q-1)}\int \left \vert \xi_i \right \vert ^q \tilde p(u)\mathrm{d}u\\
%& {\color{green}= A h_n^{-r(q-1)}\int \left \vert \xi_i \right \vert ^q \int_{ker(W)}p(y+W^\top u)\mathrm{d}y \mathrm{d}u}\\
  & \leq A^q h_n^{-r(q-1)}\mathbb{E}(\vert\xi_i\vert^q)\leq A^q \sigma^q h_n^{-r(q-1)} q!\mathbb{E}\left(\frac{\vert\xi_i\vert^q}{\sigma^q q!}
    \right)\\
 &\leq A^q \sigma^q h_n^{-r(q-1)} q! \mathbb{E}\left(\exp(\vert \xi_i\vert/\sigma)\right)\leq A^q \sigma^q h_n^{-r(q-1)} q! C,
\end{align*}
{\color{black} where the last inequality comes from the sub-Gaussian assumption \eqref{ass: noise}.}}
\textcolor{black}{Denoting $\nu' = 2C(A\sigma)^2h_n^{-r}$ and $H' = A\sigma h_n^{-r}$ we get
\[
\mathbb{E}\left[\left|h_n^{-r}\textcolor{black}{F_i^{(s)}(z)}\right|^q\right] \leq \frac{q!}{2}\nu' H'^{q-2}
\]
and the result follows from the Bernstein inequality given in Lemma \ref{lemma: Bernstein}.}
\end{proof}

\begin{lemma}\label{lemma 26}
{\color{black} Let $\Theta_r$ be defined in \eqref{def: real theta_r} and}  \textcolor{black}{let $n\geq 3$.} Assume that \eqref{ass: assumption K}, {\color{black}\eqref{ass: noise} and \eqref{ass: density}} hold {\color{black} and} {\color{black} {suppose} that $nh_n^r\geq 1$. Let also $\varepsilon_n$ be such that $\varepsilon_n\in ]0, \vert \Theta_r \vert]$ and}  let $\mathcal N$ be the \textcolor{black}{minimal $\varepsilon_n$-net} of $\Theta_r$ with respect to the Euclidean norm and assume that there exists a positive constant $c$ such that
\textcolor{black}{  
\[ n^{-1}h_n^{-r}\log\left(2\frac{\vert \Theta_r \vert }{\varepsilon_n}\right)\leq c. 
\]
}
Then {\color{black} there exists a positive constant $A$ such that} %for all $z\in \Theta_r$} %{\color{salmon} \sout{defined in (\ref{def: real theta_r})}}}
\begin{equation}\label{eq-lemma26}
\int_0^{+\infty} t^3 \mathbb{P}\left(\sup_{z \in \mathcal{N}}\left|
n^{-1} h_n^{-r}
\sum_{i=1}^{n} F_i^{(s,s')}(z)
\right| \geq t\right)\mbox{d}t
%\left(2\frac{\vert \Theta_r \vert}{\varepsilon_n}\right)^r \int_a^{+\infty} t^3 \mathbb{P}\left(\left|\frac{1}{n} h_n^{-r}\sum_{i=1}^n F_i^{(s,s')}{\color{olive}(z)}\right| > t \right)\mathrm{d}t 
\leq An^{-2}h_n^{-2r}\log\left(2\frac{\vert \Theta_r \vert}{\varepsilon_n}\right)^2,
\end{equation}
\textcolor{black}{where $F_i^{(s,s')}$ is defined in (\ref{eq: def fiss}).} 
\end{lemma}

\begin{proof} 

\textcolor{black}{Let $a_n$ be  a positive number. It holds}
\begin{multline*}
\int_0^{+\infty} t^3 \mathbb{P}\left(\sup_{z \in \mathcal{N}}\left|
n^{-1} h_n^{-r}
\sum_{i=1}^{n} F_i^{(s,s')}(z)
\right| \geq t\right)\mbox{d}t \\\leq \textcolor{black}{\frac{a_n^4}{4} + \int_{a_n}^{+\infty} t^3 \mathbb{P}\left(\sup_{z \in \mathcal{N}}\left|
n^{-1} h_n^{-r}
\sum_{i=1}^{n} F_i^{(s,s')}(z)
\right| \geq t\right)\mbox{d}t}.
\end{multline*}
Observe that
\begin{align*}
\mathbb{P}\left(\sup_{z \in \mathcal{N}}\left|
n^{-1} h_n^{-r}
\sum_{i=1}^{n} F_i^{(s,s')}(z)
\right| \geq t\right)
& = \mathbb{P}\left(\bigcup_{z \in \mathcal{N}} \left(\left|
n^{-1} h_n^{-r}
\sum_{i=1}^{n} F_i^{(s,s')}(z)
\right| \geq t\right)\right)\\
& \leq \sum_{z\in \mathcal{N}} \mathbb{P} \left(\left|
n^{-1} h_n^{-r}
\sum_{i=1}^{n} F_i^{(s,s')}(z)
\right| \geq t\right)\\
& \leq \#(\mathcal{N}) \sup_{z \in \mathcal{N}} \mathbb{P}\left(\left|
n^{-1} h_n^{-r}
\sum_{i=1}^{n} F_i^{(s,s')}(z)
\right| \geq t\right)\\
& \leq \left(\frac{\vert \Theta_r \vert}{\varepsilon_n}+1\right)^r \sup_{z \in \mathcal{N}} \mathbb{P}\left(\left|
n^{-1} h_n^{-r}
\sum_{i=1}^{n} F_i^{(s,s')}(z)
\right| \geq t\right)
\end{align*}
where the last inequality comes from \textcolor{black}{the standard bound: $\#(\mathcal{N})\leq\left(\frac{\vert \Theta_r \vert}{\varepsilon_n}+1\right)^r$ given in Corollary 4.2.11 of \cite{Vershynin_2018}} %{\color{salmon} [ajouter une ref?]}
% Lemma \ref{eq: cardinal}. {\color{orange} [NB \`a modifier si on vire le lemme sur $\#(\mathcal N)$]}
 Since, by assumption, $\varepsilon_n \leq \vert \Theta_r \vert$, \textcolor{black}{we have}
\begin{multline*}
%\mathbb{E}\!\left[\sup_{z \in \mathcal{N}}
%\left|
%n^{-1} h_n^{-r}
%\sum_{i=1}^{n} F_i^{(s,s')}(z)
%\right|^4\right]
\int_{\textcolor{black}{a_n}}^{+\infty}t^3 \mathbb{P}\left(\sup_{z \in \mathcal{N}}\left|n^{-1} h_n^{-r}\sum_{i=1}^{n} F_i^{(s,s')}(z)\right| \geq t\right)\mbox{d}t\\
%\mathbb{E}\!\left[\sup_{z \in \mathcal{N}}
%\left|
%n^{-1} h_n^{-r}
%\sum_{i=1}^{n} F_i^{(s,s')}(z)
%\right|^4\right] &
% \leq {\color{magenta}4}\left(2\frac{\vert \Theta_r \vert}{\varepsilon_n}\right)^r \int_0^{+\infty} t^3 \sup_{z \in \mathcal{N}}\mathbb{P}\left(\left|
%n^{-1} h_n^{-r}
%\sum_{i=1}^{n} F_i^{(s,s')}(z)
%\right| \geq t\right)\mbox{d}t 
\leq \left(2\frac{\vert \Theta_r \vert}{\varepsilon_n}\right)^r \int_{\textcolor{black}{a_n}}^{+\infty} t^3 \sup_{z \in \mathcal{N}}\mathbb{P}\left(\left|
n^{-1} h_n^{-r}
\sum_{i=1}^{n} F_i^{(s,s')}(z)
\right| \geq t\right)\mbox{d}t.
\end{multline*}
\textcolor{black}{By Lemma \ref{lemma 27},}
\begin{multline}\label{integ_tsy}
\left(2\frac{\vert \Theta_r \vert}{\varepsilon_n}\right)^r \int_{\textcolor{black}{a_n}}^{+\infty} t^3\textcolor{black}{\sup_{z \in \mathcal{N}}} \mathbb{P}\left(\left|\frac{1}{n} h_n^{-r}\sum_{i=1}^n F_i^{(s,s')}(z)\right| > t \right)\mathrm{d}t \\\leq 2 \int_{\textcolor{black}{a_n}}^{+\infty} t^3\exp \left( -\frac{t^2nh_n^r}{A_3(1+t)}+A_3\log\left(2\frac{\vert \Theta_r \vert}{\varepsilon_n}\right)\right)dt,
\end{multline}
\textcolor{black}{where $A_3 = \text{max}(C,r)$, $C$ being the constant appearing in \eqref{eq-lemma27}}.
%{\color{olive} We consider the interval $[a,b]$
To compute the integral, we define 
  \begin{align*}
a_n = A_3 \gamma u_n  \quad \textrm{with} \quad u_n = \left( n^{-1}h_n^{-r}\log\left(2\frac{\vert \Theta_r \vert }{\varepsilon_n}\right) \right)^{1/2} \quad \text{and} \quad 
b = \frac{\gamma^2}{A_3+1}-1,
\end{align*}
where $\gamma = 2 (A_3+1)\max\{A_3\sqrt c,1\}.$
\textcolor{black}{Since we will rewrite the integral as $\int_{a_n}^{+\infty}=\int_{a_n}^{b}+\int_{b}^{+\infty}$}, \textcolor{black}{we prove in the next lines  that $a_n \leq b$ for all $n \geq 3$. Note that proving $a_n \leq b$ is equivalent to proving $\gamma^2 - u_n A_3 (A_3+1) \gamma - (A_3+1) \textcolor{black}{\geq}0.$}
% \begin{align*}
%   a_n \leq b \iff \gamma^2 - \textcolor{black}{u_n} A_3 (A_3+1) \gamma - (A_3+1) \textcolor{black}{\geq}0.
% \end{align*}
\textcolor{black}{Since $u_n\leq\sqrt{c}$, it is thus enough to prove that
\begin{equation}\label{eq:gamma}
 \gamma^2 - \sqrt{c} A_3 (A_3+1) \gamma - (A_3+1) \geq 0,
\end{equation}
which holds true if
$$
\gamma \geq \frac{ \sqrt{c} A_3 (A_3+1) + \sqrt{c  A_3^2 (A_3+1)^2 + 4 (A_3+1) } }{2}.
$$
Hence, \eqref{eq:gamma} is satisfied if
\begin{multline*}
  \gamma \geq \frac{\max\{A_3\sqrt c,1\}(A_3+1) +\sqrt{\max\{A_3\sqrt c,1\}^2(A_3+1)^2 + 4\max\{A_3\sqrt c;1\}^2(A_3+1)^2}}{2}\\=\max\{A_3\sqrt c,1\}(A_3+1)
  \left(\frac{1 +\sqrt{5}}{2}\right),
\end{multline*}
which is true by the definition of $\gamma$.
}
% which, using the assumption $u_n \leq \sqrt c$ and $A_3+1 \geq 1$, is less restrictive than
% $$ \gamma > \frac{\max\{A_3\sqrt c,1\}(A_3+1) +\sqrt{\max\{A_3\sqrt c,1\}^2(A_3+1)^2 + 4\max\{A_3\sqrt c;1\}(A_3+1)^2}}{2}.$$
% The previous condition is exactly 
% $$ \gamma > \max\{A_3\sqrt c,1\}(A_3+1)\frac{1 +\sqrt{5}}{2}$$
% and is satisfied by the chosen $\gamma$ which proves $a_n \leq b$.
For $t \in [a_n,b]$, \textcolor{black}{observe} that
\begin{multline*}
{\color{black} -\frac{t^2nh_n^r}{A_3(1+t)}+A_3\log\left(2\frac{\vert \Theta_r \vert}{\varepsilon_n}\right) }
%& = -\frac{t^2nh_n^r}{A_3} \left( \frac{1}{t+1} - \frac{A_3^2 \log\left(2\frac{\vert \Theta_r \vert}{\varepsilon_n}\right)}{t^2 n h^r} \right)\\
= - \frac{t^2nh_n^r}{A_3} \left( \frac{1}{t+1} - \frac{a_n^2}{\gamma^2 t^2}\right)\\
 \leq  - \frac{t^2nh_n^r}{A_3} \left( \frac{1}{t+1} - \frac{1}{\gamma^2} \right) 
  \leq  - \frac{t^2nh_n^r}{A_3} \left( \frac{1}{b+1} - \frac{1}{\gamma^2} \right) {\color{black} = -t^2nh_n^r}.
%& =  - \frac{t^2nh_n^r}{A_3} \tilde A
\end{multline*}
%where $\tilde A = \frac{1}{b+1} - \frac{1}{\gamma^2} >0$.
Moreover, for $t>b$, 
\begin{multline*}
{\color{black} -\frac{t^2nh_n^r}{A_3(1+t)}+A_3\log\left(2\frac{\vert \Theta_r \vert}{\varepsilon_n}\right) }
%& = -\frac{tnh_n^r}{A_3} \left( \frac{t}{t+1} - \frac{A_3^2 \log\left(2\frac{\vert \Theta_r \vert}{\varepsilon_n}\right)}{t n h^r} \right)\\
 = - \frac{tnh_n^r}{A_3} \left( \frac{t}{t+1} - \frac{a_n^2}{\gamma^2 t}\right)\\
 \leq  - \frac{tnh_n^r}{A_3} \left( \frac{t}{t+1} - \frac{a_n}{\gamma^2} \right) 
  \leq  - \frac{tnh_n^r}{A_3} \left( \frac{b}{b+1} - \frac{a_n}{\gamma^2} \right)
   \leq  - \frac{tnh_n^r}{A_3} \left( \frac{b}{b+1} - \frac{a}{\gamma^2} \right),
\end{multline*}
where $a = A_3\gamma \sqrt c$ is a positive constant that satisfies $a_n \leq a$ for all $n\geq 3$.

Let $\tilde{A}=\frac{b}{b+1} - \frac{a}{\gamma^2}.$ \textcolor{black}{Let us first prove that $\tilde{A}$ is positive. It is enough to prove that
  $\gamma^2 -A_3\sqrt{c}\gamma -(A_3+1)>0$ which holds true if
  $$
  \gamma>\frac{A_3\sqrt{c}+\sqrt{A_3^2c+4(A_3+1)}}{2},
  $$
  which is satisfied if
  \begin{multline*}
    \gamma>\frac{\max(A_3\sqrt{c},1)(A_3+1)+\sqrt{\max(A_3\sqrt{c},1)^2(A_3+1)^2+4\max(A_3\sqrt{c},1)^2(A_3+1)^2}}{2}\\
    =\max\{A_3\sqrt c,1\}(A_3+1)\left(\frac{1 +\sqrt{5}}{2}\right),
\end{multline*}
which is true by the definition of $\gamma$.
}
%With the definition of $a$ and $b$, one can verify that $\tilde{A}>0$ by following the same steps as the proof of $a_n \leq b$.
%\begin{align*}
%\tilde A >0 & \iff \gamma^2 -(A_3+1) -a >0 \\
%& \iff \gamma > \frac{A_3\sqrt c +\sqrt{A_3^2c + 4(A_3+1)}}{2},
%\end{align*}
%which is less restrictive than
%$$ \gamma > \frac{\max\{A_3\sqrt c,1\} +\sqrt{\max\{A_3\sqrt c,1\}^2 + 4\max\{A_3\sqrt c;1\}(A_3+1)}}{2}.$$
%also less restrictive than 
%$$ \gamma > \frac{\max\{A_3\sqrt c,1\}(A_3+1) +\sqrt{\max\{A_3\sqrt c,1\}^2(A_3+1)^2+ 4\max\{A_3\sqrt c;1\}(A_3+1)^2}}{2}.$$
%The previous condition is exactly 
%$$ \gamma > \max\{A_3\sqrt c,1\}(A_3+1)\frac{1 +\sqrt{5}}{2}$$
%and is satisfied by the chosen $\gamma$ which proves $\tilde A >0$.
%
Hence, \eqref{integ_tsy} becomes 
\begin{align*}
 {\color{black} \left(2\frac{\vert \Theta_r \vert}{\varepsilon_n}\right)^r }& \int_{\textcolor{black}{a_n}}^{+\infty} t^3 \textcolor{black}{\sup_{z \in \mathcal{N}}} \mathbb{P}\left(\left|\frac{1}{n} h_n^{-r}\sum_{i=1}^n F_i^{(s,s')}(z)\right| > t \right)\mathrm{d}t \\
& \leq  %{\color{magenta} \frac{a^4}{4} }+ 
 {\color{black}2} \int_{\textcolor{black}{a_n}}^{b} t^3\exp \left( - t^2nh_n^r \right)dt
 + {\color{black}2} \int_b^{+\infty} t^3\exp \left( - \frac{tnh_n^r}{A_3} \tilde A\right)dt\\
 & \leq  %  {\color{magenta} \frac{a^4}{4} } + 
{\color{black}2}  \int_0^{b} t^3\exp \left( - t^2nh_n^r \right)dt
 +  {\color{black}2}\int_b^{+\infty} t^3\exp \left( - \frac{tnh_n^r}{A_3} \tilde A\right)dt.
%
%\int_0^{+\infty} t^3\exp \left( -\frac{t^2nh_n^r}{A_3(1+t)}+A_3\log\left(2\frac{\vert \Theta_r \vert}{\varepsilon_n}\right)\right)dt
%
\end{align*}
%{\color{magenta} [on est sur que c'est exactement $a^4/4$ sans constante ?]}}
%}
{\color{black} Note that, by defining $v_n = nh_n^r$, the first integral becomes
%\begin{equation}\label{eq:calcul_int_t2}
\[\int_0^{b}
t^3 \,
\mbox{exp}\left(-v_n t^2\right)
\mathrm{d}t = \frac{1}{2v_n^2} - \frac{1}{2v_n}\mbox{exp}(-v_n b^2)\left(b^2+\frac{1}{v_n}\right)\leq \frac{1}{2v_n^2}.
%\end{equation}
\]
To compute the second integral, denote $w_n =  \frac{\tilde A nh_n^r}{A_3} $ and observe that
\[
%\begin{equation}\label{eq:calcul_int_t}
 \int_{b}^{+\infty}
t^3 \,
\exp\left(-w_n t \right)
\mathrm{d}t = \frac{1}{w_n} \mbox{exp}(-w_n b)\left(b^3+\frac{3 b^2}{w_n}+\frac{6 b}{w_n^2}+\frac{6}{w_n^3}\right) %\leq \frac{K(b)}{v^2},
\leq  \frac{1}{w_n^2 b} \left(b^3+\frac{3 b^2}{w_n}+\frac{6 b}{w_n^2}+\frac{6}{w_n^3}\right) 
%\end{equation} 
\]
where the last inequality comes from the fact that  $x\exp(-x) \leq 1$ for all $x$ in $\mathbb{R}$.
Moreover, according to the definition of $w_n$ and the assumption that $nh_n^r \geq 1$, we have that $w_n \geq \frac{\tilde A}{A_3}$. Hence 
\[
 \int_{b}^{+\infty}
t^3 \,
\exp\left(-w_n t \right)
\mathrm{d}t \leq n^{-2}h_n^{-2}C(b),
\]
where $C(b) = A_3^2\tilde A^{-2} \left(b^2+{3bA_3}{\tilde A^{-1}} +{6A_3^2}{\tilde A^{-2}} +{6b^{-1}A_3^3}{\tilde A_3^{-3}}\right)>0$. Then 
\[
 \left(2\frac{\vert \Theta_r \vert}{\varepsilon_n}\right)^r \int_{a_n}^{+\infty} t^3  \mathbb{P}\left(\left|\frac{1}{n} h_n^{-r}\sum_{i=1}^n F_i^{(s,s')}(z)\right| > t \right)\mathrm{d}t 
 \leq % {\color{magenta} \frac{a^4}{4} }+ 
n^{-2}h_n^{-2r} (1+2C(b)).
\]
{\color{black} By the definition of $a_n$} and using the fact that $\log\left(2\frac{\vert \Theta_r \vert}{\varepsilon_n}\right) \geq \log(2)$ since $\varepsilon_n \leq \vert \Theta_r \vert$, we conclude the proof.
% \[
% \left(2\frac{\vert \Theta_r \vert}{\varepsilon_n}\right)^r \int_a^{+\infty} t^3  \mathbb{P}\left(\left|\frac{1}{n} h_n^{-r}\sum_{i=1}^n F_i^{(s,s')}(z)\right| > t \right)\mathrm{d}t 
%  \leq A n^{-2}h_n^{-2r} \log\left(2\frac{\vert \Theta_r \vert}{\varepsilon_n}\right)
% \]
% for a suitable constant $A>0$.
}

%where $K(\xi) >0$ is a constant depending only on $\xi$.
%The last inequality is due to the fact that $\mbox{exp}(-\gamma'\xi)\leq\frac{1}{\gamma'\xi}$ and that the condition $nh_n^r \geq 1$ implies $\frac{1}{\gamma'} \leq A'$.
%Using \eqref{eq:calcul_int_t2} and \eqref{eq:calcul_int_t} with $\xi = B_A$, we conclude that for $A_1$ large enough and $a_{A_1,n} = \sqrt{A_1}A_3u_n$ we have the bound
%\begin{equation*}
% \int_{a_{A_1,n}}^{+\infty}t^3 \exp\left(-\frac{t^2 n h_n^r}{A_3(1+t)} + A_3 \log\left(\frac{2\vert \Theta_r \vert}{\varepsilon_n}\right)\right) \mathrm{d}t \leq Cn^{-2}h_n^{-2r},
%\end{equation*}
%where $C>0$ is a constant.

\end{proof}

\begin{lemma}\label{lemma 21}
Let $\mathcal W$ be the Lambert function defined in \eqref{def: lambert functions} and let $\Theta_r$ be the \textcolor{black}{set} defined in \eqref{def: real theta_r}.
Under assumptions \eqref{ass: assumption K}, \eqref{ass: noise} and \eqref{ass: density}, if $nh_n^r \geq \max \left(\mathcal{W}(n),\lambda_n^{-2}\mathcal{W}(n)\right)$, we have
\begin{equation*}
 \mathbb{E}
\left(\sup_{z\in\Theta_r}
\bigl\| G_n(z) \bigr\|_{\mathrm{op}}^{4}
\right)
\leq An^{-2}h_n^{-2r}\mathcal{W}(n)^2,
\end{equation*}
where $G_n$ \textcolor{black}{is} defined \eqref{def: Gn} and $A$ is a positive constant. 
\end{lemma}

\begin{proof}
For $i \in \{1,...,n\}$ and $s \in \{1,...,S\}$, \textcolor{black}{let us} define 
\begin{equation}\label{eq:def fis}
F^{(s)}_i : z \longmapsto U\left(\frac{z_i-z}{h_n}\right)_{s}K\left(\frac{z_i-z}{h_n}\right)\xi_i.
\end{equation}
where  \textcolor{black}{$U_s$ denotes the $s$th coordinate of $U$,} the $\xi_i$'s are introduced in \eqref{eq: regression model} and are centered and indenpendent of the $z_i$'s \textcolor{black}{by \eqref{ass: noise}}. Thus, $\mathbb{E}(G_n(z)) = 0$ for all $z\in \Theta_r$.
%so the inequality that need to be proved is similar to \eqref{eq:lemma20.1} of Lemma \ref{lemma 20}.
%In the same way as in the beginning of its proof, we can get the bound
\textcolor{black}{Following the lines of the proof of Lemma \ref{lemma 20}, we get}
\begin{align*}
\mathbb{E}\!\left[
\sup_{z \in \Theta_r} \| G_{n,\lambda_n}(z)\|_{\mathrm{op}}^{4}
\right]
&\leq S^3 \sum_{s=1}^{S}
\mathbb{E}\!\left[
\sup_{z \in \Theta_r}
\left|
n^{-1} h_n^{-r}
\sum_{i=1}^{n} F_i^{(s)}(z)
\right|^{4}
\right].
\end{align*}
% which is similar to the bound obtained for  $\mathbb{E}\!\left[
% \sup_{z \in \Theta_r} \| B_{n,\lambda_n}(z)-\mathbb{E}(B_{n,\lambda_n}(z)) \|_{\mathrm{op}}^{4}
% \right]$. Then, the rest of the proof is analogous but replacing the functions $F_i^{(s,s')}$'s defined in \eqref{eq: def fiss} by the $F_i^{(s)}$'s. To prove the needed Lipschitz property of the $F_i^{(s)}$'s, we use that each $F_i^{(s)}$ has a Lipschitz constant proportionnal to $\vert \xi_i\vert$ so the inegality \eqref{eq: Fiss lipschitz} becomes
\textcolor{black}{Moreover,
$$\left \vert \sum_{i=1}^n \left ( F_i^{(s)}(x)-F_i^{(s)}(y) \right) \right \vert \leq  A h_n^{-1}\Vert x-y\Vert\sum_{i=1}^n \vert \xi_i\vert,
$$
where, for $i\in \{1,...,n\}$,
\begin{align*} \mathbb{E}\left (\vert \xi_i \vert \right) &=  \int_0^{\infty}\mathbb{P}\left(\vert\xi_i\vert > t\right)dt =  \int_0^{\infty}\mathbb{P}\left(e^{\vert\xi_i\vert/\sigma} > e^{t/\sigma}\right)dt  =  \sigma\int_1^{\infty}\frac{1}{u}\mathbb{P}\left(e^{\vert\xi_i\vert/\sigma} > u\right)du\\
& \leq  \sigma\int_1^{\infty}\mathbb{P}\left(e^{\vert\xi_i\vert/\sigma} > u\right)du
\leq  \sigma\mathbb{E}\left(e^{\vert\xi_i\vert/\sigma}\right)\leq A,
\end{align*}
the last inequality coming from \eqref{ass: noise} {\color{black} and} $A$ being a positive constant that does not depend on $i$.
Thus, 
$\mathbb{E}(\sum_{i=1}^n \vert \xi_i \vert) \leq nA,$
which leads to a an inequality similar to \eqref{star}. The rest of the proof \textcolor{black}{comes from Lemma \ref{lemma 27} and an adaptation of Lemma \ref{lemma 26} to the $F_i^{(s)}$'s.}}

%\textcolor{black}{C'est complètement évident que toute la preuve va bien marcher ? Il y a toute l'histoire sur le choix de $\varepsilon_n$......}

%$\forall t > 0,$\[ \mathbb{P}\left(\left|\frac{1}{n}h_n^{-r}\sum_{i=1}^n H_i^{(s)} \right| > t \right) \leq 2 \exp\left(-\frac{{\color{magenta}t^2n}h_n^{-r}}{A(1+t)}\right)\] for a suitable positive constant $A$.
\end{proof}

\begin{lemma}\label{lemma 23}
Let $\mathcal W$ be the Lambert function defined in \eqref{def: lambert functions} and let $\Theta_r$ be the \textcolor{black}{set} defined in \eqref{def: real theta_r}.
{\color{black} Under \eqref{ass: assumption K}, \eqref{ass: min g}, \eqref{ass: regularity f} and \eqref{ass: density},} if $nh_n^r \geq \max \left(\mathcal{W}(n),\lambda_n^{-2}\mathcal{W}(n)\right)$,
\begin{equation*}
 \mathbb{E}
\left(\sup_{z\in \Theta_r} 
\bigl\| C_{n}(z) - \mathbb{E}(C_{n}(z)) \bigr\|_{\mathrm{op}}^{4}
\right)
\leq Ah_n^{-2r}n^{-2}\mathcal{W}(n)^2, 
\end{equation*}
where $C_n$ is defined in \eqref{def: Cn} and $A$ is a positive constant.
\end{lemma}

\begin{proof}
For $i \in \{1,...,n\}$ and $s \in \{1,...,S\}$, \textcolor{black}{let us define} 
\begin{equation}\label{def his}
H_i^{(s)}(z)
=
U\!\left(\frac{z_i - z}{h_n}\right)
_{s}
K\!\left(\frac{z_i - z}{h_n}\right)g(z_i)
-
\mathbb{E}
\left[
U\!\left(\frac{z_i - z}{h_n}\right)
_{s}
K\!\left(\frac{z_i - z}{h_n}\right)g(z_i)
\right],
\end{equation}
where  $g$ is introduced in \eqref{def: g} \textcolor{black}{and $U_s$ denotes the $s$th coordinate of the vector $U$.}
\textcolor{black}{Similarly to the beginning of the proof of Lemma \ref{lemma 20}, we get}
\begin{align*}
\mathbb{E}\!\left[
\sup_{z \in \Theta_r} \| C_{n,\lambda_n}(z)-\mathbb{E}(C_{n,\lambda_n}(z)) \|_{\mathrm{op}}^{4}
\right]
&\leq S^3 \sum_{s=1}^{S}
\mathbb{E}\!\left[
\sup_{z \in \Theta_r}
\left|
n^{-1} h_n^{-r}
\sum_{i=1}^{n} H_i^{(s)}(z)
\right|^{4}
\right].
\end{align*}
%which is similar to the bound obtained for  $\mathbb{E}\!\left[
%\sup_{z \in \Theta_r} \| B_{n,\lambda_n}(z)-\mathbb{E}(B_{n,\lambda_n}(z)) \|_{\mathrm{op}}^{4}
% \right]$.
\textcolor{black}{To prove the Lipschitz property of the $H_i^{(s)}$, we use that the function $g$ is bounded on the compact $\Theta_r$ and the rest of the proof is similar.}
%Thus, this proof is analogous to the proof of \eqref{eq:lemma20.1} of Lemma \ref{lemma 20} but replacing the functions $F_i^{(s,s')}$'s defined in \eqref{eq: def fiss} by the $H_i^{(s)}$'s and using the fact that the function $g$ is bounded on the compact $\Theta_r$ to prove the wanted Lipschitz property of the $H_i^{(s)}$'s. Both Lemma \ref{lemma 27} and Lemma \ref{lemma 26}, needed in this proof, can be easily adapted for the $H_i^{(s)}$'s instead of the $F_i^{(s,s')}$'s.
%\textcolor{black}{{\color{salmon} \sout{The rest of the proof comes from}}
\textcolor{black}{Lemma \ref{lemma 27} and an adaptation of Lemma \ref{lemma 26} to the $H_i^{(s)}$'s} {\color{black} conclude the proof.}
%\textcolor{black}{C'est complètement évident que toute la preuve va bien marcher ? Il y a toute l'histoire sur le choix de $\varepsilon_n$......}

%$\forall t > 0,$\[ \mathbb{P}\left(\left|\frac{1}{n}h_n^{-r}\sum_{i=1}^n H_i^{(s)} \right| > t \right) \leq 2 \exp\left(-\frac{{\color{magenta}t^2n}h_n^{-r}}{A(1+t)}\right)\] for a suitable positive constant $A$.
\end{proof}

\textcolor{black}{For the reader convenience}, the Bernstein's inequality used in the proofs is recalled in the following Lemma.
\begin{lemma}\label{lemma: Bernstein}
Let $\{\zeta_i\}_{i=1}^k$ be a collection of real-valued independent random variables with zero means such that, for any $i=1,\ldots,k$ and any $q\in\{2,3,\ldots\}$, the Bernstein condition
\[
\mathbb{E}\left[|\zeta_i|^q\right]
\le \frac{q!}{2}\,\nu H^{q-2}
\]
is satisfied with some $\nu>0$, $H>0$. Then, for any $t>0$,
\[
\mathbb{P}\left(
\left|
\frac{1}{k}\sum_{i=1}^k \zeta_i
\right|
> t
\right)
\le
2\exp\left(
-\frac{t^2k}{2(\nu+Ht)}
\right).
\]
\end{lemma}

\textcolor{black}{
\begin{lemma}\label{lem:lipschitz}
Under the assumptions of Lemma \ref{lemma 20} the functions $\phi^{(s,s')}_i$ defined in \eqref{eq:def_phi} are Lipschitz with constant $\gamma_K h_n^{-1}$ where $\gamma_K$ is a constant depending on $L_K$ and $\Theta_r$.
\end{lemma}
\begin{proof}
 To alleviate the notations, we rewrite %{\color{salmon} \sout{$\phi^{(s,s')}_i$ as follows:  }}
\[\phi^{(s,s')}_i(z)=\mathcal{U}(z)\mathcal{K}(z)\] with
  $\mathcal{U}(z)=\left(U\left(\frac{z_i-z}{h_n}\right)U\left(\frac{z_i-z}{h_n}\right)^{\!\top}\right)_{s,s'}$
  and $\mathcal{K}(z)=K\left(\frac{z_i-z}{h_n}\right)$ {\color{black} and where,} by \eqref{ass: assumption K}, $\mathcal{K}$ has a compact support depending on $z_i$. Thus for $z$ and $z'$ in the compact support of $\mathcal{K}$, we have
  \begin{align*}
{\color{black} \phi^{(s,s')}_i(z) - \phi^{(s,s')}_i(z')=\mathcal{U}(z)\mathcal{K}(z)-\mathcal{U}(z')\mathcal{K}(z')=\left(\mathcal{U}(z)-\mathcal{U}(z')\right)\mathcal{K}(z)
    +\mathcal{U}(z')\left(\mathcal{K}(z)-\mathcal{K}(z')\right),}
  \end{align*}
  which gives the result using that $K$ is $L_K$-Lipschitz by \eqref{ass: assumption K} and the fact that $\mathcal{U}$ is a polynomial function and thus bounded and Lipschitz on a compact support. If neither $z$ nor $z'$ are in the support of $\mathcal{K}$, the difference is equal to zero. If $z$ is in the support of $\mathcal{K}$ but not $z'$, then
  $$
  \mathcal{U}(z)\mathcal{K}(z)-\mathcal{U}(z')\mathcal{K}(z')=\mathcal{U}(z)\mathcal{K}(z)=\mathcal{U}(z)(\mathcal{K}(z)-\mathcal{K}(z')),
  $$
  which gives the result using that $K$ is $L_K$-Lipschitz.
\end{proof}
}

\end{document}